\documentclass{article}
\usepackage[utf8]{inputenc}
\usepackage{amsmath}
\usepackage{amssymb}
\usepackage{amsthm}
\usepackage{verbatim}
\usepackage{graphicx}
\usepackage{subcaption}
\usepackage[shortlabels]{enumitem}
\usepackage{xcolor}
\usepackage{bm}
\usepackage{float}
\usepackage{lineno}
\usepackage{wasysym}
\usepackage{amsfonts}
\usepackage[normalem]{ulem}
\usepackage{pgfplots}
\usepackage{color}
\usepackage[backend=biber, style=numeric, doi=true, giveninits=true]{biblatex}
\AtEveryBibitem{
  \clearfield{number}
}

\pgfplotsset{compat=1.18}

\newtheorem{theorem}{Theorem}[section]

\newtheorem{proposition}[theorem]{Proposition}
\newtheorem{lemma}[theorem]{Lemma} 
\newtheorem{remark}[theorem]{Remark}

\newcommand{\AAA}{{\mathsf{A}}}
\newcommand{\WW}{\widetilde{\mathsf{W}}}
\newcommand{\xx}{\mathbf{x}}

\newcommand{\bb}{\mathbf{b}}

\newcommand{\DD}{\mathcal{D}}
\newcommand{\Dw}{\mathcal{D}_w}

\newcommand{\ee}{{\mathbf{e}}}
\newcommand{{\oTV}}{\overline{TV}}

\newcommand{\nullspace}[1]{{\mathcal{N}(#1)}}

\newcommand{\ws}{\widetilde{w}}
\newcommand{\Mu}{\mathcal{U}}

\newcommand{\anti}{\mathcal{L}^n}

\colorlet{revcol}{black}          
\newcommand{\rev}[1]{{\color{revcol}#1}}

\newenvironment{revblock}{\color{revcol}}{}

\providecommand{\keywords}[1]{\textbf{\textit{Keywords:}} #1}

\DeclareMathOperator*{\argmin}{arg\,min}

\title{Weighted total variation regularization for inverse problems with significant null spaces}
\author{Martin Burger\thanks{ Helmholtz Imaging, Deutsches Elektronen-Synchroton DESY, Notkestr. 85, 22607 Hamburg, Germany, and Department of Mathematics, University of Hamburg, Bundesstr. 55, 20146 Hamburg, Germany.}, Ole L{\o}seth Elvetun\thanks{Faculty of Science and Technology, Norwegian University of Life Sciences. Email: ole.elvetun@nmbu.no.} and Bj{\o}rn Fredrik Nielsen\thanks{Faculty of Science and Technology, Norwegian University of Life Sciences. Email: bjorn.f.nielsen@nmbu.no.}}
\date{}

\begin{document}

\maketitle

\begin{abstract}
    We consider inverse problems with large null spaces, which arise in important applications such as in inverse ECG and EEG procedures. Standard regularization methods typically produce solutions in or near the orthogonal complement of the forward operator's null space. This often leads to inadequate results, where internal sources are mistakenly interpreted as being near the data acquisition sites -- e.g., near or at the body surface in connection with EEG and ECG recordings.

    To mitigate this, we previously proposed weighting schemes for Tikhonov and sparsity regularization. Here, we extend this approach to total variation (TV) regularization, which is particularly suited for identifying spatially extended regions with approximately constant values. We introduce a weighted TV-regularization method, provide supporting analysis, and demonstrate its performance through numerical experiments. Unlike standard TV regularization, the weighted version successfully recovers the location and size of large, piecewise constant sources away from the boundary, though not their exact shape.

    Additionally, we explore a hybrid weighted-sparsity and TV regularization approach, which better captures both small and large sources, albeit with somewhat more blurred reconstructions than the weighted TV method alone.
\end{abstract}

\noindent \keywords{Weighted total variation regularization, null space, ECG, EEG, hybrid weighted-sparsity and total variation.}

\section{Introduction}
We will consider the problem 
\begin{equation*}
    \min_{f\in BV(\Omega)} \left\{ \frac{1}{2}\|Kf - d\|^2_{L^2(E)} + \alpha TV_w(f) \right\},
\end{equation*}
where 
\begin{equation} \label{eq:forward_operator_X}
K: \rev{X} \rightarrow L^2(E)
\end{equation}
is a linear and non-injective \rev{continuous} operator, \rev{$X=X(\Omega)$ is a Banach space, $\Omega \subset \mathbb{R}^n$ is bounded, $n=1,2$} and $E$ is the \rev{bounded} observation domain. \rev{Typically, $K$ is compact and will, depending on the context, be regarded as an operator defined on \rev{$X=L^q(\Omega)$, $q=1,2$, or on the space $X=\mathcal{M}(\Omega)$ of bounded Radon measures}, which will be specified below}\footnote{\rev{$BV(\Omega)$ is continuously embedded in $L^2(\Omega)$ and hence $K$ can be applied to $BV(\Omega)$-functions.}}. The weighted regularization functional $TV_w: BV(\Omega) \rightarrow \mathbb{R}$ is defined by
\begin{align*}
    TV_w(f) &= \int_\Omega | \bm{\omega}(x)\nabla f(x)|_1 dx \\
    &= \int_\Omega w_1(x) |\partial_{x_1} f(x)| +w_2(x) |\partial_{x_2} f(x)|dx, 
\end{align*}
provided that $f \in W^{1,1}(\Omega)$. For any $f \in BV(\Omega) \setminus W^{1,1}(\Omega)$, the definition relies on the well-known dual formulation of the TV-functional presented in Section \ref{sec:preliminaries}. Note that we employ the anisotropic version of TV and that we assume that $\bm{w}(x)$ is a diagonal weight matrix with \rev{continuous and non-negative} functions $w_1(x)$ and $w_2(x)$ at the diagonal. Their precise form is discussed below.

This work is motivated by the need, in many applications, to estimate the source term $f$ in an equation in the form  
\begin{equation} \label{intro:main_equation}
    Kf=d. 
\end{equation}
Assuming that $f$ has a "blocky" structure, one would typically want to apply TV-regularization to this problem because this technique has proven to handle deblurring tasks very well, see, e.g., \cite{dobson1996recovery,BrediesHoller2020}. 

For standard \rev{imaging} problems, $K$  equals the identity operator $I$, and one can employ the uniform weights $w_1(x)=w_2(x)=1, \, x \in \Omega,$ without introducing any further\footnote{TV-regularization yields a "blocky" biased, i.e., a preference for reconstructions with small total variation.} bias. On the other hand, when $K$ has a significant null space, the textbook form of TV-regularization might produce results of rather low quality: Figure \ref{fig:motivation} shows computations undertaken with $K=T (-\Delta+I)^{-1}$, where $\Delta$ represents the Laplace operator, $I$ is the identity, $T:H^1(\Omega) \rightarrow L^2(\partial \Omega)$ is the trace operator and $\Omega = (0,1) \times (0,1)$. That is, we seek to use boundary data $d \in L^2(\partial \Omega)$ to recover the source term $f$ in the following equations  
\begin{eqnarray*}
    -\Delta u + u &=& f, \quad x \in \Omega, \\ 
    \nabla u \cdot \mathbf{n} &=&0, \quad x \in \partial\Omega. 
\end{eqnarray*}
We refer to pages 159--161 in \cite{burger2013inverse} for a discussion of why standard regularization schemes fail to handle this type of problem adequately. 

In this paper, we define the weight functions $w_1(x)$ and $w_2(x)$ in terms of the forward mapping and the (partial) derivatives of the Green's function of a suitable differential operator. This enables the recovery in 1D of single-jump-sources, i.e., shifted Heaviside functions, from boundary data, which is proven in Section \ref{sec:1D_analysis}. The ideas are extended to 2D problems in Section \ref{sec:2D3D_analysis}, including some theory. 
Section \ref{sec:numerical_experiments} contains a series of numerical experiments, which reveal that the weighted approach yields significantly better results than the standard TV scheme. In fact, the position and size of the source are estimated rather well, but its shape cannot be recovered, unless some rather strict conditions are satisfied.      

Section \ref{sec:hybrid} contains a discussion and analysis of a hybrid weighted-sparsity and TV regularization approach, which turns out to be better suited at identifying both small and larger sources. Moreover, its analysis is transparent, leading to short proofs of its approximate reconstruction capabilities. However, it yields somewhat more blurred recoveries than the (pure) weighted TV methodology. 

The present text may be considered as a follow-up work to our earlier investigations of similar issues for Tikhonov and sparsity regularization, see e.g., \cite{Elv24} and references therein. Nevertheless, it turns out that not only the implementation of the weighted versions of the former methods is easier than the numerical treatment of weighted TV, but also their analysis is more accessible. This is due to the fact that Tikhonov and sparsity regularization can be employed to \rev{identify} small- and medium-sized local sources, whereas the TV approach typically is studied in terms of its capacity to detect borders between larger subregions. 

There are many papers discussing and analyzing various weighting procedures in connection with total variation regularization, covering theory, numerics, and applications, e.g., \cite{bui2021weighted,hintermuller2017optimal_I,hintermuller2017optimal,li2021novel,rahiman2020multi,sheng2024weighted}. Nevertheless, as far as the authors know, this is the first text focusing on how to improve the performance of TV schemes when applied to inverse problems with significant null spaces. \rev{We would like to stress that the weighting philosophy adopted here differs from the one underlying most of the papers just cited: our weights are derived from the forward operator and from the gradient of a Green's function, and thus only depend on the problem geometry and on $K$, whereas the weights in, e.g., \cite{bui2021weighted,hintermuller2017optimal_I,hintermuller2017optimal} are derived from the image/data itself (edge indicators, bilevel-optimized or learned weights). The two approaches therefore compensate for different sources of bias and could in principle be compared and combined, but this is beyond the scope of the present paper.}

\begin{revblock}
\begin{figure}
    \centering

\resizebox{.8\textwidth}{!}{
    \begin{subfigure}[t]{0.35\linewidth}
        \centering
        \includegraphics[trim=40 30 40 30, clip, width=\linewidth]
        {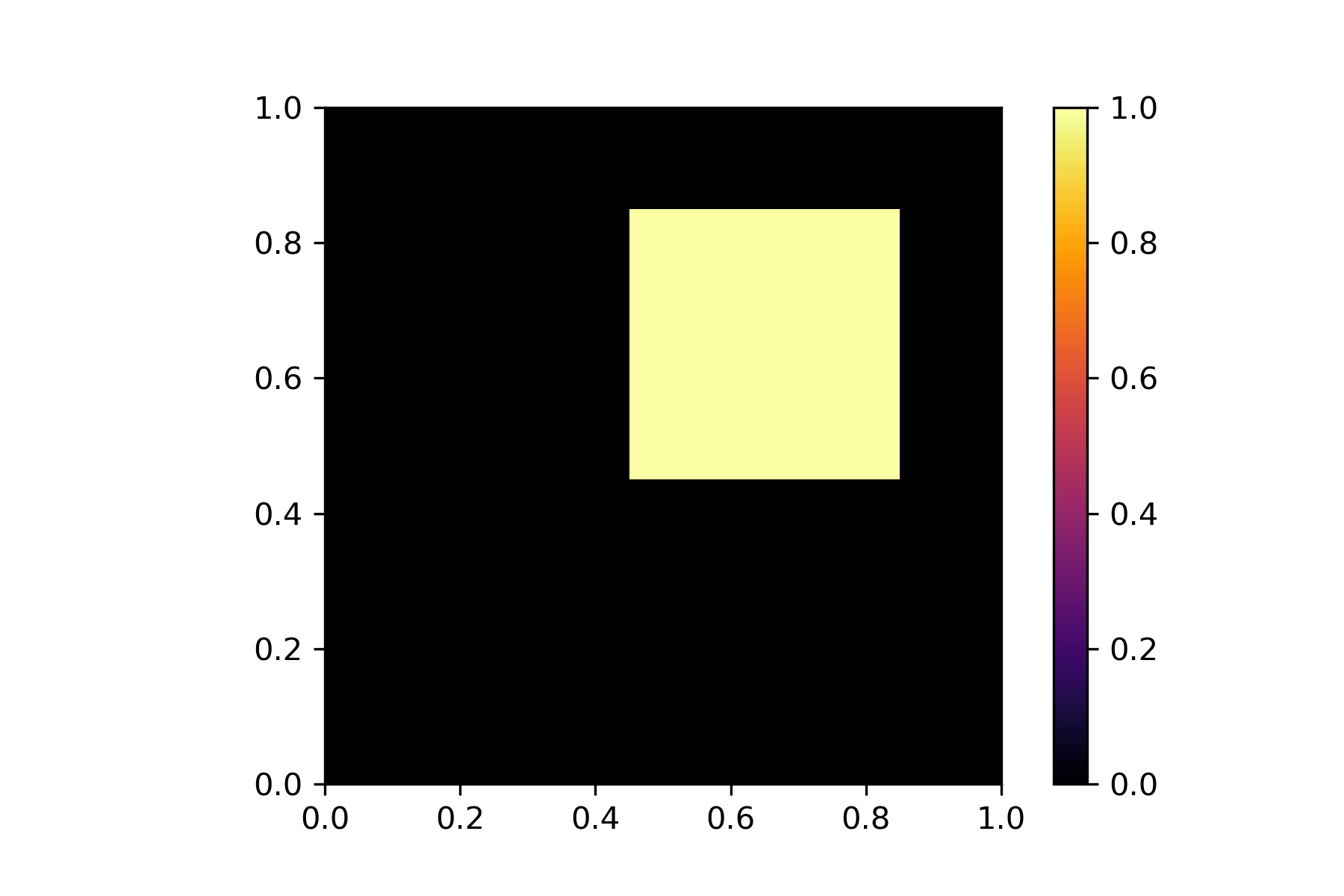}
        \caption{True source}
    \end{subfigure}
    \begin{subfigure}[t]{0.35\linewidth}
        \centering
        \includegraphics[trim=40 30 40 30, clip, width=\linewidth]
        {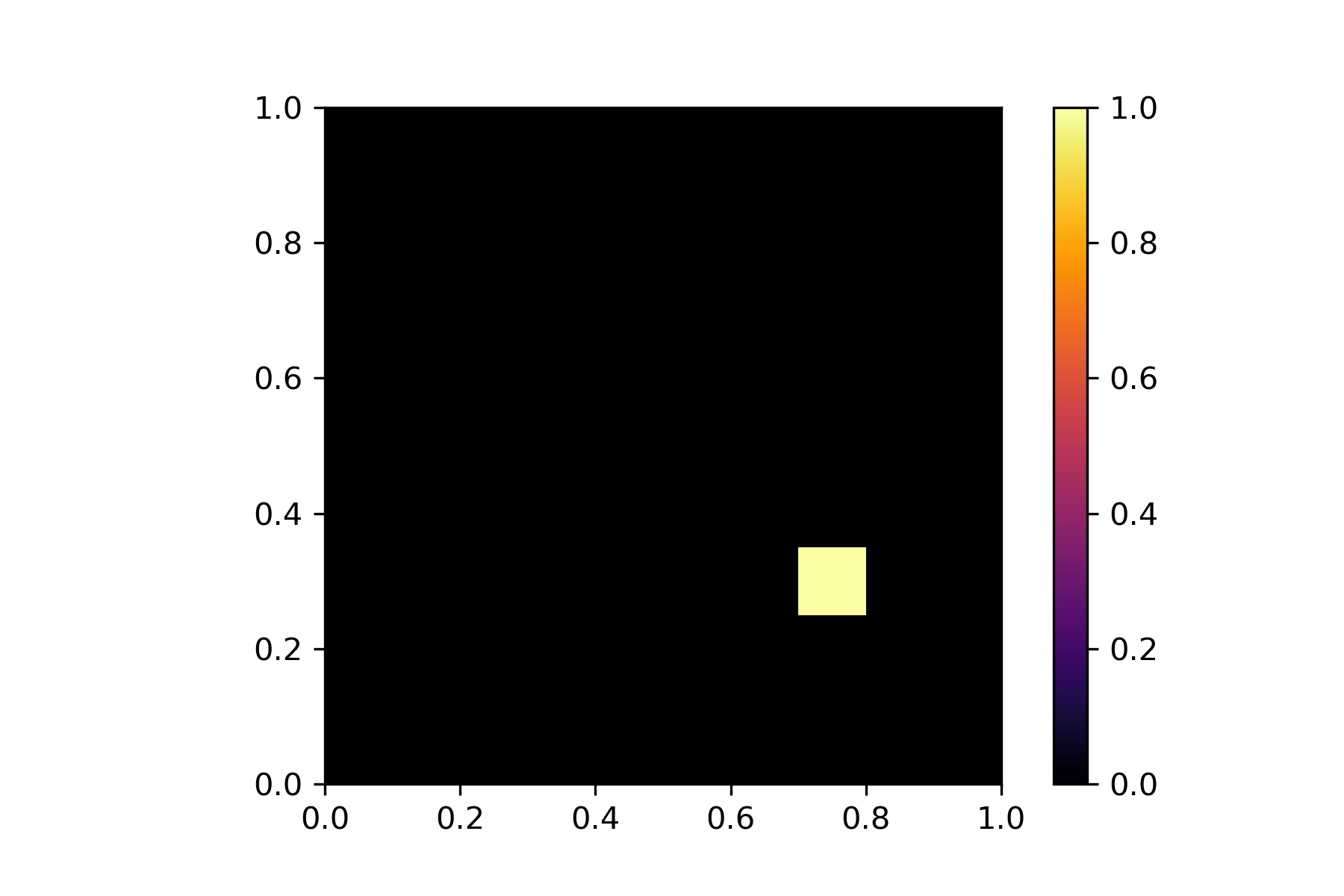}
        \caption{True source}
    \end{subfigure}
    \begin{subfigure}[t]{0.35\linewidth}
        \centering
        \includegraphics[trim=40 30 40 30, clip, width=\linewidth]
        {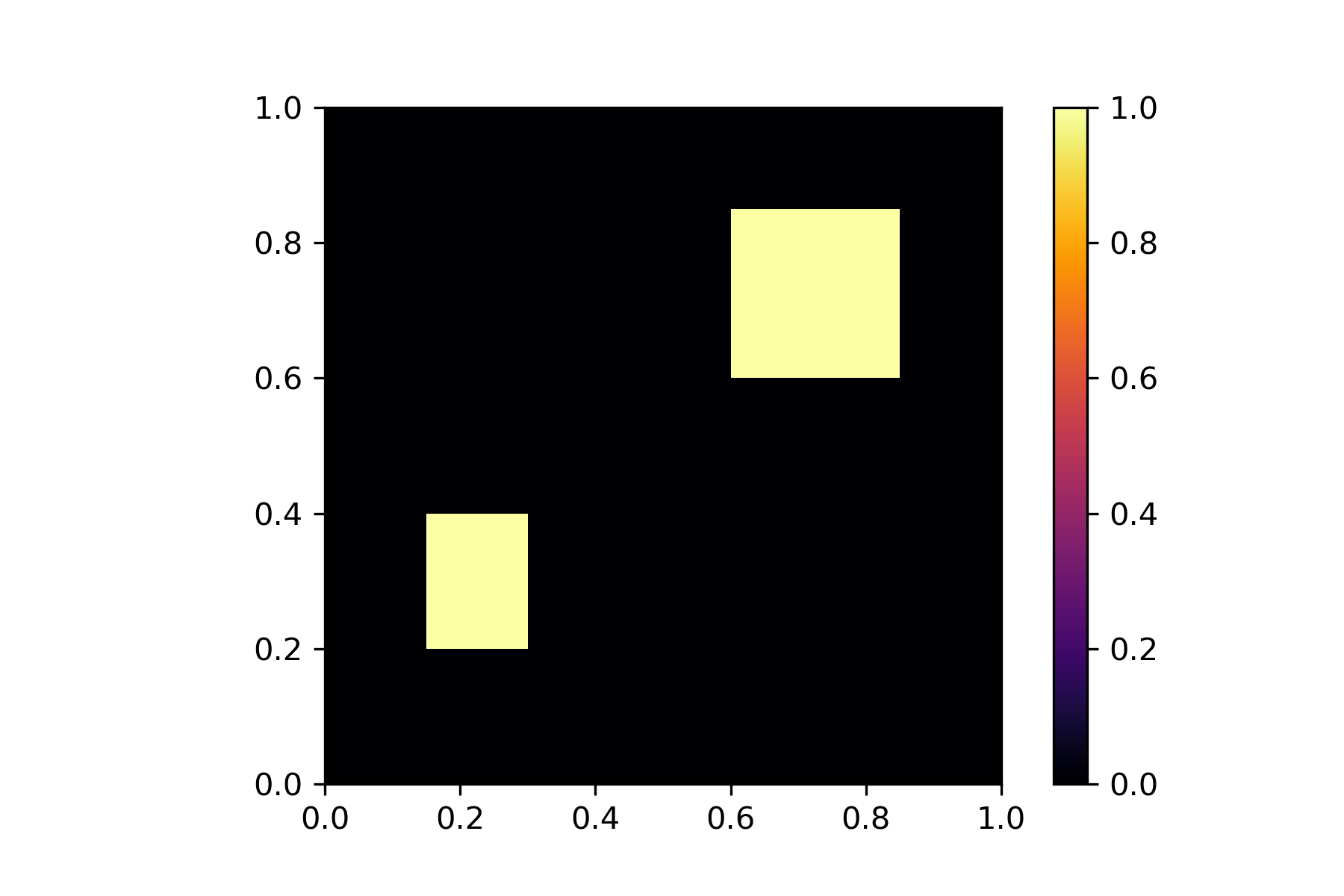}
        \caption{True source}
    \end{subfigure}
    \begin{subfigure}[t]{0.35\linewidth}
        \centering
        \includegraphics[trim=40 30 40 30, clip, width=\linewidth]
        {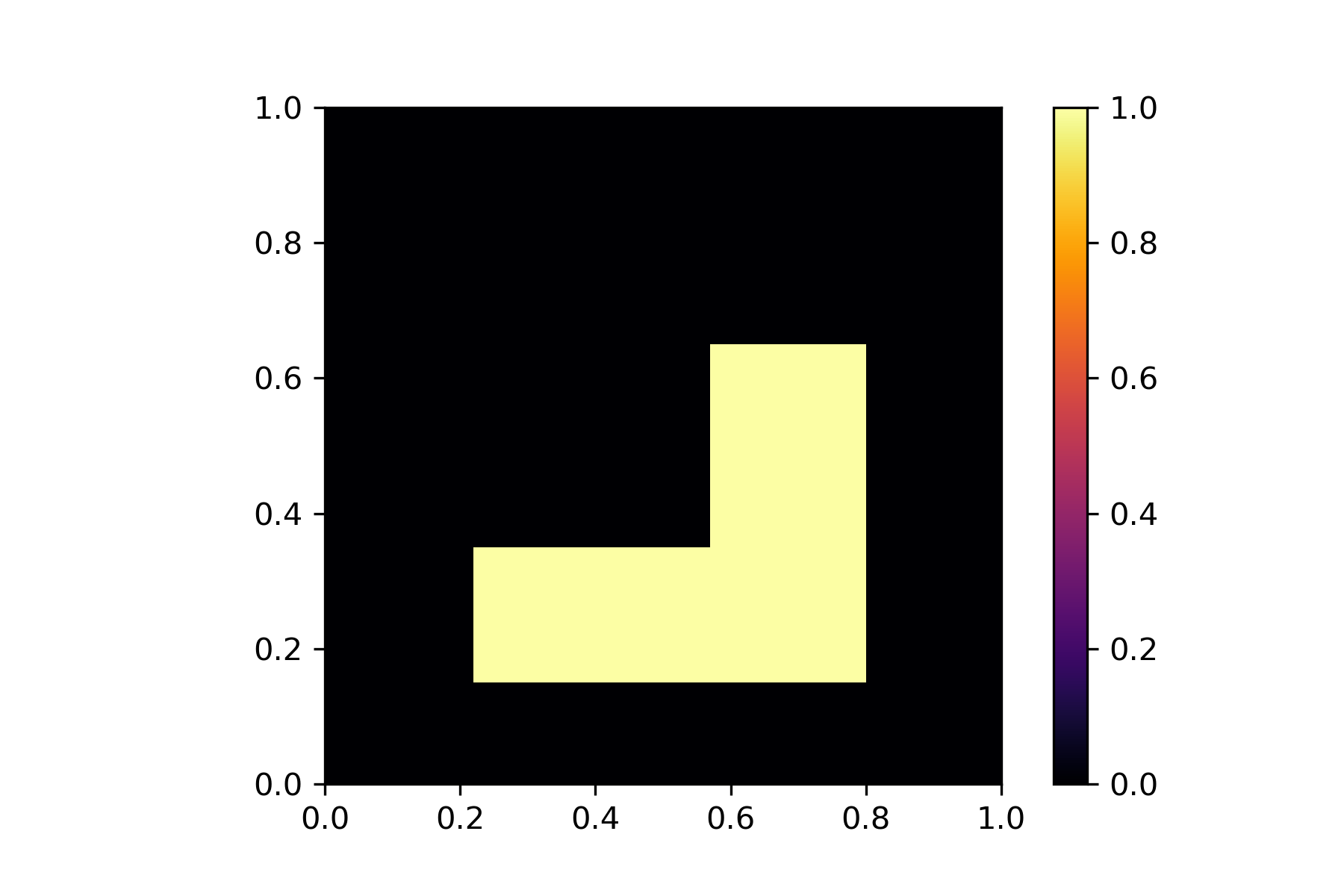}
        \caption{True source}
    \end{subfigure}\par
}\\

\resizebox{.8\textwidth}{!}{
    \begin{subfigure}[t]{0.35\linewidth}
        \centering
        \includegraphics[trim=40 30 40 30, clip, width=\linewidth]
        {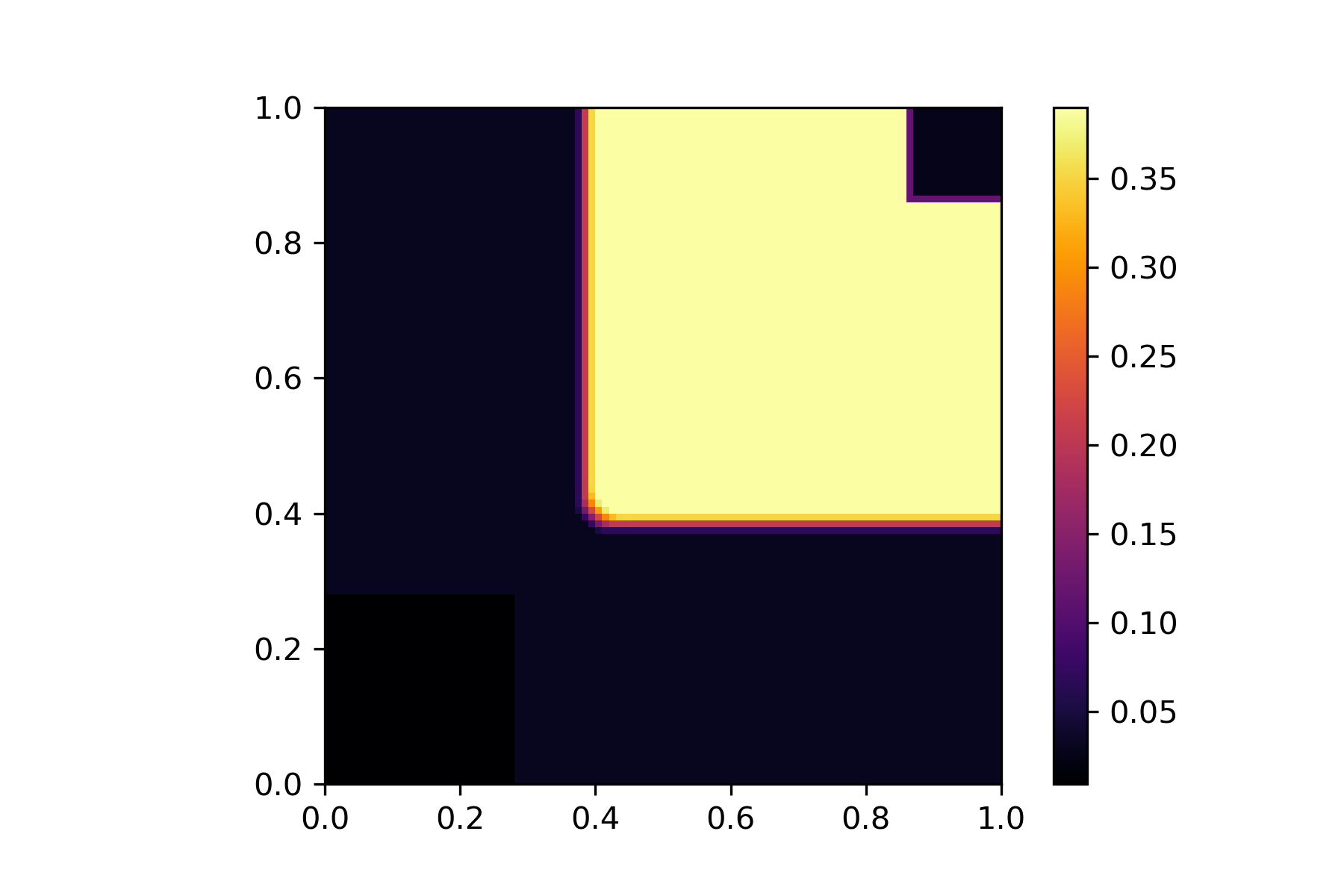}
        \caption{Reconstruction}
    \end{subfigure}
    \begin{subfigure}[t]{0.35\linewidth}
        \centering
        \includegraphics[trim=40 30 40 30, clip, width=\linewidth]
        {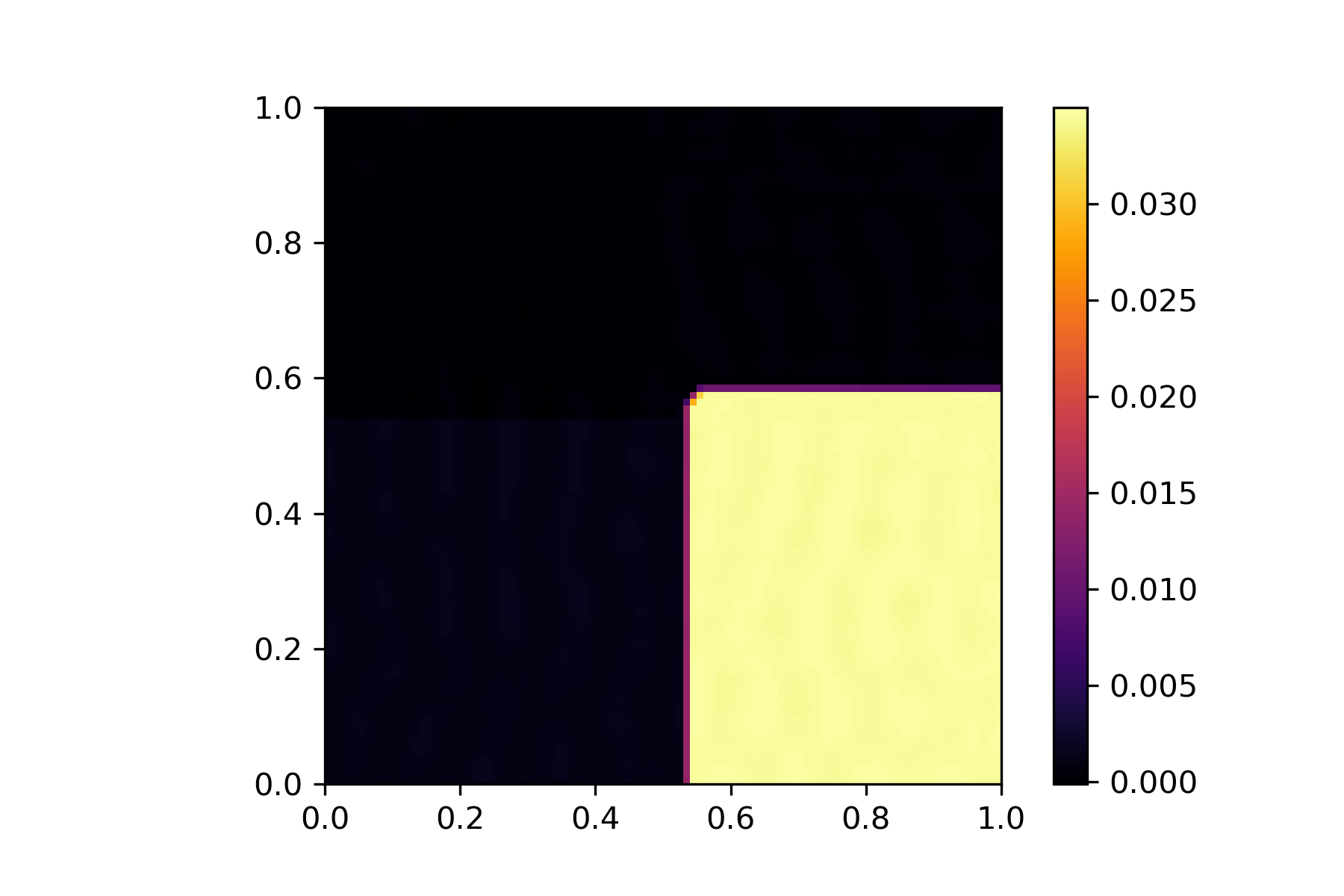}
        \caption{Reconstruction}
    \end{subfigure}
    \begin{subfigure}[t]{0.35\linewidth}
        \centering
        \includegraphics[trim=40 30 40 30, clip, width=\linewidth]
        {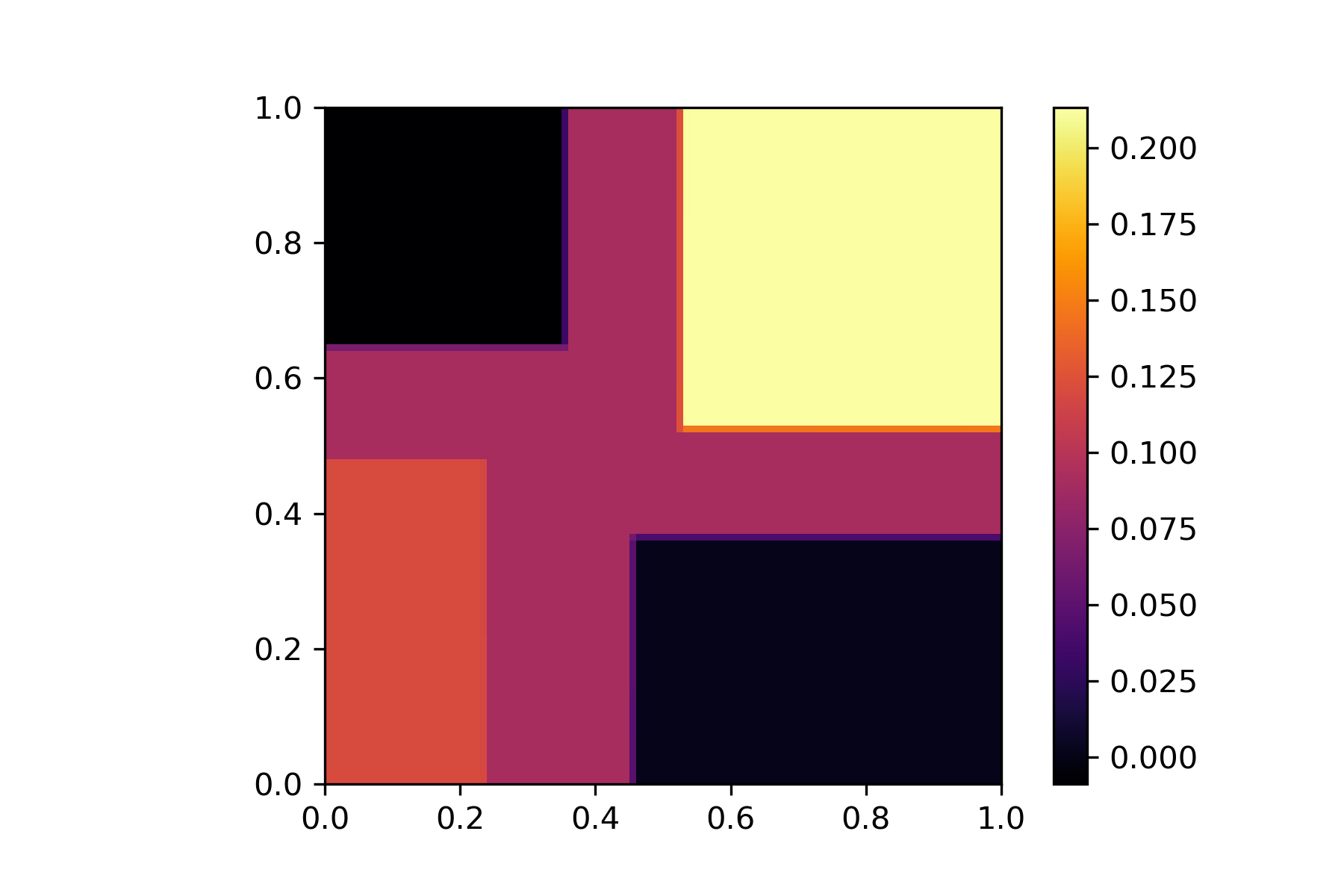}
        \caption{Reconstruction}
    \end{subfigure}
    \begin{subfigure}[t]{0.35\linewidth}
        \centering
        \includegraphics[trim=40 30 40 30, clip, width=\linewidth]
        {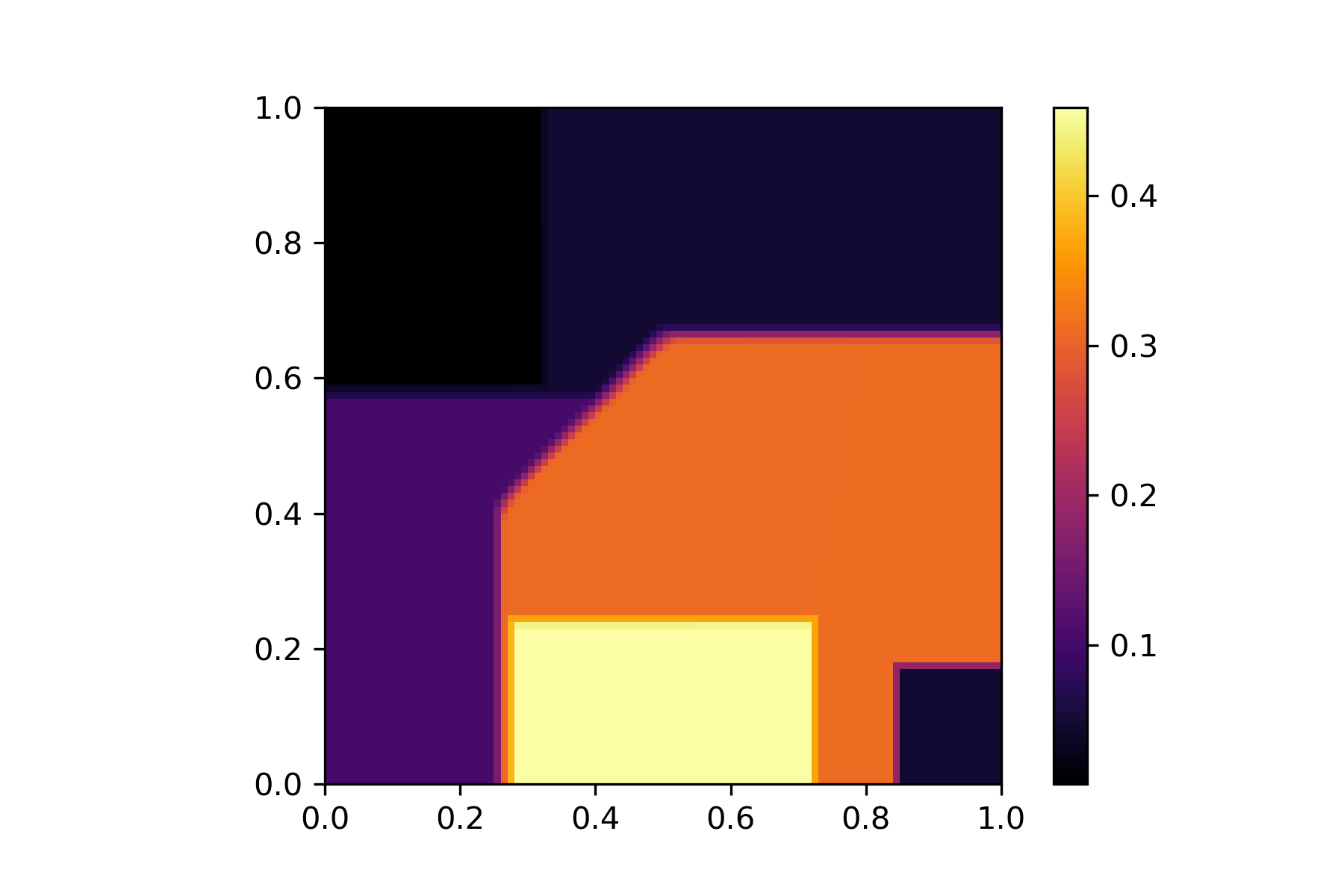}
        \caption{Reconstruction}
    \end{subfigure}\par
}\\
     \caption{\rev{First row: true sources (large, small, double, and L-shaped). Second row: corresponding reconstructions obtained with standard {\em unweighted} total variation regularization. No noise was added to the synthetic boundary data $d=Kf_{\textnormal{true}}$ and $\alpha=10^{-6}$.}}
    \label{fig:motivation}
\end{figure}
\end{revblock}

\section{Two versions of weighted TV-regularization} \label{sec:two_weights}
As mentioned above, our weights are defined in terms of the images under the forward operator $K$ of the partial derivatives of suitable Green's functions. To be more concrete, let us consider the Green's functions for the  Laplace operator, using homogeneous Dirichlet or Neumann boundary conditions:
\[
\begin{minipage}{0.45\linewidth} 
\[
\underbrace{
\begin{aligned}
    -\Delta G^D &= \delta_x &&\text{in } \Omega, \\
    G^D &= 0, &&\text{on } \partial\Omega,
\end{aligned}}_{\textnormal{Green's problem w/ Dirichlet B.C.} }
\]
\end{minipage}
\hfill
\begin{minipage}{0.45\linewidth}
\begin{equation} \label{two_weights:Greens_Neumann}
\underbrace{
\begin{aligned} 
    -\Delta G^N &= \delta_x - \frac{1}{|\Omega|} &&\text{in } \Omega, \\
    \frac{\partial G^N}{\partial n} &= 0 &&\text{on } \partial\Omega. \\
\end{aligned}}_{\textnormal{Green's problem w/ Neumann B.C.}}
\end{equation}
\end{minipage}
\]
Here, $\delta_x(y)=\delta(y-x)$, where $\delta$ denotes the delta function/distribution. 

As will become evident below, a crucial step in defining our TV-weights is to express the source term $f$ in \eqref{intro:main_equation} in terms of $\nabla G^D$ or $\nabla G^N$. 
\begin{revblock}
Throughout the present section $f$ is assumed to be smooth, say $f \in C^1(\overline{\Omega})$, and $\Omega$ is assumed to be a bounded domain with a smooth boundary $\partial \Omega$. The identities \eqref{two_weights:GD} and \eqref{two_weights:GN} below are the standard representation formulas associated with the Green's problems above, see, e.g., \cite{ziemer89}; they are used exclusively to \emph{motivate} the definition of the weights, and none of the results proven below relies on them. In particular, the first integral on the right hand side of \eqref{two_weights:GD} is not meaningful for a general $f \in BV(\Omega)$; the corresponding rigorous statements are Lemma \ref{lem:f_H_expression} in 1D and Theorem \ref{thm:disjointLines} in 2D, where the necessary care is taken. Finally, in order to avoid writing the Dirac distribution inside a Lebesgue integral, we use the duality pairing notation
\begin{equation*}
\langle \delta_x, f \rangle = f(x), \qquad f \in C(\overline{\Omega}),
\end{equation*}
here and in the remainder of the paper.
\end{revblock}
\begin{align}
    \nonumber
    f(x) &= {\color{revcol}\langle \delta_x , f \rangle} \\  
    \nonumber
    &= -\int_\Omega \Delta G^D(y;x)f(y)dy \\ 
    \label{two_weights:GD}
    &=\int_\Omega \nabla G^D(y;x)\cdot \nabla f(y) dy - \int_{\partial\Omega} f(y) \partial_{\mathbf{n}}G^D(y;x)dS(y),
\end{align}
or, similarly,
\begin{align}
    \nonumber
    f(x) &= {\color{revcol}\langle \delta_x , f \rangle} \\
    \nonumber
    &= -\int_\Omega \left(\Delta G^N(y;x)+\frac{1}{|\Omega|}\right)f(y)dy \\ 
    \label{two_weights:GN}
    &= \int_\Omega \nabla G^N(y;x)\cdot \nabla f(y) dy + \frac{1}{|\Omega|}\int_{\Omega}f(y)dy.
\end{align}

From \eqref{two_weights:GD} we find that 
\begin{equation} \label{two_weights_Kf_expression}
    (Kf)(z) = K \left( \int_\Omega \nabla G^D(y;\cdot)\cdot \nabla f(y) dy - \int_{\partial\Omega} f(y) \partial_{\mathbf{n}}G^D(y;\cdot)dS(y) \right) (z), 
\end{equation}
and similarily for \eqref{two_weights:GN}. This expression for $Kf$ motivates, at least to some extent, why suitable weights for TV-regularization might be defined in terms of the images of the partial derivatives of a Green's function: Note that \eqref{two_weights_Kf_expression} involves, potentially in a measure-theoretical sense, $\nabla f$ which is used in the standard definition of the total variation of $f$.  

\rev{The following discussion should be read as a heuristic guideline: it explains which of the two representations we employ in the two situations studied in this paper, and it is not used in any proof.} Whether $f$ should be expressed in the form \eqref{two_weights:GD} or \eqref{two_weights:GN} in the derivation of the weights depends on the problem under consideration. To expand on this, consider Figure \ref{fig:penaltyCases}: If we are mainly searching for sources with a jump discontinuity stretching throughout the entire domain $\Omega$ (panel (a)), it turns out that it is beneficial to apply Neumann boundary conditions, as we do not wish to penalize the support of $f$ along the boundary $\partial \Omega$. If, on the other hand, we are searching for a source inside $\Omega$ (panel (b)), it is wise to keep the boundary integral, i.e., to employ \eqref{two_weights:GD}, in order to not make it erroneously "cheap" to position the source along $\partial \Omega$. \rev{More precisely, the formal mechanism behind this rule of thumb is simply that the Dirichlet representation \eqref{two_weights:GD} produces an extra boundary term, which after the weighting procedure of sections \ref{sec:preliminaries} and \ref{sec:2D3D_analysis} appears as the boundary integral in the extended functional $\overline{TV}_w$, cf.\ \eqref{def:TVew2}, and hence penalizes sources that touch $\partial\Omega$; the Neumann representation \eqref{two_weights:GN} does not.} 

In sections \ref{sec:1D_analysis} and \ref{sec:2D3D_analysis} we will use \eqref{two_weights:GN} and \eqref{two_weights:GD}, respectively. 

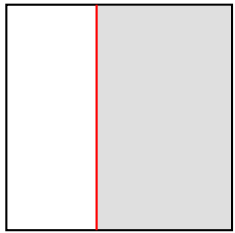
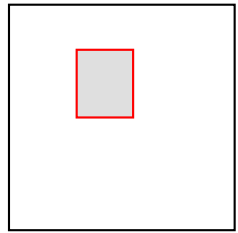
\begin{figure}[h!]
    \centering
    \begin{subfigure}[t]{0.45\linewidth} 
     \centering
      {\begin{tikzpicture}[scale=3]
        \draw[thick] (0,0) rectangle (1,1);
        \filldraw[gray!25, thick] (0.4,0.01) rectangle (0.99,0.99);
        \draw[red, thick] (0.4,0.0) -- (0.4,1);
      \end{tikzpicture}}
      \caption{Source with a jump discontinuity stretching throughout the entire domain.}
    \end{subfigure}
    \hfill
    \begin{subfigure}[t]{0.45\linewidth}
    \centering
      {\begin{tikzpicture}[scale=3]
        \draw[thick] (0,0) rectangle (1,1);
        \filldraw[gray!25, thick] (0.3,0.5) rectangle (0.55,0.8);
        \draw[red, thick] (0.3,0.5) rectangle (0.55,0.8);
      \end{tikzpicture}} 
      \caption{Source inside the domain (away from $\partial \Omega$).}
    \end{subfigure}
    \caption{A domain $\Omega$, with boundary $\partial \Omega$, and the support (in gray) of two prototypical sources.}
    \label{fig:penaltyCases}
\end{figure}

\section{Weighted total variation} \label{sec:preliminaries}

In the following, we provide the basic definitions of weighted total variation to be used throughout. 

\subsection{Weighted total variation}
Let us introduce the \rev{test} space notation  
$$\DD = \{\phi \in C^1_c(\Omega; \mathbb{R}^n): |\phi_i(x)|\leq 1 \ \forall x \in \Omega, \ i \in \{1,\ldots,n\}\},$$
and the weighted \rev{test} space 
\begin{equation} \label{def:D_w}
\Dw = \{\phi \in C^1_c(\Omega; \mathbb{R}^n): |\phi_i(x)|\leq w_i(x) \ \forall x \in \Omega, \ i \in \{1,\ldots,n\}\},
\end{equation}
\begin{revblock} assuming that $\Omega$ is a bounded Lipschitz domain and that the weights $w_i: \overline{\Omega} \rightarrow [0,\infty)$, $i=1,\ldots,n$, are continuous and bounded, $w_i \leq w_{\max} < \infty$. We will return to the continuity of our designed weights in lemmas \ref{lem:contW1d} and \ref{lem:contW2d}.     
\end{revblock}

We can then define the TV-\rev{semi}norm
\begin{equation*}
    TV(f) = \sup_{\phi\in\DD}\int_\Omega f(x) \ \textnormal{div}\ \phi(x) dx,
\end{equation*}
\begin{revblock}
where 
\[ 
f \in BV(\Omega) = \left\{ f \in L^{1}(\Omega) \;:\; \sup_{\varphi \in \DD} \int_{\Omega} f(x)\,\operatorname{div}\varphi(x)\,dx < \infty \right\},  
\]
\end{revblock}
and similarly the weighted TV-\rev{semi}norm
\begin{equation}
    TV_w(f) = \sup_{\phi\in\Dw}\int_\Omega f(x) \ \textnormal{div} \ \phi(x) dx,\label{eq:TVw}
\end{equation}
\rev{where we note from the boundedness assumption that $TV_w(f) \leq w_{\max}TV(f)$.}

Invoking Riesz' representation theorem, we can define the bounded Radon measure $Df$ on $\Omega$, for any $f \in BV(\Omega)$, by
\begin{equation} \label{def:Radon_measure}
    \int_\Omega f(x) \ \textnormal{div} \ \phi(x) \ dx = -\int_\Omega \phi(x) \cdot Df(x), \quad \forall \phi \in C^1_c(\Omega;\mathbb{R}^n).
\end{equation}
It follows that the \rev{anisotropic} variation of this measure coincides with the total variation of the generating function $f \in BV(\Omega)$, i.e.,
\begin{equation*}
    TV(f) = \int_\Omega \sum_i^n|D_if(x)|,
\end{equation*}
see, e.g., \cite{AFP} for details. Here, $D_if$ represents the i'th component of the vector-valued Radon measure $Df$. Similarly, \rev{the dual representation of the weighted anisotropic variation yields}
\begin{align} \label{def:TVw2}
    TV_w(f) &= \sum_{i=1}^n \int_\Omega  w_i(x)|D_if(x)|,
\end{align}
where \rev{the continuity of the weights ensures the validity of the representation with the admissible class $\Dw$.} 

\subsection{Weighted extended total variation}
We will now introduce an \rev{extended} weighted dual space notation, where we do no longer assume compact support of the test functions. That is, we let
\begin{equation} \label{def:De_w}
\Dw^e = \{\phi \in C^1(\overline{\Omega}; \mathbb{R}^n): |\phi_i(x)|\leq w_i(x) \ \forall x \in \overline\Omega, \ i \in \{1,\ldots,n\}\}.
\end{equation}
We can then define the extended weighted TV-functional
\begin{equation}
    \overline{TV}_w(f) = \sup_{\phi\in\Dw^e}\int_\Omega f(x) \ \textnormal{div} \ \phi(x) dx.\label{eq:TVew}
\end{equation}
 \rev{Applying the Gauss-Green formula for BV functions, followed by the dual representation of the weighted anisotropic variation in the interior and the pointwise maximization of the boundary term, cf. also \eqref{two_weights:GD}, yields}
\begin{align} \label{def:TVew2}
    \overline{TV}_w(f) &= \sum_{i=1}^n \int_\Omega  w_i(x)|D_if(x)| + \int_{\partial\Omega}w_\partial(x) |f(x)|dS(x),
\end{align}
\rev{where the continuity of the weights ensures the validity of this representation with the admissible class $\Dw^e$.}
Here, $w_\partial = \sum_i |n_i|w_i$ $(n_i$ is the i'th component of the outer unit normal vector of $\partial\Omega$) \rev{and the function $f$ appearing in the boundary integral is to be understood in the sense of traces: $f|_{\partial\Omega} \in L^1(\partial\Omega)$ for every $f \in BV(\Omega)$ because $\Omega$ is a bounded Lipschitz domain.}

\begin{revblock}
\begin{remark}\label{rem:weighted_BV_space}
For general weighting procedures it would be natural to require a uniform
positive lower bound on the weights, which would immediately give
well-definedness in $BV(\Omega)$ and equivalence of the seminorms $TV$ and
$TV_w$. This is not directly available for weights generated by partial derivatives of Green's functions, as the boundary conditions, cf. \eqref{two_weights:Greens_Neumann}, may cause the weights to deteriorate in specific directions towards the boundary. The natural domain of the regularizer is consequently
\begin{equation*}
    BV_w(\Omega)=\{f\in L^1(\Omega):TV_w(f)<\infty\},
\end{equation*}
which contains $BV(\Omega)$. An analysis of $BV_w(\Omega)$ --- traces, compact
embeddings, and existence of minimizers --- requires additional assumptions on the null space $\mathcal{N}(K)$ of $K$ and is left to future research. Nevertheless, exact reconstruction results can still be achieved in $BV(\Omega)$, which will be the main focus in the remainder of this paper. 
We also note that $w_i(y)>0$ for all $y\in\Omega$ when $\partial_{y_i}G(\cdot\,;y)\notin\nullspace{K}$, i.e.\ when the operator does not annihilate any interior dipole. Under this mild condition on
$\nullspace{K}$, the seminorms $TV$ and $TV_w$ are equivalent on any $\Sigma\Subset\Omega$. 
\end{remark}
\end{revblock}

\section{1D analysis} \label{sec:1D_analysis}
In this section we focus on the one dimensional case, i.e., $\Omega=(0,1)$, \rev{and the Banach space $X$ equals $L^2(\Omega)$, cf. \eqref{eq:forward_operator_X}}. 
Let $H_{x^*}$ denote the $x^*$-shifted Heaviside function 
\begin{equation*}
    H_{x^*}(x) = H(x-x^*)
\end{equation*}
with associated shifted delta Dirac function $\delta_{x^*}$. 
We will now analyze the following "jump pursuit" problem
\begin{equation} \label{eq:jump_pursuit_problem}
    \min_{f \in BV(0,1)} TV_w(f) \textnormal{ subject to } K f=K (\rho H_{x^*} + \tau), 
\end{equation} 
where $\rho$ and $\tau$ are constants incorporating the hight of the jump and the "base line" of the jump function, respectively, and $w(x)$ is an appropriate weight function (defined below). 

We will prove that $\rho H_{x^*} + \tau$ solves \eqref{eq:jump_pursuit_problem}. This is a challenging issue in the present context because the involved forward operator $K:L^2(0,1) \rightarrow L^2(E)$ has a nontrivial null space, i.e., $E$ is typically a subset of $\overline{\Omega}$. Note that our analysis covers the boundary-observations-only case $E=\{0,1\}$. Moreover, our exploration is not limited to state equations defined in terms of elliptic PDEs, but rather holds for general linear maps $K$ with significant null spaces.     

As will become clear below, functions with zero integral play an important role in this section:
\begin{equation*}
    \bar{f}(x)=f(x)-\int_0^1 f(z) \, dz, 
\end{equation*}
and in particular  
\begin{equation}
    \bar{H}_y (x) = H_y(x) - \int_0^1 H_y(z) \, dz 
    = \left\{ 
    \begin{array}{cc}
       y-1,  & x<y, \\
       y,  &   x>y.
    \end{array}\right. \label{eq:heaviY} 
\end{equation}


The total variation of functions which only differ by a constant is the same. Therefore, the image under $K$ of the constant function $1$ plays an important role in the present context. More precisely, defining the orthogonal projection onto the orthogonal complement of the space spanned by $K1$, 
\begin{equation}\label{eq:ortho_proj}
    Q:L^2(E) \rightarrow \left\{ t K1 \, | \, t \in \mathbb{R} \right\}^\perp, 
\end{equation}
we define the weight function to be used in connection with TV regularization in 1D as follows
\begin{equation} \label{def:w}
    w(y) = \| C \bar{H}_y \|_{L^2(E)}, \quad y \in \Omega,     
\end{equation}
where 
\begin{equation} \label{def:C}
    C=QK. 
\end{equation}
Note that, $C d = 0$ for any constant (function) $d$. 

\rev{If $K1 = 0$, then $Q=I$, which does not introduce any significant extra challenges. Nevertheless, we only present versions of Theorem \ref{thm:variational_rigorous_proof} and for Lemma \ref{lem:KtoC} below for the case $K1 \neq 0$.} 

\subsubsection*{Remarks}
Recall the definition \eqref{eq:heaviY} of $\bar{H}_y (x)$. Consider the function 
\begin{equation*}
    G(y;x)=-\left\{ 
    \begin{array}{cc}
       \frac{1}{2} y^2 + \frac{1}{2} x^2 -y,  & x<y, \\
       \frac{1}{2} y^2 + \frac{1}{2} x^2 -x,  &   x>y,
    \end{array}\right.
\end{equation*}
which satisfies 
\begin{equation*}
    \frac{d}{dy} G(y;x) = - \bar{H}_y (x)
\end{equation*}
and 
\begin{equation*}
    \frac{d^2}{dy^2} G(y;x) = \delta_x(y)-1, 
\end{equation*}
cf. \eqref{two_weights:Greens_Neumann}, the version with homogeneous Neumann boundary conditions. This choice of a Green's function in the present situation is motivated by the fact that we want to recover a single-jump source $\rho H_{x^*} + \tau$, and this constitute the 1D version of the source depicted in panel (a) in Figure \ref{fig:penaltyCases}.   

We emphasize that, in this section, the weighting is not defined in terms of the original forward operator $K$, but we rather use the modified operator $C$. Note that any solution $f$ of \eqref{intro:main_equation} also satisfies $QK f = Q d$. The rationale behind this choice will become clear below.


\subsection{Jump pursuit}
We need two lemmas in order to prove the abovemention recovery result for \eqref{eq:jump_pursuit_problem}.

\begin{lemma} \label{lem:f_H_expression}
    Assume that $f \in BV(0,1)$ and let $\bar{H}_y$ be defined as in \eqref{eq:heaviY}. Then, recalling the definition \eqref{def:Radon_measure} of the Radon measure $Df$, we can write
    \begin{equation} \label{eq:1D_GN}
        f(x) = \int_0^1 \bar{H}_y(x)Df(y) \ + \eta \quad \textnormal{a.e.,} 
    \end{equation}  
    for $\eta$ a constant.
\end{lemma}
(Before we prove this lemma, we observe the similarities between equations \eqref{eq:1D_GN} and \eqref{two_weights:GN}, keeping in mind that $\bar{H}_y$ is the derivative of the aforementioned Green's function $\bar{G}_y$.)


\begin{proof}
    \begin{revblock}
      Let us first comment on the existence of the integral on the right-hand side of \eqref{eq:1D_GN}. For fixed $x \in (0,1)$, the map $y \mapsto \bar{H}_y(x)$ is Borel measurable and bounded by $1$, cf.\ \eqref{eq:heaviY}, and is therefore integrable with respect to $|Df|$, the latter being a finite Radon measure since $f \in BV(0,1)$. 
      
      To derive the equality in \eqref{eq:1D_GN}, we will multiply both sides
  \end{revblock} by the derivative of an arbitrary smooth function $\phi$ with compact support and integrate, showing that these integrals become identical. That is, we will show that 
  \begin{equation}
      \int_0^1 f(x)\phi'(x) dx = \int_0^1 \left(\int_0^1 \bar{H}_y(x)Df(y)\right)\phi'(x) dx \quad \forall \phi \in \DD. \label{eq:vareq}
  \end{equation}
  Indeed, by the definition of the Radon measure, the left-hand side satisfies 
  \begin{equation} \label{eq:left_hand_side}
      \int_0^1 f(x)\phi'(x)dx = -\int_0^1\phi(x)Df(x).
  \end{equation}
   A straightforward computation reveals that, for any $\phi \in \DD$,   
  \begin{eqnarray}
      \int_0^1\bar{H}_y(x)\phi'(x)dx &=& \int_0^y (y-1) \phi'(x)dx + \int_y^1 y\phi'(x) dx \nonumber \\&=& 
      y\left[\phi(x)\right]_0^1 - \left[\phi(x)\right]_0^y \nonumber \\ &=& -\phi(y). \label{eq:intphi}
  \end{eqnarray}
  Using Fubini's theorem, \rev{which is applicable since $(x,y) \mapsto \bar{H} _y(x)\phi'(x)$ is bounded and Borel measurable on $(0,1)^2$ and $\mathcal{L}^1 \otimes |Df|$ is a finite measure,} we get that the right-hand side of \eqref{eq:vareq} obeys 
  \begin{eqnarray} \nonumber
      \int_0^1 \left(\int_0^1 \bar{H}_y(x)Df(y)\right)\phi'(x) dx &=& \int_0^1\left(\int_0^1\bar{H}_y(x)\phi'(x) 
      dx\right) Df(y) \\
      \label{eq:right_hand_side} 
      &=&
      -\int_0^1 \phi(y)Df(y).
  \end{eqnarray}
  Equality \eqref{eq:vareq} now follows from \eqref{eq:left_hand_side} and \eqref{eq:right_hand_side}. 

  Before proceeding, note that
  \begin{eqnarray}
      \int_0^1 \left|\int_0^1 \bar{H}_y(x)Df(y)\right| dx &\leq& \int_0^1 \int_0^1 \left|\bar{H}_y(x) \right| dx |Df(y)| \nonumber \\ &\leq& \int_0^1 |Df(y)| < \infty,  \label{eq:l1bound}
  \end{eqnarray}
  which shows that $\int_0^1 \bar{H}_y(x)Df(y) \in L^1(0,1)$. 

  Now, by rewriting \eqref{eq:vareq} as  
  \begin{equation*}
      \int_0^1 \left[f(x) - \int_0^1\bar{H}_y(x)Df(y) \right]\phi'(x) dx = 0 \quad \forall \phi\in\DD,
  \end{equation*}
  it only remains to show that this will imply that $g(x) := f(x) - \int_0^1\bar{H}_y(x)Df(y)$ is constant a.e. Clearly, from \eqref{eq:l1bound} and since $f \in BV(0,1)$, we conclude that $g \in L^1(0,1)$. Furthermore, since the weak derivative of $g$ equals zero, it immediately follows that $g \in W^{1,1}(0,1)$. Therefore, from the fundamental lemma of calculus of variations, and a standard argument using mollifiers, we get that $g(x) = \eta$ a.e., which completes the proof.  
\end{proof}

Next, we argue that the proposed weight function is continuous, which implies that $TV_w(f)$ is well-defined for any $f \in BV(0,1)$. 
\begin{lemma} \label{lem:contW1d}
    Let $w: \Omega \rightarrow \mathbb{R}_+$ be defined as in \eqref{def:w} and assume that $K: L^2(0,1) \rightarrow L^2(E)$ is continuous. Then $w \in C(\Omega)$. 
\end{lemma}
\begin{proof}
    It is a straightforward calculation to show that the map $y \mapsto \bar{H}_y$ is continuous as a mapping from $\Omega=(0,1)$ to $L^p(0,1)$, $1 \leq p < \infty$. Since we assume that $K$ is continuous and the orthogonal projection $Q$ is continuous, it follows that $C = QK$ and consequently $y \mapsto \|C\bar{H}_y\|_{L^2(E)}$ are continuous.  
\end{proof}

\rev{In this section we do not need to extend the weight function $w$ to $\overline{\Omega}$, and to investigate the continuity of $w$ on $\overline{\Omega}$, because the definition \eqref{def:w} does not involve any boundary term. (As mentioned above, we employ a Green's function defined in terms of the 1D Laplace operator and homogeneous Neumann boundary conditions).} 

\begin{theorem} \label{thm:jump_pursuit_rigorous_proof}
    \rev{Assume that $K: L^2(0,1) \rightarrow L^2(E)$ is continuous and let} $w$ be the weight function defined in \eqref{def:w}. Then, $\rho H_{x^*} + \tau$ solves the jump pursuit problem \eqref{eq:jump_pursuit_problem}.
\end{theorem}
\begin{proof} 
 We could use the Lebesgue decomposition of $\bar{H}_{x^*}$ to show that the measure $D\bar{H}_{x^*}$ becomes the Dirac mass $\delta_{x^*}$. Alternatively, we can obtain this result from the distributional definition, employing \eqref{eq:TVw}, \eqref{eq:intphi} and \eqref{def:D_w}, 
\begin{align}
    \nonumber
    TV_w(\rho\bar{H}_{x^*}) &= \sup_{\phi \in \Dw} \int_0^1 \rho\bar{H}_{x^*}(x) \phi'(x) dx 
   =\sup_{\phi \in \Dw} \left\{- \rho\phi(x^*)  \right\} \\ 
    \label{eq:first_part_1D}
    &= |\rho|w(x^*).
\end{align}

Finally, we can now show the main result. Consider any $f \in BV(0,1)$ for which $Kf=K(\rho H_{x^*}+\tau)$, or $Cf=C(\rho H_{x^*})$ because $C$ annihilates constants, cf. \eqref{def:C}. We get, recalling \eqref{eq:first_part_1D} and the definition \eqref{def:w} of the weight function $w$, 
\begin{eqnarray} \nonumber
    TV_w(\rho\bar{H}_{x^*}+\tau) &=& TV_w(\rho\bar{H}_{x^*}) = |\rho| w(x^*) 
    = |\rho|\|CH_{x^*}\|_{L^2(E)} \\
    \nonumber
    &=& \|C(\rho H_{x^*})\|_{L^2(E)} 
    =\|C(\rho H_{x^*}+\tau)\|_{L^2(E)} \\ 
    \nonumber
    &=&\|Cf\|_{L^2(E)} \\
   \nonumber
    &=& \left(\int_E \left[\left(\int_0^1 C\bar{H}_y(\cdot) D{f}(y) \right) \right]^2dz\right)^{1/2} \\ 
    \nonumber
    &\leq&\int_0^1\left(\int_E \left[C\bar{H}_y(\cdot)\right]^2 dz\right)^{1/2}|D{f}(y)| 
    = \int_0^1 w(y)|Df(y)| \\
    \label{eq:ground_state_inequality}
    &=& TV_w(f),
\end{eqnarray}
where we have used Lemma \ref{lem:f_H_expression} in the seventh equality. We have also employed Minkowski's integral inequality and that 
\begin{equation*}
    C \left(\int_0^1 \bar{H}_y(\cdot) D{f}(y) + \eta \right) = \int_0^1 C\bar{H}_y(\cdot) D{f}(y),  
\end{equation*} 
which holds because the mapping \rev{$y \mapsto \bar{H}_y(\cdot)$ is Bochner integrable as an $L^2(0,1)$-valued function, and $C=QK$ is continuous/bounded  on $L^2(0,1)$. This implies that the operator $C$ may be interchanged with the Bochner integral. Also note that $C \eta =0$ because $\eta$ is a constant.} 
\end{proof}

\subsubsection*{Remark}
Here, we defined the weight function $w(y)$ in terms of the $L^2$-norm, see \eqref{def:w}. Since Minkowski's integral inequality holds for $L^p$-norms, Theorem \ref{thm:jump_pursuit_rigorous_proof} also holds if one employs the weight function $$w(y)=\| C \bar{H}_y \|_{L^p(E)}, \, y \in \Omega,$$ and $1\leq p < \infty$, provided that $Cd=0$ for any constant $d$, jmf. \eqref{def:C}. The latter would require that one can define an appropriate operator $Q$.

\subsection{Regularized problem}
Let us also consider the associated variational formulation
\begin{equation}\label{eq:variational_problem}
    \min_{f\in BV(0,1)} \left\{ \frac{1}{2}\left\|Kf-K(\rho \bar{H}_{x^*}+\tau)\right\|^2 + \alpha TV_w(f) \right\}.
\end{equation}
The proof of the following theorem is presented in Appendix \ref{sec:regularized_problem_1D}. 

\begin{theorem} \label{thm:variational_rigorous_proof}
    \rev{Assume that $K: L^2(0,1) \rightarrow L^2(E)$ is continuous and let} $w$ be the weight function defined in \eqref{def:w}. Then, 
    $$f_\alpha = \gamma\rho \bar{H}_{x^*} + \eta$$ solves the weighted total variational regularization problem \eqref{eq:variational_problem},
    where
    \begin{equation*}
        \gamma = 1-\frac{\alpha}{|\rho|\|C\bar{H}_{x^*}\|} \quad \textnormal{and} \quad \eta = \tau + \frac{(1-\gamma)\rho(K\bar{H}_{x^*},K1)}{\|K1\|^2}, 
    \end{equation*}
    provided that \rev{$K1 \neq 0$} and that $0 < \alpha < |\rho|\|C\bar{H}_{x^*}\|$. Note that $\gamma \rightarrow 1$ and $\eta \rightarrow \tau$ as $\alpha \rightarrow 0$. 
\end{theorem}

\begin{revblock}
Let us also mention that the analysis above can also be carried out in the abstract framework of ground states as defined in \cite{benning2012ground} ( see also \cite{burger2022nonlinear} for a more recent overview), where we recall that ground states correspond to maximizers of Rayleigh quotients. While we give explicit computations in this paper for the sake of self-containedness, it is worth noting that for the weighted total variation the step function $\bar{H}_y$ is a ground state regardless of $y \in (0,1)$. More specifically, \eqref{def:w} implies that  
\begin{equation*}
    \frac{\| C \bar{H}_y \|_{L^2(E)}}{TV_w(\bar{H}_y)} = \frac{w(y)}{\int_0^1 w(x) |D \bar{H}_y(x)|} =\frac{w(y)}{w(y)} = 1, 
\end{equation*}
and for any $f \in BV(0,1)$ we can conclude from the four last lines in \eqref{eq:ground_state_inequality} that 
\begin{equation*}
    \frac{\| C f \|_{L^2(E)}}{TV_w(f)} \leq 1.  
\end{equation*}
 A similar result does {\em not} hold for the Rayleigh quotient when $TV_w$ is replaced by $TV$. (In general, ground states are reconstructed exactly by the constrained problem and approximately by the unconstrained problem \cite{benning2012ground,burger2022nonlinear}).
\end{revblock}

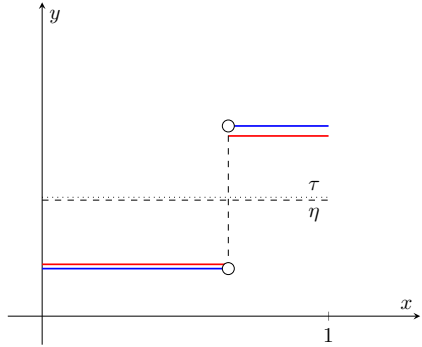
\begin{figure}[htbp]
    \centering
    \begin{tikzpicture}[scale=.8]
    \begin{axis}[
        axis x line=middle,
        axis y line=middle,
        xlabel={$x$},
        ylabel={$y$},
        xmin=0, xmax=1.2,
        ymin=0, ymax=3,
        xtick={0, 1},
        ytick={0},
        enlargelimits=true,
        domain=0:1,
        samples=100,
        legend pos=north east
    ]
    
    \def\xstar{0.65}  
    \def\rho{1.5}    
    
    \addplot[blue, thick, domain=0:\xstar] {1.25 - 0.5 * 1.5} node[pos=0.5, below] {};
    \addplot[blue, thick, domain=\xstar:1] {1.25 + 0.5 * 1.5};
    
    \addplot[red, thick, domain=0:\xstar] {1.22 - 0.5 * 0.9*1.5} node[pos=0.5, below] {};
    \addplot[red, thick, domain=\xstar:1] {1.22 + 0.5 * 0.9*1.5};
    
    \draw[dashed] (axis cs:\xstar,1.25-0.5*1.5) -- (axis cs:\xstar,{1.25 + 0.5 * 1.5});
    
    \node[fill=white, inner sep=2pt, draw, circle] at (axis cs:\xstar,{1.25 - 0.5 * 1.5}) {};
    \node[fill=white, inner sep=2pt, draw, circle] at (axis cs:\xstar,{1.25 + 0.5 * 1.5}) {};
    
    
    \addplot[dotted, black] coordinates {(0,1.25) (1,1.25)} node[pos=0.95, above] {$\tau$};
    
    \addplot[dashed, black] coordinates {(0,1.22) (1,1.22)} node[pos=0.95, below] {$\eta$};
    
    \end{axis}
    \end{tikzpicture}
    \caption{A visualization of how the exact solution $\rho \bar{H}_{x^*}+\tau$ (in blue color) and the inverse solution $f_\alpha$ (in red color) presented in Theorem \ref{thm:variational_rigorous_proof} are related. The jump is correctly located, but slightly smaller and the center line is slightly shifted depending on the inner-product $(K\bar{H}_{x^*}, K1)$.}
\end{figure}

\section{2D analysis} \label{sec:2D3D_analysis}
In 1D we defined the weights in terms of the shifted Heaviside function \eqref{eq:heaviY}. And, as made clear by the proofs of theorems \ref{thm:jump_pursuit_rigorous_proof} and \ref{thm:variational_rigorous_proof}, the crucial step in showing that \eqref{eq:jump_pursuit_problem} and \eqref{eq:variational_problem} admit "one-jump" minimizers lies in the relationship between the Heaviside function $\bar{H}_{x^*}$ and the Dirac mass $\delta_{x^*}$. The purpose of this section is to generalize these ideas to 2D. To that end, it is useful to recall that one may regard $\bar{H}_{x^*}$ as a derivative of a Green's function associated with $x^*$. 

Given $\Omega \subset \mathbb{R}^2$, let $G(x;x^*)$ be a Green's function associated with this domain with a singularity at $x^*$ and subject to homogeneous Dirichlet boundary conditions, cf. \eqref{two_weights:GD} and the discussion in Section \ref{sec:two_weights}. Keeping the close connection between the Dirac function $\delta_{x^*}$ and $G(x;x^*)$ in mind, we can express a \rev{smooth} function $f:\Omega \rightarrow \mathbb{R}$ in the form
\begin{eqnarray*}
    f(x) &=& {\color{revcol}\langle \delta_x, f \rangle} = -\int_\Omega f(y)\Delta_y G(y;x)dy \\
    &=& \int_\Omega \nabla_{\!y} G(y;x) \cdot Df(y)  - \int_{\partial\Omega} f(y)\partial_\mathbf{n}G(y;x)dS(y) \\
    &=& \int_\Omega \nabla_{\!y}G(x;y) \cdot Df(y) - \int_{\partial\Omega} f(y)\partial_\mathbf{n}G(x;y)dS(y),
\end{eqnarray*}
where we used the symmetry of the Green's function in the last equality. \rev{As in Section \ref{sec:two_weights}, this computation is formal and serves only to motivate the definition of the weights; the rigorous counterpart of the resulting estimate is Lemma \ref{lem:upper_bound_BV} below.} 

Now, assuming sufficient regularity, and following the 1D analysis, we get the upper bound 
\begin{eqnarray}
    \|Kf\|_{L^1(E)} &=& \int_{E} |(Kf)(z)| \ dz \nonumber \\
    &=& \int_{E} \left|K\left[\int_\Omega \nabla f(y) \cdot \nabla_{\!y} G(\cdot;y) dy - \int_{\partial\Omega} f(y)\partial_\mathbf{n}G(\cdot;y)dS(y)\right]\right| \ dz \nonumber\\
    &=& \int_E\left|\int_\Omega \partial_{y_1}f(y) K \partial_{y_1} G(\cdot;y) dy + \int_\Omega \partial_{y_2}f(y) K \partial_{y_2} G(\cdot;y) dy \right. \nonumber\\ &-& \left. \int_{\partial\Omega} f(y)K\partial_\mathbf{n}G(\cdot;y)dS(y)\right| \ dz \nonumber \\
    &\leq& \int_\Omega \left[\int_E \left|K\partial_{y_1}G(\cdot;y) \right|dz\right] |\partial_{y_1}f(y)| dy \nonumber\\ &+& 
    \int_\Omega \left[\int_E \left|K\partial_{y_2}G(\cdot;y) \right|dz\right] |\partial_{y_2}f(y)| dy \nonumber\\
    &+& \int_{\partial{\Omega}} \left[\int_E \left| K\partial_{\mathbf{n}}G(\cdot;y) \right| dz \right] |f(y)| dS(y) := \overline{TV}_w(f), \label{eq:alfred2}
\end{eqnarray}
where the weights are defined by
\begin{eqnarray} \label{eq:weights_definition}
    w_1(y) &=& \int_E \left|K\partial_{y_1}G(\cdot;y) \right|dz, \nonumber \\
    w_2(y) &=& \int_E \left|K\partial_{y_2}G(\cdot;y) \right|dz, \\
    w_\partial(y) &=& \int_E \left|K\partial_{\mathbf{n}}G(\cdot;y) \right|dz. \nonumber
\end{eqnarray}
Note that $w_\partial$ is only present when the Green's function is generated using Dirichlet boundary conditions. Here we use Dirichlet boundary conditions because we want to recover sources well/properly inside $\Omega$, cf. Section \ref{sec:two_weights}. 

\rev{The definition of the weights \eqref{eq:weights_definition} require that we can apply the forward operator $K$ to the partial derivatives of the Green's functions $G$, which in $2$D typically do not have $L^2$-regularity. Hence we assume that $X=L^1(\Omega)$ throughout this section, cf. \eqref{eq:forward_operator_X}.}

\begin{revblock}
The following remark and Lemma \ref{lem:contW2d} concern the existence and continuity of the weights \eqref{eq:weights_definition}.

\begin{remark}[Boundary gradient trace and the Poisson kernel]\label{rem:poissonker}

Assume that $\Omega\subset\mathbb{R}^2$ is bounded with $C^{1,1}$-boundary. Let $x\in\Omega$ be fixed and let $y_0\in\partial\Omega$. Since $x\neq y_0$, there exists a neighbourhood $U$ of $y_0$ such that $x\notin U$. Hence the map 
\begin{equation*} 
    y\mapsto G(x;y)
\end{equation*}
has no singularity in $U\cap\Omega$. In particular, $G(x;\cdot)$ is harmonic in $U\cap\Omega$ and vanishes continuously on $U\cap\partial\Omega$. By local boundary $W^{2,p}$-regularity for the Dirichlet problem in $C^{1,1}$ domains, for every $p<\infty$ and every $U'\Subset U$ with $y_0\in U'$, we have 
\begin{equation*}
G(x;\cdot)\in W^{2,p}(U'\cap\Omega).
\end{equation*}
Choosing $p>2$ and using the Sobolev--Morrey embedding in dimension two, it follows that 
\begin{equation*}
G(x,\cdot)\in C^{1,\alpha}(\overline{\Omega}\cap U')
\end{equation*}
for some $\alpha\in(0,1)$. Therefore the boundary gradient trace 
\begin{equation*}
\nabla_y G(x;y_0) := \lim_{\Omega\ni y\to y_0}\nabla_y G(x;y)
\end{equation*}
exists. Moreover, since $G(x,\cdot)=0$ on $\partial\Omega$, its tangential derivatives vanish on the boundary. Thus the boundary gradient is purely normal: 
\begin{equation*}
\nabla_y G(x;y_0) = \partial_{\bm{n}}G(x;y_0)\,\bm n(y_0),
\end{equation*}
where $\bm n := \bm n(y)$ denotes the exterior unit normal. The quantity $-\partial_{\bm n}G(x;y_0)$ is the Poisson kernel associated with the Dirichlet Laplacian (up to the sign convention used for the Green function), i.e., 
\begin{equation*}
P(x;y_0) = -\partial_{\bm n_y}G(x;y_0), \qquad x\in\Omega,\ y_0\in\partial\Omega . 
\end{equation*}
This follows from Green's representation formula; see, for instance \cite{GilbargTrudinger1983,DallAcquaSweers2004}. 
\end{remark}

\begin{lemma}\label{lem:contW2d}
    Assume $\Omega \subset \mathbb{R}^2$ is $C^{1,1}$ and that $K: L^1(\Omega) \rightarrow L^2(E)$ is bounded. Then the weights $w_1$ and $w_2$ defined in \eqref{eq:weights_definition} are continuous, i.e., $w_1, w_2 \in C(\overline\Omega)$.
\end{lemma}

\begin{proof}
See Appendix \ref{app:lemmata}.
\end{proof}

We will now prove that \eqref{eq:alfred2} holds for smooth functions with bounded variation and use this, by a density argument, to conclude that 
\begin{equation} \label{eq:alfred2_short}
    \|Kf\|_{L^1(E)} \leq \overline{TV}_w(f) \quad \forall f \in BV(\Omega), 
\end{equation}
provided that the boundary of $\Omega$ is sufficiently regular. 

\begin{lemma}\label{lem:upper_bound_BV} Assume that \(\Omega \subset \mathbb{R}^2\) is bounded with \(C^{1,1}\)-boundary and that \(K: L^1(\Omega) \rightarrow L^2(E)\) is continuous. Then \eqref{eq:alfred2_short}, with the weights \eqref{eq:weights_definition}, holds with the boundary term of $\overline{TV}_w(f)$  interpreted in the sense of traces. 
\end{lemma} 
\begin{proof} See Appendix \ref{app:lemmata}. \end{proof}
    
\end{revblock}

\begin{remark} \label{eq:Lp_defining_weights}
    Note that in \eqref{eq:alfred2} we could have used any $L^p$-norm and employed Minkowski's integral inequality to obtain 
    \begin{eqnarray}
        \|Kf\|_{L^p(E)} &\leq& \int_\Omega \left\|K\partial_{y_1}G(\cdot;y) \right\|_{L^p(E)} |\partial_{y_1}f(y)| dy \nonumber\\ &+& 
    \int_\Omega \left\|K\partial_{y_2}G(\cdot;y) \right\|_{L^p(E)} |\partial_{y_2}f(y)| dy \nonumber\\
    &+& \int_{\partial{\Omega}} \left\|K\partial_{\mathbf{n}}G(\cdot;y) \right\|_{L^p(E)} |f(y)| dS(y). \label{eq:alfred2p}
    \end{eqnarray}
    That is, provided sufficient regularity, we have the freedom to choose the $L^p$-norm for the weights, including $p=\infty$. Of course, this leads to different weights: $w_1(y) = \left\|K\partial_{y_1}G(\cdot;y) \right\|_{L^p(E)}$, etc. \rev{In this setup, given $K: L^1(\Omega) \rightarrow L^p(E)$ bounded, it is straightforward to adapt the proof of Lemma \ref{lem:contW2d} to preserve continuity of the weights.}
\end{remark}

\begin{figure}[H]
    \centering
    \begin{tikzpicture}[scale=.8]
    \draw[thick] (0,0) rectangle (5,5);
    \node at (2.5, 0.5) {\Large \(\Omega\)};
    \draw[thick] (1,2) rectangle (3.5,4);
    \node at (2.25,3) {\Large \(R\)};
    \node[left] at (1,3) {\(\Gamma_1\)};
    \node[below] at (2.25,2) {\(\Gamma_2\)};
    \node[right] at (3.5,3) {\(\Gamma_3\)};
    \node[above] at (2.25,4) {\(\Gamma_4\)};
    \node[left] at (0,2.5) {\(\partial\Omega_1\)};
    \node[below] at (2.5,0) {\(\partial\Omega_2\)};
    \node[right] at (5,2.5) {\(\partial\Omega_3\)};
    \node[above] at (2.5,5) {\(\partial\Omega_4\)};
    \filldraw[blue] (1,3.7) circle (2pt);
    \node[blue, left] at (1,3.7) {\(y\)};
    \filldraw[blue] (3.5,2.3) circle (2pt);
    \node[blue, right] at (3.5,2.3) {\(y'\)};
    \filldraw[blue] (1.2,4) circle (2pt);
    \node[blue, above] at (1.2,4) {\(\tilde{y}\)};
    \end{tikzpicture}
    \caption{A rectangular domain $\Omega$ and the support $R \subset \Omega$, with  boundary \(\Gamma_1 \cup \Gamma_2 \cup \Gamma_3 \cup \Gamma_4\), of the true source $f^*=\chi_R$. The blue dots represent some arbitrary \(y \in \Gamma_1\) and \(y' \in \Gamma_3\).}
    \label{fig:inner-rectangle}
\end{figure}
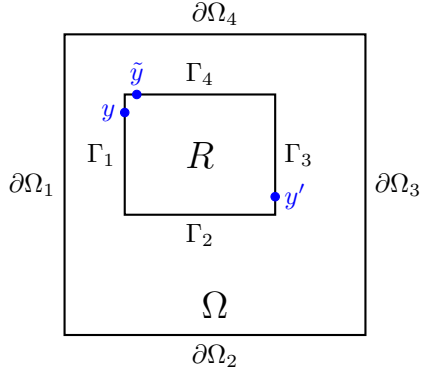

Consider the problem 
\begin{equation} \label{eq:jump_pursuit_2D}
    \min_{f \in BV(\Omega)} \overline{TV}_w(f) \quad \textnormal{subject to} \quad Kf = Kf^*. 
\end{equation} 
If the true source $f^*$ obeys  
\begin{equation} \label{eq:assume_alfred1}
    \overline{TV}_w(f^*) = \|Kf^*\|_{L^1(E)},
\end{equation}
and $f$ is any function satisfying $$Kf=Kf^*,$$
then we can employ \rev{\eqref{eq:alfred2_short}} to conclude that 
\begin{align*}
    \overline{TV}_w(f^*) &= \|Kf^*\|_{L^1(E)} \\
    &= \|Kf\|_{L^1(E)} \\
    &\leq \overline{TV}_w(f). 
\end{align*}
Consequently, $f^*$ is a solution of \eqref{eq:jump_pursuit_2D}. 

Provided that suitable assumptions are fulfilled, we will now argue that \eqref{eq:assume_alfred1} holds when $f^* = \chi_R$, where $\chi_R$ denotes the  characteristic function of a rectangle, see Figure \ref{fig:inner-rectangle}. Roughly speaking, our analysis requires disjoint images, under $K$, of the line integrals along $\Gamma_1, \Gamma_2, \Gamma_3, \Gamma_4$ of the normal derivatives of the Green's function. The second assumption \eqref{eq:propG2}, expresses that, for any fixed $z$, the sign of $K(\nabla_{\!y} G(\cdot;y) \cdot \mathbf{n}(y)) (z)$ must be the same for all $y \in \Gamma_j$. These assumptions are discussed in somewhat more detail after the proof of the following result:

\begin{theorem}\label{thm:disjointLines}
    Assume that $f^* = \chi_R$, where $R = (a,b) \times (c,d) \subset \Omega$, cf. Figure \ref{fig:inner-rectangle}. Also, we assume that the observation domain $E$ can be decomposed into mutually disjoint sets $E_j$, $\cup_{j=1}^4 E_j = E$, for which the "sources" $\partial_{y_i}G(y;x)$ associated with the boundaries $\Gamma_1$ (left), $\Gamma_2$ (bottom), $\Gamma_3$ (right), and $\Gamma_4$ (top) of $R$ obey the following: 
    \begin{itemize}
        \item For $j \in \{1,2,3,4\}$: 
    \begin{align}
        \textnormal{supp}\left(K \int_{\Gamma_j} \nabla_{\!y} G(\cdot;y) \cdot \mathbf{n}(y) d\sigma(y)\right) &\subset E_j. \label{eq:propG1} 
    \end{align}
    \item For $j \in \{1,2,3,4\}$: For each $z \in E_j$,  
    \begin{align}
        \textnormal{sgn}[K(\nabla_{\!y} G(\cdot;y) \cdot \mathbf{n}(y)) (z)] &= c_j \quad \forall y \in \Gamma_j, \label{eq:propG2}
    \end{align}
    where $c_j \in \{-1,0,1\}$ is a constant. 
    \end{itemize}
    Then, 
    \begin{equation*}
        f^* \in \argmin_{f \in BV(\Omega)} \overline{TV}_w(f) \quad \textnormal{subject to} \quad Kf = K\rev{f^*},
    \end{equation*}
    \rev{provided that \(K: L^1(\Omega) \rightarrow L^2(E)\) is bounded and that $\Omega$ is such that \eqref{eq:alfred2_short} holds.} 
\end{theorem}
\begin{proof}
    First, since $f^* = \chi_R$, we note that
    \begin{eqnarray*}
        \overline{TV}_w(f^*) &=& \int_\Omega w_1 |\partial_{x_1}f^*| + \int_\Omega w_2 |\partial_{x_2}f^*| + \int_{\partial\Omega} w_n|f^*| \\ 
        &=& \int_{\Gamma_1} w_1 + \int_{\Gamma_2} w_2 + \int_{\Gamma_3} w_1 + \int_{\Gamma_4} w_2.
    \end{eqnarray*}
    Secondly, we can express $f^* = \chi_R$ in the following form  
    \begin{eqnarray*}
        \chi_R(x) &=& {\color{revcol}\langle \delta_x, \chi_R \rangle} 
        =-\int_\Omega \chi_R(y)\Delta G(y;x)dy 
       = -\int_R \Delta G(y;x)dy \\ &=&
         -\int_{\partial R} \nabla_{\!y} G(y;x) \cdot \mathbf{n}(y) d\sigma(y),
         = -\int_{\partial R} \nabla_{\!y} G(x;y) \cdot \mathbf{n}(y) d\sigma(y).
    \end{eqnarray*}
    Thus, if we evaluate $Kf^*$ at a point $z \in E$ using the above representation of $f^*$, we get
    \begin{align*}
        (Kf^*)(z) &= - \left(K\int_{\partial R} \nabla_{\!y}G(\cdot;y)\cdot \mathbf{n}(y)d\sigma(y)\right)(z) \\ 
        &= - \left(K \sum_i \int_{\Gamma_i} \nabla_{\!y}G(\cdot;y)\cdot \mathbf{n}(y)d\sigma(y)\right)(z).
    \end{align*}
    Consequently, we obtain, using \eqref{eq:propG1} and \eqref{eq:propG2} in the second and fifth equality below, respectively, 
    \begin{eqnarray*}
        \|Kf^*\|_{L^1(E)} &=& 
        \int_E \left| \sum_i K \int_{\Gamma_i} (\nabla_{\!y} G(\cdot;y)\cdot \mathbf{n}(y))d\sigma(y)\right| dz \\ &=&
        \int_E \sum_i \left|\int_{\Gamma_i}K(\nabla_{\!y} G(\cdot;y)\cdot \mathbf{n}(y))d\sigma(y)\right| dz \\ &=&
        \sum_j \int_{E_j} \sum_i \left|\int_{\Gamma_i}K(\nabla_{\!y} G(\cdot;y)\cdot \mathbf{n}(y))d\sigma(y)\right| dz \\ &=&
        \int_{E_1}\left|-\int_{\Gamma_1} K\partial_{y_1}G(\cdot;y)d\sigma(y) \right|dz + \int_{E_2}\left|-\int_{\Gamma_2} K\partial_{y_2}G(\cdot;y)d\sigma(y) \right| dz \\ &+& \int_{E_3}\left|\int_{\Gamma_3} K\partial_{y_1}G(\cdot;y)d\sigma(y) \right|dz + \int_{E_4}\left|\int_{\Gamma_4} K\partial_{y_2}G(\cdot;y)d\sigma(y) \right|dz \\ &=&
        \int_{\Gamma_1}\int_{E_1}\left| K\partial_{y_1}G(\cdot;y)\right|dz d\sigma(y) + \int_{\Gamma_2}\int_{E_2}\left| K\partial_{y_2}G(\cdot;y)\right|dz d\sigma(y) \\ &+& \int_{\Gamma_3}\int_{E_3}\left| K\partial_{y_1}G(\cdot;y)\right|dz d\sigma(y) + \int_{\Gamma_4}\int_{E_4}\left| K\partial_{y_2}G(\cdot;y)\right|dz d\sigma(y)  \\ &=&
        \int_{\Gamma_1} w_1 + \int_{\Gamma_2} w_2 + \int_{\Gamma_3} w_1 + \int_{\Gamma_4} w_2 = \overline{TV}_w(f^*).
    \end{eqnarray*}
    The result now follows from \rev{\eqref{eq:alfred2_short}}.
\end{proof}

Note that Theorem \ref{thm:disjointLines} holds whenever $\Omega$, $R$ and $K$ are such that \eqref{eq:propG1} and \eqref{eq:propG2} hold for a finite number of sets $\{ E_j \}$ and $\{ \Gamma_j \}$. That is, the argument is not restricted to the rectangle-inside-rectangle setup, which we decided to consider for an easy exposition. Nevertheless, $f^*$ must be the  characteristic function of the subdomain $R$. 

\begin{revblock}
\begin{remark} \label{rem:strength_of_conditions}
Before turning to the interpretation of \eqref{eq:propG1} and \eqref{eq:propG2}, we want to be explicit about their strength. These conditions are restrictive, and we do not claim that they hold \emph{exactly} for the elliptic forward operator employed in Section \ref{sec:numerical_experiments}: \eqref{eq:propG1} requires the images under $K$ of the four line sources to have disjoint supports, which for an operator with a strictly positive kernel is typically not the case. Theorem \ref{thm:disjointLines} should therefore be read as an idealized limit statement, whose purpose is to isolate the mechanism that makes the weighted functional recover a rectangle exactly: the weights are constructed so that the upper bound \eqref{eq:alfred2} is attained by $f^*$, and \eqref{eq:propG1}--\eqref{eq:propG2} are conditions under which no cancellation occurs in that bound. Moreover, these conditions can be satisfied in an approximate sense for  rectangles that are not too small. The numerical experiments in Section \ref{sec:numerical_experiments} show precisely the behaviour predicted by this discussion: the reconstruction is very good for large rectangles, and degrades exactly in the regime where the conditions are strongly violated (small or thin rectangles, sources far from the observation boundary).
\end{remark}
\end{revblock}

We will now motivate why assumptions \eqref{eq:propG1} and \eqref{eq:propG2} are (approximately) satisfied when $K$ exhibits a strong diffusive effect, provided that we consider a rectangle-inside-rectangle scenario; see Figure \ref{fig:inner-rectangle}. To this end, consider the specific case where $E = \partial\Omega$, that is, $K: BV(\Omega) \rightarrow L^1(\partial\Omega)$. \rev{By a \emph{strong diffusive (or smoothing) property} of $K$ we mean, loosely speaking, the following localization behaviour: the image $K\psi$ of a source $\psi$ concentrated near a point $y \in \Omega$ has most of its mass on the part of $\partial\Omega$ closest to $y$, and decays rapidly with the distance along the boundary from that point.} 

In this setting, let us examine the two Green functions that have singularities at the points $y$ and $y'$, respectively; cf. Figure \ref{fig:inner-rectangle}. 
The normal derivatives of these Green's functions resemble horizontally oriented dipoles. Due to the strong diffusive property of $K$, the image of $\partial_{y_1}G(\cdot,y)$, $y \in \Gamma_1$, under $K$ has most of its relative magnitude concentrated along $\partial\Omega_1$, while the image of $\partial_{y_2}G(\cdot,y')$, $y \in \Gamma_3$, under $K$ is relatively concentrated along $\partial\Omega_3$. A similar property hold for Green's functions associated with points located along $\Gamma_2$ and $\Gamma_4$.  

When comparing sources located along vertical and horizontal line segments — for example $y \in \Gamma_1$ and $\tilde{y} \in \Gamma_4$, we consider the images of $\partial_{y_1}G(\cdot,y)$ and $\partial_{y_2}G(\cdot,\tilde{y})$ under $K$. These correspond to images of horizontally and vertically oriented dipole structures, which implies that these images are (relatively) concentrated on $\partial\Omega_1$ and $\partial\Omega_4$, respectively. 

Based on the discussion presented in the last two paragraphs, we conclude that (roughly) \eqref{eq:propG1} and \eqref{eq:propG2} are satisfied with $E_j=\partial \Omega_j$, $j=1,2,3,4$. 
We also remark that when the rectangle is small or thin, then $y \in \Gamma_1$ and $y' \in \Gamma_3$ are close to each other, making the assumption of disjoint images under $K$ of the partial derivatives of the involved Green's function more dubious. In the numerical experiments section, we will see this effect quite clearly: small rectangles are overestimated (in size) in the recovery process.

\section{Hybrid regularization method} \label{sec:hybrid} 
We will see in the numerical experiments section that weighted TV-regularization has a tendency to overestimate the size of small sources. This suggests that it might be beneficial to combine weighted TV-regularization with sparsity regularization, incorporated in terms of a measure theoretical framework in the infinite dimensional setting:    
\begin{equation} \label{def:hybrid_cost_func}
    \min_{f\in BV(\Omega)} \left\{ \frac{1}{2}\|K \mu - d\|^2_{L^2(E)} + \alpha TV_w(f) + \beta \| \mu \|_{\mathcal{M}_{\ws} (\Omega)} \right\},
\end{equation}
where $\mu$ is the measure associated with $f$, see \eqref{eq:BVfunction_yields_measure} below. In this section we thus assume that $K$ operates on measures instead of functions, which is common when one considers sparsity regularization in an infinite dimensional setting\footnote{\rev{Sparsity regularization is typically formulated via an $\ell^1$-penalty in finite dimensions. In infinite-dimensional settings, however, the canonical sparse objects are atomic sources $\delta_x$, which are Radon measures rather than elements of $L^1(\Omega)$.}}. 
We now define the additional notation and the weight function $\ws$ to be employed in the sparsity term. 

Let $\mathcal{M}(\Omega)$ be the space of Radon measures and assume that we have a bounded linear operator $K:\mathcal{M}(\Omega) \rightarrow L^2(E)$ which admits the kernel representation
\begin{equation} \label{eq:kernel_representation2}
    K\mu = \int_\Omega k(x,\cdot)d\mu(x).
\end{equation}
\rev{Here, $k: \Omega \times E \rightarrow \mathbb{R}$ is measurable, $E$ denotes the observation domain and $k$ is assumed to be such that the mapping 
\begin{equation} \label{eq:kernel_condition}
\Omega \ni x \mapsto k(x,\cdot) \in L^2(E)
\end{equation} 
is continuous and obeying
\begin{equation} \label{eq:kernel_condition_2}
    \sup_{x \in \Omega} \| k(x,\cdot) \|_{L^2(E)} < \infty.
\end{equation}
These two conditions on $k$ imply that $K$ is bounded.
We now} define the weight function $\ws: \Omega\rightarrow \mathbb{R}$ by  
\begin{equation} \label{def:weight_operator_sparsity}
    \ws(x) = \|K\delta_x\|_{L^2(E)},
\end{equation}
where $\delta_x$ is the Dirac measure associated with the point $x \in \Omega$ and we assume that there \rev{exists a constant $\ws_{\min}$  such that $$0 < \ws_{\min} \leq \ws(x) \leq \ws_{\max} < \infty \quad \forall x \in \Omega.$$ The existence of $\ws_{\max}$ is assured by \eqref{eq:kernel_condition_2}.}

This also enables us to introduce the weighted Radon space $\mathcal{M}_{\ws}(\Omega)$ equipped with the norm
\begin{equation*}
    \|\mu\|_{\mathcal{M}_{\ws}(\Omega)} = \int_\Omega \ws(x) \, d|\mu|,
\end{equation*}
where $|\mu|$ is the total variation measure of $\mu$. \rev{Here, the integral on the right-hand-side is well-defined because the continuity of the mapping \eqref{eq:kernel_condition} implies that $\ws$ is continuous.}

Let $R \subset \Omega$ be a set of finite perimeter. Since $\chi_{R} \in L^1_{loc}(\Omega)$, it follows  that 
\begin{equation} \label{eq:radonnikodym}
    \mu_R(A) = \int_A \chi_{R}(x)dx,
\end{equation}
where $A \subset \Omega$ is any Lebesgue measurable set, defines a Radon measure $\mu_R$. 
%
%
Let $\mu_R$ constitute the true source, for which we use the notation;   
\begin{equation} \label{eq:positive_true_measure}
\mu^* = \mu_R, 
\end{equation}
which consistent with the symbolism used in the previous sections. 
Also, assume that the following relation between the subdomain $R$ and the forward operator holds: There exist a continuous function $\tau: \Omega \rightarrow \mathbb{R}$ and a point $\bar{x} \in R$ such that
\begin{equation} \label{eq:parallel_true_measure}
    K \delta_z = \tau(z) K \delta_{\bar{x}}, \quad \tau(z)>0, \quad \; \forall z \in R. 
\end{equation}
That is, we assume that the images under $K$ of the individual Dirac measures associated with the points in $R$ are parallel. (Typically, $\tau(z) \approx 1$, reflecting the smoothing property of the forward mapping). Roughly speaking, it is plausible that \eqref{eq:parallel_true_measure} approximately holds provided that the extent of the region $R$ is not too large compared with $R$'s distance to the boundary $\partial \Omega$, provided that $E=\partial \Omega$. Under these circumstances, we will now prove that the region $R$ can be recovered, in a measure theoretical sense, by solving a sparsity basis pursuit problem. Thereafter, we employ this result to analyze the basis pursuit counterpart to \eqref{def:hybrid_cost_func} when $\alpha \rightarrow 0$ for a fixed $\beta > 0$    

\begin{revblock}
\begin{remark} \label{rem:parallel_idealization}
Condition \eqref{eq:parallel_true_measure} is an idealization, as are \eqref{eq:propG1}--\eqref{eq:propG2}, and we are not aware of a nontrivial forward operator of the type considered here for which it holds \emph{exactly}: collinearity of $\{K\delta_z\}_{z \in R}$ means that $K$, restricted to measures supported in $R$, has rank one. It is, however, a good approximation in the regime of interest, namely when $R$ is small compared with its distance to the observation boundary $E=\partial\Omega$, because $z \mapsto K\delta_z = k(z,\cdot) \in L^2(E)$ then varies ``slowly'' for $z \in R$.
\end{remark}
\end{revblock}

\begin{proposition} \label{prop:f_star_solves_sparsity}
    \rev{Assume that $K:\mathcal{M}(\Omega) \rightarrow L^2(E)$ admits the kernel representation \eqref{eq:kernel_representation2} and that $k$ is such that the mapping \eqref{eq:kernel_condition} is continuous and that \eqref{eq:kernel_condition_2} is satisfied.} 
    Let $\mu^* = \mu_R$ and assume that \eqref{eq:parallel_true_measure} holds. Then 
    \begin{equation} \label{eq:f_star_solves_sparsity}
        \mu^* \in \argmin_{\mu \in \mathcal{M}(\Omega)} \| \mu \|_{\mathcal{M}_{\ws} (\Omega)} \quad \textnormal{subject to} \quad K\mu = K\mu^*, 
    \end{equation}
    where $\ws: \Omega \rightarrow \mathbb{R}$ is defined in \eqref{def:weight_operator_sparsity}. 
\end{proposition}

\begin{proof}
    We first observe that the true source $\mu^* = \mu_R$ satisfies 
    \begin{align*}
        \| \mu^* \|_{\mathcal{M}_{\ws}(\Omega)} &= \int_\Omega \ws(x)d\mu_R(x) 
        = \int_R \ws(x)dx \\
        &= \int_R ||K\delta_x\|_{L^2(E)} dx \\
        &= \int_R \tau(x) dx~\|K\delta_{\bar{x}}\|_{L^2(E)},
    \end{align*}
where we used \eqref{eq:parallel_true_measure} for the last line,   and
    \begin{align*}
        \| K \mu^* \|_{L^2(E)} &= \left\| \int_\Omega k(x,\cdot)d\mu_R(x) \right\|_{L^2(E)}  
        = \left\| \int_R k(x,\cdot)dx\right\|_{L^2(E)} \\
        &= \left\| \int_R \int_\Omega k(z,\cdot)d\delta_x(z) dx\right\|_{L^2(E)} \\
        &= \left\| \int_R K\delta_x dx\right\|_{L^2(E)} \\
       &= \left\| \int_R \tau(x) dx ~K\delta_{\bar{x}}\right\|_{L^2(E)} \\
        &= \int_R \tau(x) dx~\|K\delta_{\bar{x}}\|_{L^2(E)}.
    \end{align*}
    Therefore, for any $\mu \in \mathcal{M}(\Omega)$ satisfying $K \mu = K \mu^*$,
    \begin{align*}
        \| \mu^* \|_{\mathcal{M}_{\ws}(\Omega)} &
        = \|K\mu^*\|_{L^2(E)}  
        = \| K\mu \|_{L^2(E)}  \\
        &= \left\| \int_\Omega k(x,\cdot) d\mu(x) \right\|_{L^2(E)} \\
        &\leq \int_\Omega \|K\delta_x\|_{L^2(E)} |d\mu(x)| \\
        &= \int_\Omega \ws(x) |d\mu(x)| \\
        &=\| \mu \|_{\mathcal{M}_{\ws}(\Omega)},  
    \end{align*}
    where we have used the relation $K\delta_x = k(x,\cdot)$ and also Minkowski's integral inequality. This completes the proof.  
\end{proof}

\begin{remark}
    If \eqref{eq:parallel_true_measure} holds, then it follows in a straightforward manner that $\hat{\mu} = \int_R \tau(x) \, dx \, \delta_{\bar{x}}$ also solves the minimization problem \eqref{eq:f_star_solves_sparsity}, i.e., $K \hat(\mu) = K \mu^*$ and $\| \hat{\mu} \|_{\mathcal{M}_{\ws}(\Omega)} = \| \mu^* \|_{\mathcal{M}_{\ws}(\Omega)}$. Hence, if one wants to identify a "blocky" structure while preventing underestimation of the size of small sources, it might be beneficial to combine sparsity methodologies with TV-regularization, which we now investigate further.    
\end{remark}

It is well known that any $f \in BV(\Omega) \subset L^1(\Omega)$ introduces a measure by
\begin{equation} \label{eq:BVfunction_yields_measure}
    \mu(A) = \int_A f(x)dx,
\end{equation}
where $A \subset \Omega$ is any Lebesgue measurable set, cf. \eqref{eq:radonnikodym} that expresses this connection between the true measure $\mu^*=\mu_R$ and $f^* = \chi_R$. ($f$ and $f^*$ are in this context Radon–Nikodym derivatives.) In the sequel, we will use the set 
\begin{equation*}
    \Mu= \left\{ \mu \in \mathcal{M}(\Omega) \, \left| \, f = \frac{d \mu}{dx} \right. \in BV(\Omega) \right\},
\end{equation*}
and, when suitable, write $f \anti$ for $\mu$, i.e., $$f \anti = \mu.$$ 
We also note that, in the following two propositions, we rely on the notation to distinguish between a measure $\mu_\alpha$ (or the optimal $\mu^*$) and the associated BV-function $f_\alpha$ (or the optimal $f^*$). 

Let us verify that including weighted TV-regularization reduces the weighted TV-semi-norm of the solution:  
\begin{lemma} \label{prop:f_gamma_bounds}
    \rev{Suppose $K:\mathcal{M}(\Omega) \rightarrow L^2(E)$ admits the kernel representation \eqref{eq:kernel_representation2} and that $k$ is such that the mapping \eqref{eq:kernel_condition} is continuous and that \eqref{eq:kernel_condition_2} is satisfied.}
    Assume that $\mu^* = \mu_R$, that \eqref{eq:parallel_true_measure} holds and let
\begin{equation} \label{def:f_gamma}
    \mu_\alpha \in \argmin_{\mu \in \Mu} \left\{ \alpha TV_w(f)+\| \mu \|_{\mathcal{M}_{\ws} (\Omega)} \right\} \quad \textnormal{subject to} \quad K \mu = K \mu^*,
\end{equation} 
where $f$ is the BV-function associated with $\mu$. 
Then, for $\alpha \geq 0$,  
\begin{equation*}
    TV_w(f_\alpha) \leq TV_w(f^*)
\end{equation*}
and 
\begin{equation*}
     0 \leq \| \mu_\alpha \|_{\mathcal{M}_{\ws} (\Omega)} - \| \mu^* \|_{\mathcal{M}_{\ws} (\Omega)} 
    \leq \alpha TV_w(f^*). 
\end{equation*}
Here, $f_\alpha$ and $f^*$ denote the Radon–Nikodym derivatives of $\mu_\alpha$ and $\mu^*$, respectively. 
\end{lemma}
\begin{proof}
It follows from \eqref{eq:f_star_solves_sparsity} and \eqref{def:f_gamma} that 
\begin{equation*}
    \alpha TV_w(f_\alpha)+\| \mu^* \|_{\mathcal{M}_{\ws} (\Omega)} \leq \alpha TV_w(f_\alpha)+\| \mu_\alpha \|_{\mathcal{M}_{\ws} (\Omega)} \leq \alpha TV_w(f^*)+\| \mu^* \|_{\mathcal{M}_{\ws} (\Omega)},  
\end{equation*}
which implies that 
\begin{equation*}
    TV_w(f_\alpha) \leq TV_w(f^*) 
\end{equation*}
and that 
\begin{equation*}
     0 \leq \| \mu_\alpha \|_{\mathcal{M}_{\ws} (\Omega)} - \| \mu^* \|_{\mathcal{M}_{\ws} (\Omega)} 
    \leq \alpha TV_w(f^*) - \alpha TV_w(f_\alpha), 
\end{equation*}    
which completes the proof. 
\end{proof}

We next explain that, if \eqref{eq:f_star_solves_sparsity} has several solutions, one might want to add weighted TV-regularization in order to choose/compute a "blocky" solution. 
\begin{proposition} \label{prop:f_gamma_converges} 
    \rev{Assume that $K:\mathcal{M}(\Omega) \rightarrow L^2(E)$ admits the kernel representation \eqref{eq:kernel_representation2} and that $k$ is such that the mapping \eqref{eq:kernel_condition} is continuous and that \eqref{eq:kernel_condition_2} is satisfied.} 
    Let $\mu_\alpha$, with Radon-Nikodym derivative $f_\alpha$, be as defined in \eqref{def:f_gamma}. 
    Suppose the true source $\mu^*$ satisfies \eqref{eq:positive_true_measure} and \eqref{eq:parallel_true_measure}, and that its associated Radon-Nikodym derivative $f^*$ is the unique solution to the optimization problem 
    \begin{equation} \label{eq:unique_TV_min_meas_solution}
        f^* = \argmin_{f \in S} TV_w(f), 
    \end{equation}
    where 
    \begin{equation*}
    S = \argmin_{f \in BV(\Omega)} \| f \anti \|_{\mathcal{M}_{\ws} (\Omega)} \quad \textnormal{subject to} \quad K\! f \anti  = K\! f^* \anti. 
\end{equation*}
Then 
\begin{equation*}
    \lim_{\alpha \rightarrow 0} f_\alpha = f^*, \quad \textnormal{in } L^1(\Omega). 
\end{equation*}
\end{proposition}
\begin{proof}
    It follows from \eqref{def:f_gamma} that
    \begin{equation*}
        \|f_\alpha\|_{BV(\Omega)} = \|f_\alpha\|_{L^1{(\Omega})} + TV_w(f_\alpha) \leq \frac{1}{\ws_{\min}}\|\mu^*\|_{\mathcal{M}_{\ws} (\Omega)} + TV_w(f^*),
    \end{equation*}
    which implies that $\{ f_\alpha \}_{\alpha > 0} \subset BV(\Omega)$ is uniformly bounded. Consequently, Rellich’s compactness theorem implies that we can extract a sequence $(f_{\alpha_k})$, $\alpha_k \rightarrow 0$, such that
    \begin{equation*}
        f_{\alpha_k} \rightarrow \bar{f} \quad \textnormal{in } L^1(\Omega),
    \end{equation*}
    for some $\bar{f} \in BV(\Omega)$ with associated measure $\bar{\mu} = \bar{f} \anti$. 

    
    It follows from \eqref{def:f_gamma} that $K \mu_\alpha = K \mu^*$ and hence $K \mu_{\alpha_k} = K \mu^*$ for all $k$.  
    Furthermore, due to the continuity of the forward operator $K: \mathcal{M}(\Omega) \rightarrow L^p(E)$, we find that 
    \begin{equation} \label{eq:feasible1}
        K\! \bar{f} \anti = \lim_{k\rightarrow \infty} K\! f_{\alpha_k} \anti =\lim_{k\rightarrow \infty} K \mu_{\alpha_k} = K \mu^* = K\! f^* \anti. 
    \end{equation}

    Next, Lemma \ref{prop:f_gamma_bounds} yields that
    \begin{equation*}
     0 \leq \| \mu_{\alpha_k} \|_{\mathcal{M}_{\ws} (\Omega)} - \| \mu^* \|_{\mathcal{M}_{\ws} (\Omega)} 
    \leq \alpha_k TV_w(f^*) 
    \end{equation*}
    or 
    \begin{equation*}
     0 \leq \| f_{\alpha_k} \anti \|_{\mathcal{M}_{\ws} (\Omega)} - \| f^* \anti \|_{\mathcal{M}_{\ws} (\Omega)} 
    \leq \alpha_k TV_w(f^*),  
    \end{equation*}
    i.e., since $\alpha_k \rightarrow 0$, 
    \begin{equation*} 
        \| \bar{f} \anti \|_{\mathcal{M}_{\ws}} = \| f^* \anti \|_{\mathcal{M}_{\ws}}, 
    \end{equation*}
    which combined with \eqref{eq:feasible1} show that $\bar{f}$ is a feasible solution of \eqref{eq:unique_TV_min_meas_solution}. 
    
    To show that $\bar{f}$ is optimal, we combine Lemma \ref{prop:f_gamma_bounds} with the fact that the total variation functional is lower semi-continuous with respect to $L^1$  to obtain the inequalities
    \begin{equation*}
        TV_w(f^*) \geq\liminf_{k} TV_w(f_{\alpha_k}) \geq TV_w(\bar{f}).
    \end{equation*}
    It follows that $\bar{f}$ is a solution of \eqref{eq:unique_TV_min_meas_solution}, and we can conclude that $\bar{f} = f^*$ because $f^*$ is assumed to be the unique minimizer of \eqref{eq:unique_TV_min_meas_solution}.

    This argument also shows that \textit{any} convergent sequence $(f_{\alpha_k}) \subset \{ f_\alpha \}_{\alpha > 0}$, $\alpha_k \rightarrow 0$, must converge to $f^*$ in $L^1(\Omega)$, which, due to the fact that $\{ f_\alpha \}_{\alpha > 0}$ is relative compact in $L^1(\Omega)$, will imply that $f_\alpha \rightarrow f^*$, as $\alpha \rightarrow 0$, in $L^1(\Omega)$. 
\end{proof}

\begin{remark}
    We remark that Lemma \ref{prop:f_gamma_bounds} and Proposition \ref{prop:f_gamma_converges} also are valid if one replaces $TV_w$ with standard total variation, denoted $TV_I$ below. The proofs are identical, replacing $TV_w$ with $TV_I$. 
\end{remark}

\section{Numerical experiments} \label{sec:numerical_experiments}
We will perform numerical experiments on the unit square $\Omega = (0,1)\times(0,1)$ with the forward operator $K:L^2(\Omega) \rightarrow L^2(\partial\Omega)$, $f \mapsto u|_{\partial\Omega}$, where $u \in H^1(\Omega)$ solves
\begin{eqnarray}
    -\nabla \cdot \rev{\kappa}\nabla u + u &=& f, \quad x \in \Omega, \label{eq:ellipticPDE} \\
    \rev{\kappa}\nabla u \cdot \mathbf{n} &=&0, \quad x \in \partial\Omega. \label{eq:neumannBC}
\end{eqnarray}
Unless otherwise stated, the conductivity is assumed to be isotropic, i.e., $\rev{\kappa} = 1$.

In all the experiments, the exact data was synthetically generated by applying $K$ to a true source $f^*$, yielding $$d^\dagger := Kf^*.$$ The data was then corrupted with additive white Gaussian noise to produce the noisy observations $$ d = d^\dagger + \epsilon,$$ where the noise $\epsilon$ was scaled to ensure a relative noise level of 1\%, i.e.,
\begin{equation*}
    \frac{\|\epsilon\|}{\|d\| } = 0.01.
\end{equation*}

\rev{All experiments were carried out on a $101\times101$ finite-difference grid, and every reported error metric is averaged over three independent noise realizations.}
When solving the discretized problem, we employed the iterative Bregman method presented in \cite{osher05}: 
\begin{align}
\label{eq:discrete_variational}
    \xx_{k+1} &= \argmin_\xx \left\{\frac{1}{2}\|\AAA\xx - \bb_k\|^2 + \alpha TV_w(\xx) + \beta \|\WW\xx\|_1\right\}, \\
    \nonumber
    \bb_{k+1} &= \bb_k + (\mathbf{d} - \AAA\xx_k),
\end{align}
where $\xx$, $\mathbf{d}$ and $\AAA$ represent discrete counterparts to the source $f$, the noisy data $d$ and the forward operator $K$, respectively. The minimization problem in the first step above was solved with the Alternating Direction Method of Multipliers (ADMM) algorithm. \rev{The regularization parameters were selected by the Morozov discrepancy principle. Starting from a common initial value, $\alpha$ was successively reduced as long as the residual exceeded three times the noise level, $\|\AAA\xx - \mathbf{d}\| > 3\delta$ with $\delta = \|\epsilon\|$, indicating over-regularization; once the residual dropped below this threshold, $\alpha$ was held fixed and a small number of Bregman updates \eqref{eq:discrete_variational} were applied to restore contrast, terminating when the discrepancy $\|\AAA\xx - \mathbf{d}\| \le \delta$ was reached. For the pure total variation experiments ($\beta = 0$) this fixes $\alpha$ directly, whereas for the hybrid experiments ($\beta > 0$) we kept the ratio $\beta/\alpha = 10^{3}$ fixed across all examples, so that the discrepancy principle determines only the overall scale. The resulting values are reported in Table~\ref{tab:alpha}.} 

\begin{remark}
    In the discrete case, we use the $\| \cdot \|_1$-norm to include sparsity regularization and thus not the measure theoretical approach employed in the infinite dimensional setting considered in Section \ref{sec:hybrid}, cf. \eqref{def:hybrid_cost_func}. Nevertheless, results analogous to Proposition \ref{prop:f_star_solves_sparsity}, Lemma \ref{prop:f_gamma_bounds} and Proposition \ref{prop:f_gamma_converges} can be proven for \eqref{eq:discrete_variational}; using very similar arguments to those presented above. Here, $\WW$ is a diagonal matrix 
    $$
    \begin{bmatrix}
        \tilde{w}_1 & & \\ 
        & \ddots & \\
        & & \tilde{w}_n
    \end{bmatrix}
    $$
    with entries 
    \begin{equation*}
        \tilde{w}_i = \| \AAA \ee_i \|_2, \quad i=1,2,\ldots,n,
    \end{equation*}
    or 
    \begin{equation} \label{eq:discrete_weights_projection}
        \tilde{w}_i = \| \AAA_q^\dagger \AAA \ee_i \|_2,  \quad i=1,2,\ldots,n,
    \end{equation}
    where $\{ \ee_i \}$ are the standard Euclidean basis vectors and $\AAA_q^\dagger$ denotes an approximation of the pseudo inverse of $\AAA$, defined in terms of the singular vectors associated with the $q$ larges singular values of $\AAA$.  In our numerical experiments, we employed the form \eqref{eq:discrete_weights_projection} with $q=120$, \rev{taken relatively large so that $\AAA_q^\dagger$ closely approximates the full pseudo inverse $\AAA^\dagger$. We did not investigate the sensitivity of the performance of the hybrid method with respect to this parameter.}
\end{remark}

As mentioned in Remark \ref{eq:Lp_defining_weights}, one may define the TV-weights in terms of any $L^p$-norm. Figure \ref{fig:TVweights} shows the weights computed with $p=1$ and $p=\infty$, and Figure \ref{fig:compare} exemplifies that the choice $p=\infty$ yields somewhat better results than using $p=1$, especially for small true sources. We therefore below present results computed with $p=\infty$, even though Theorem \ref{thm:disjointLines} presupposes $p=1$. The authors regard it as an open problem to formulate conditions similar to \eqref{eq:propG1} and \eqref{eq:propG2} such that Theorem \ref{thm:disjointLines} also holds for $p=\infty$.    

\begin{figure}[H]
    \centering
    \begin{subfigure}[t]{0.35\linewidth}        
        \centering
        \includegraphics[trim=60 30 60 30, clip, width=\linewidth]{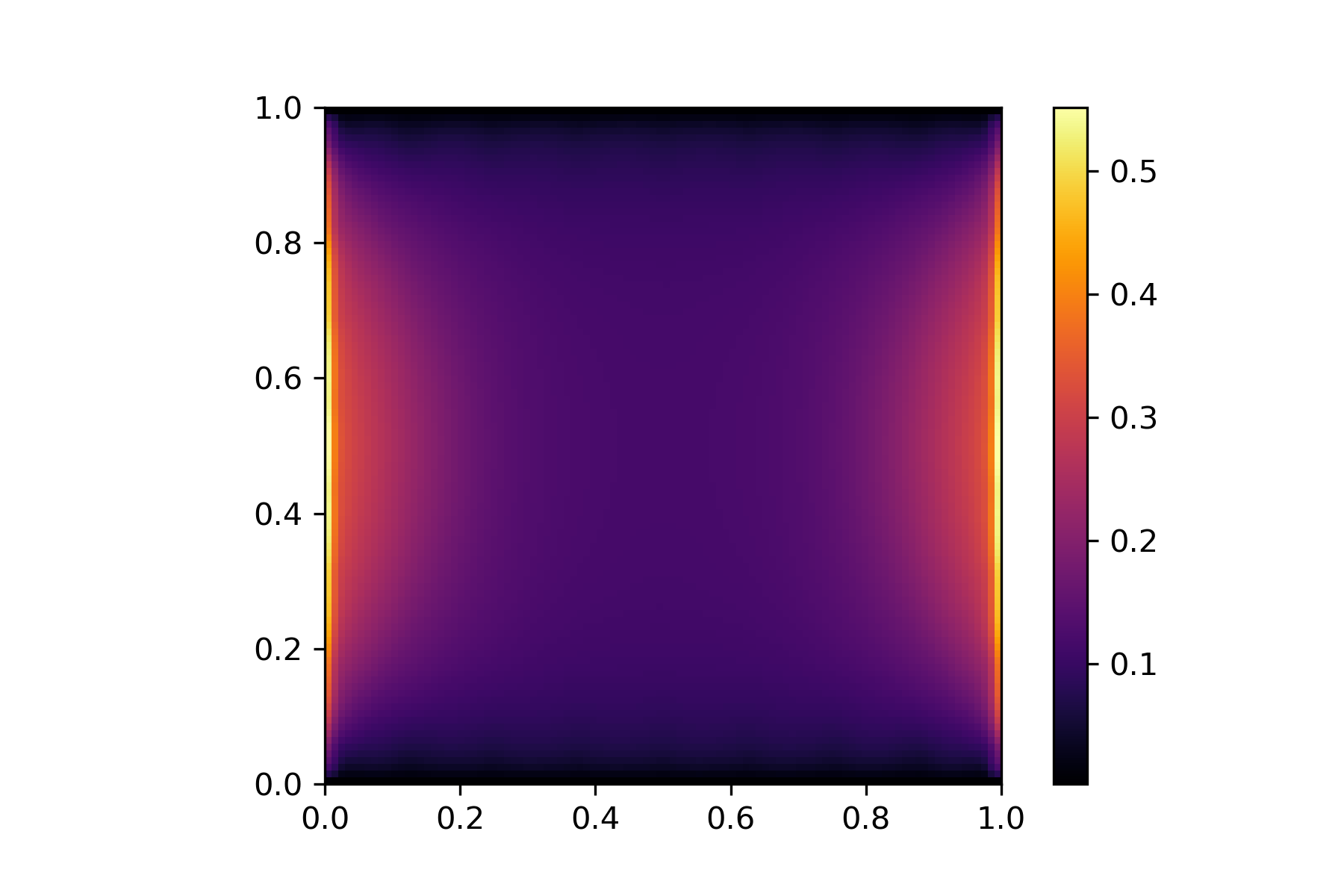}
        \caption{$i = 1, p = 1$}
    \end{subfigure}
    \begin{subfigure}[t]{0.35\linewidth}        
        \centering
        \includegraphics[trim=60 30 60 30, clip, width=\linewidth]{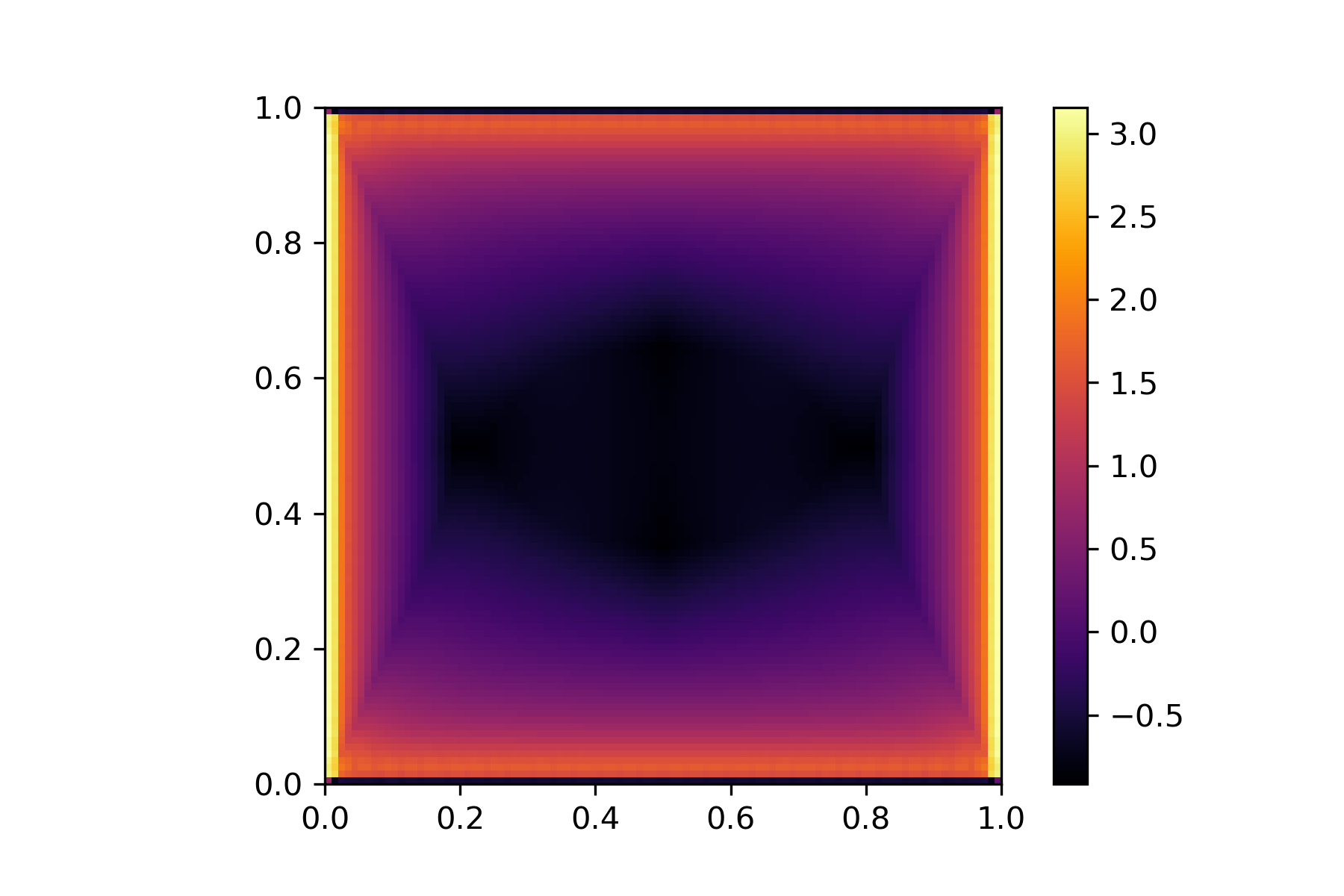}
        \caption{$i = 1, p = \infty$ (log plot)}
    \end{subfigure}\par
    \begin{subfigure}[t]{0.35\linewidth}        
        \centering
        \includegraphics[trim=60 30 60 30, clip, width=\linewidth]{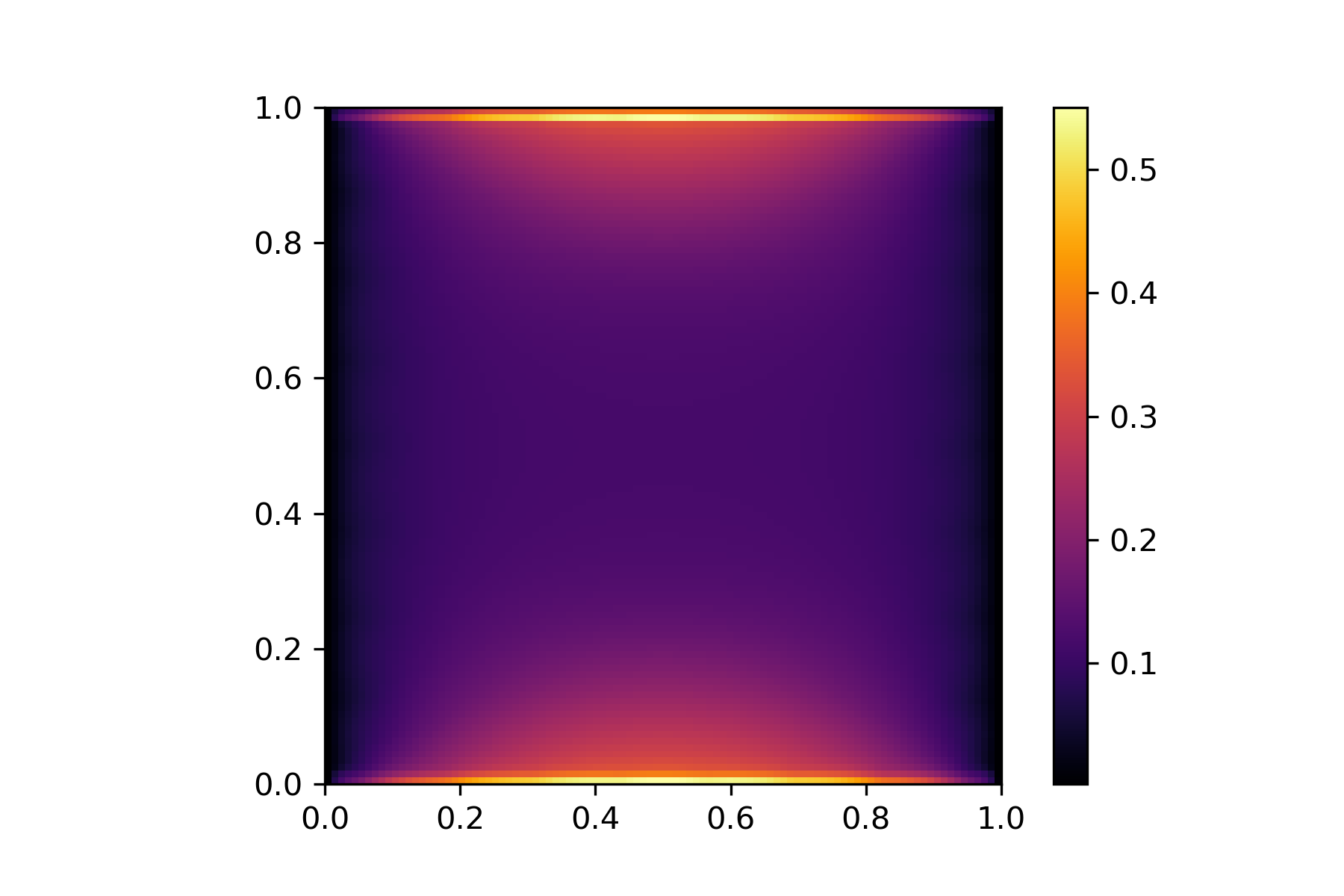}
        \caption{$i = 2, p = 1$}
    \end{subfigure}
    \begin{subfigure}[t]{0.35\linewidth}        
        \centering
        \includegraphics[trim=60 30 60 30, clip, width=\linewidth]{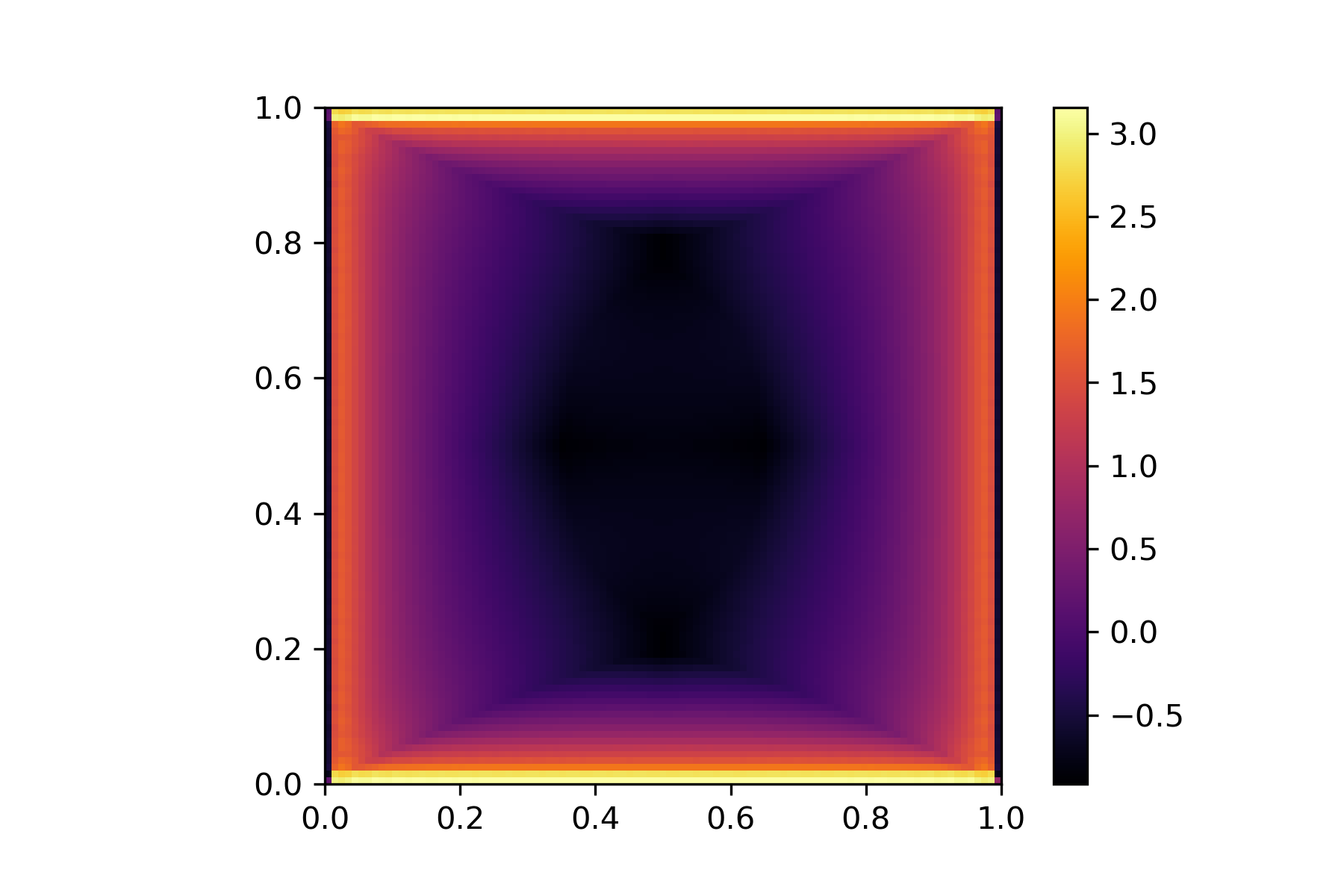}
        \caption{$i = 2, p = \infty$ (log plot)}
    \end{subfigure}\par
    \caption{Total variation weights $\left\|K\partial_{y_i}G(\cdot;y) \right\|_{L^p(E)}$ for $i= 1, 2$ and different values of $p$, cf. \eqref{eq:alfred2p}.}
    \label{fig:TVweights}
\end{figure}

\begin{figure}[H]
    \centering
    \begin{subfigure}[t]{0.32\linewidth}        
        \centering
        \includegraphics[trim=60 30 60 30, clip, width=\linewidth]{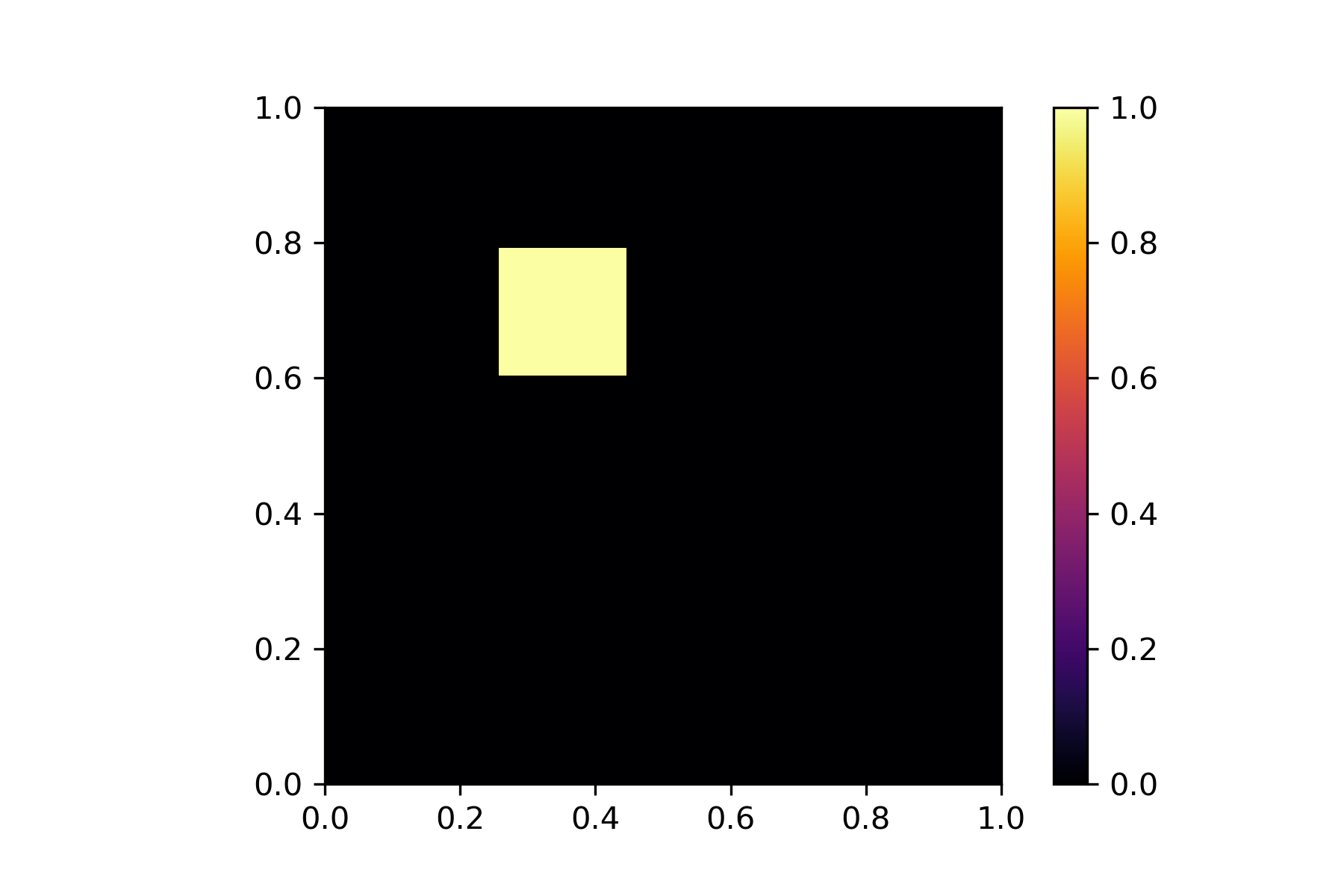}
        \caption{True source}
    \end{subfigure}
    \begin{subfigure}[t]{0.32\linewidth}        
        \centering
        \includegraphics[trim=60 30 60 30, clip, width=\linewidth]{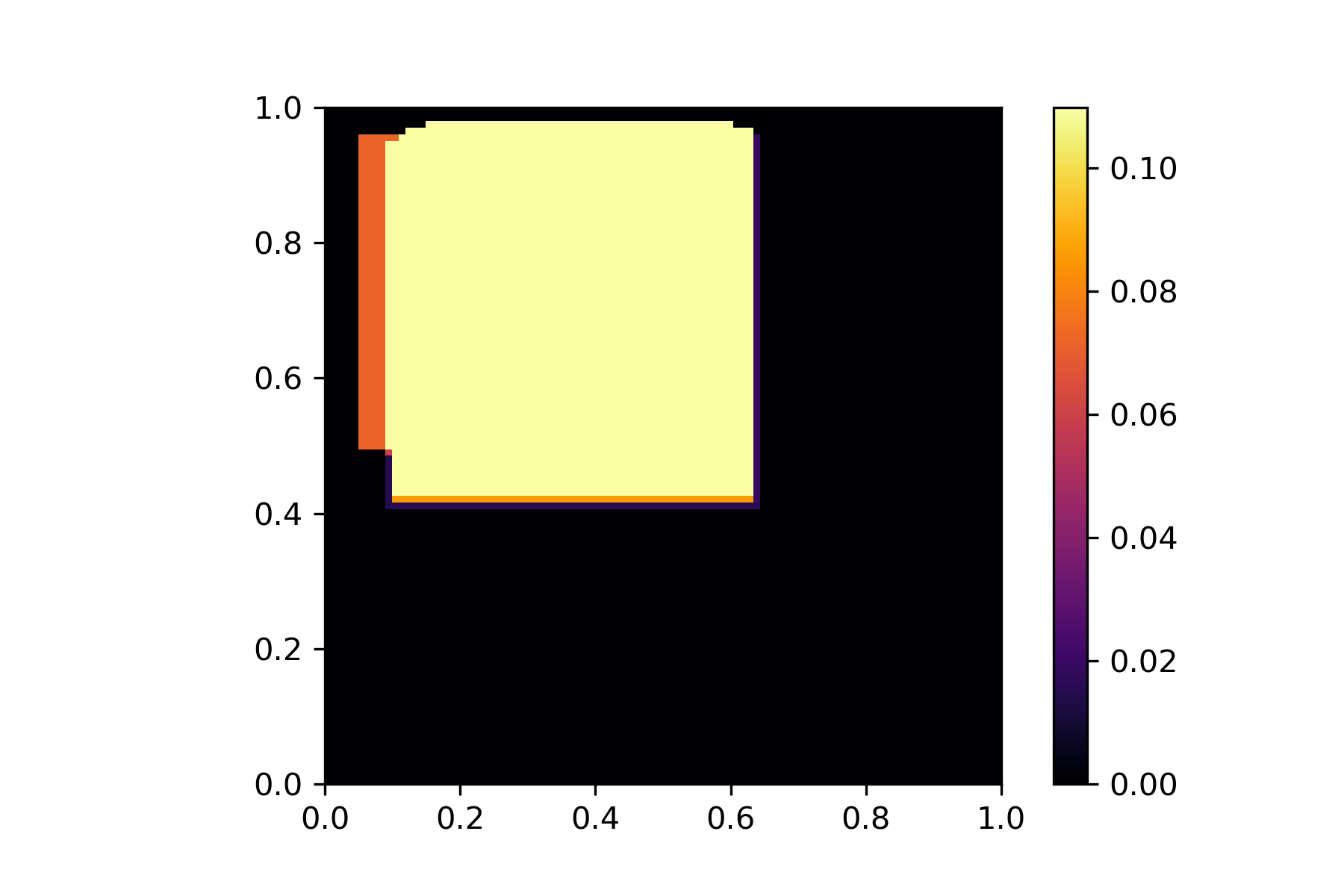}
        \caption{$p = 1$}
    \end{subfigure}
    \begin{subfigure}[t]{0.32\linewidth}        
        \centering
        \includegraphics[trim=60 30 60 30, clip, width=\linewidth]{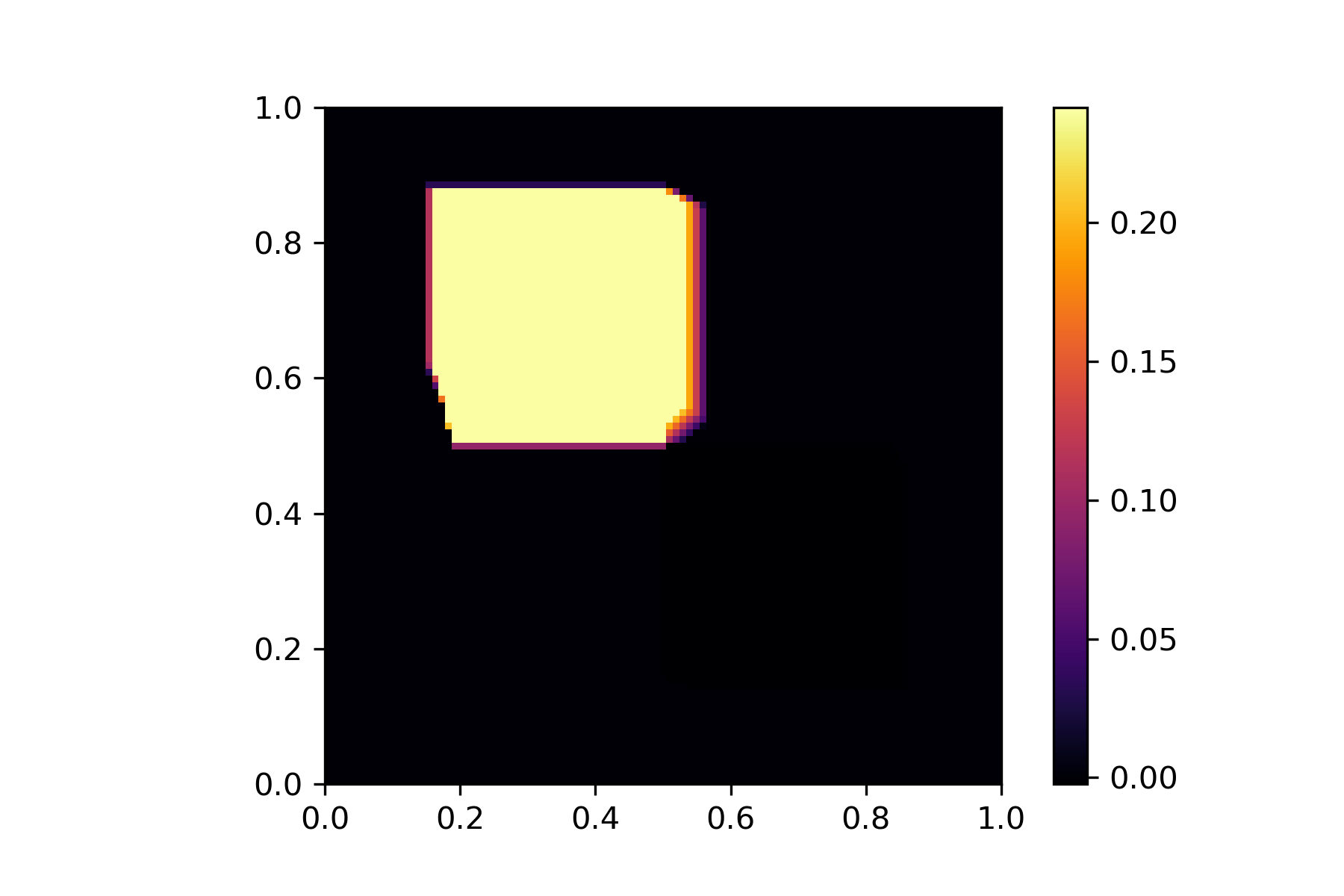}
        \caption{$p = \infty$}
    \end{subfigure}
    \caption{Comparison of the true source and inverse recoveries using weighted TV ($TV_w$) for different $p$-norms in the weights $\left\|K\partial_{y_i}G(\cdot;y) \right\|_{L^p(E)}$, cf. \eqref{eq:alfred2p}.}
    \label{fig:compare}
\end{figure}

\subsection*{Small rectangular source}
\captionsetup{justification=centering}
The first example concerns a small square-shaped (true) source positioned somewhat off the center of the grid, as shown in Figure \ref{fig:smallsource}(a). We can observe that the position of the source is recovered using both weighted TV and the hybrid method. However, the size of the true source is clearly overestimated by weighted TV approach, as seen in panel b). 

\begin{figure}[H]
    \centering
    \begin{subfigure}[t]{0.35\linewidth}        
        \centering
        \includegraphics[trim=25 15 15 30, clip, width=\linewidth]{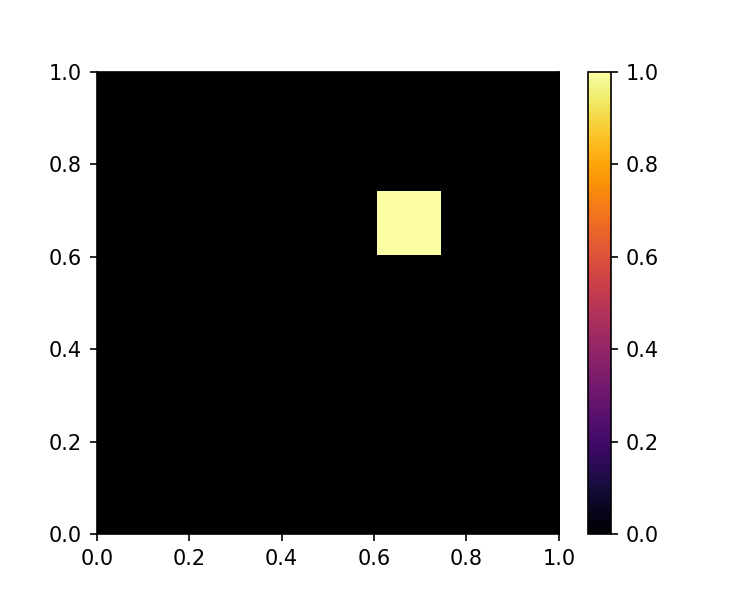}
        \caption{True small source}
    \end{subfigure}
    \begin{subfigure}[t]{0.35\linewidth}        
        \centering
        \includegraphics[trim=25 15 15 30, clip, width=\linewidth]{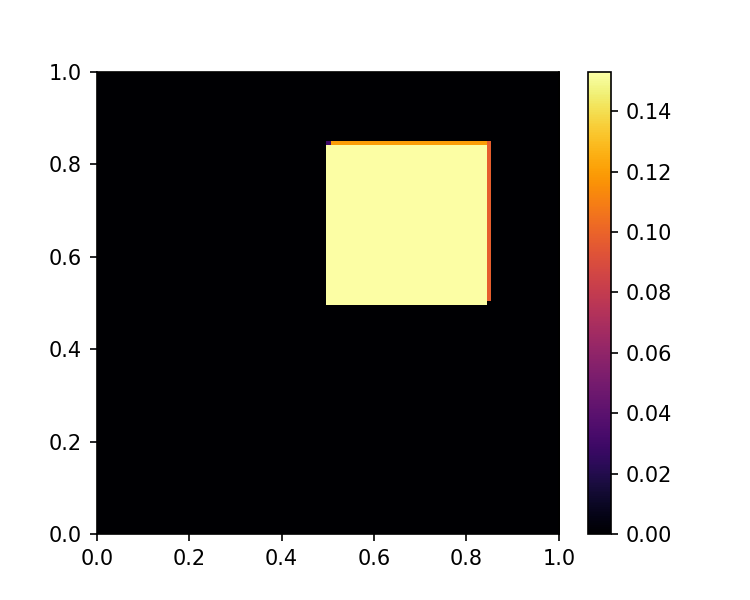}
        \caption{$\oTV_w$}
    \end{subfigure}
    \begin{subfigure}[t]{0.35\linewidth}        
        \centering
        \includegraphics[trim=25 15 15 30, clip, width=\linewidth]{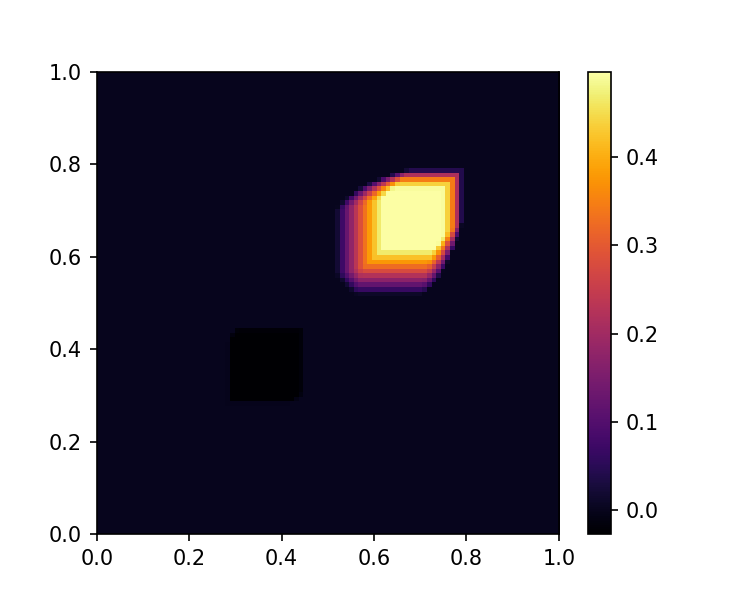}
        \caption{$\oTV_w+\ell^1_{\tilde{w}}$,
            }
    \end{subfigure}
    \begin{subfigure}[t]{0.35\linewidth}        
        \centering
        \includegraphics[trim=25 15 15 30, clip, width=\linewidth]{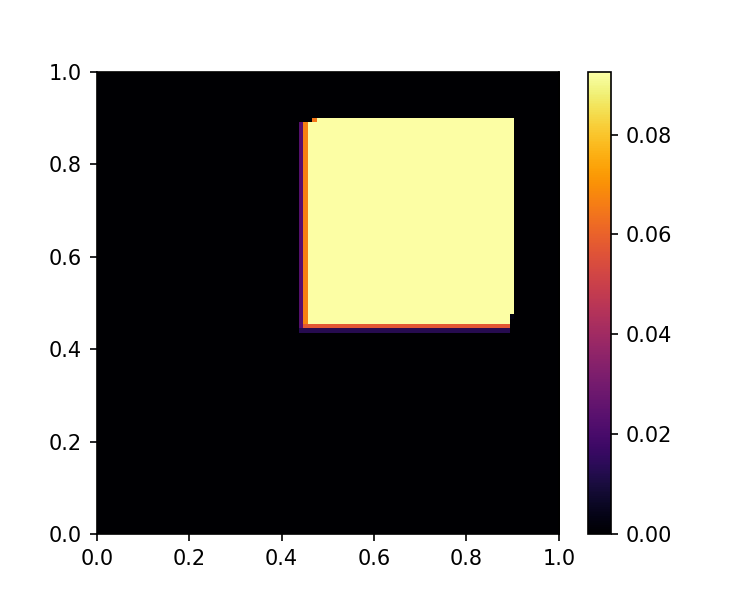}
        \caption{$\oTV_I+\ell^1_{\tilde{w}}$,
        }
    \end{subfigure}
    \caption{Comparison of the true source and inverse recoveries using weighted TV ($TV_w$) or unweighted/standard TV ($TV_I$), as well as hybrid TV and weighted $\ell^1_w$-norm.}
    \label{fig:smallsource}
\end{figure}

\subsection*{Large rectangular source}
\captionsetup{justification=centering}
In contrast to what was observed for the previous setup, the recovery is almost perfect with weighted TV when the true source is a larger square, cf. panels a) and b) in Figure \ref{fig:largesource}. If we apply the hybrid method with too strong emphasis on the $\ell_w^1$-term, the size of the true source is somewhat underestimated, see panel c).

A rough discussion of why weighted TV handles larger rectangles well, is presented in Section \ref{sec:2D3D_analysis} after the proof of Theorem \ref{thm:disjointLines}.  

\begin{figure}[H]
    \centering
    \begin{subfigure}[t]{0.35\linewidth}        
        \centering
        \includegraphics[trim=25 15 15 30, clip, width=\linewidth]{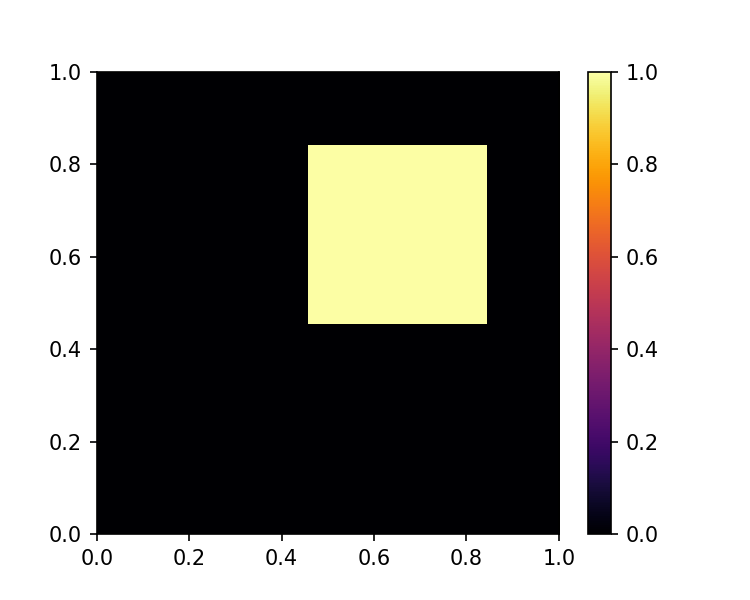}
        \caption{True large source}
    \end{subfigure}
    \begin{subfigure}[t]{0.35\linewidth}        
        \centering
        \includegraphics[trim=25 15 15 30, clip, width=\linewidth]{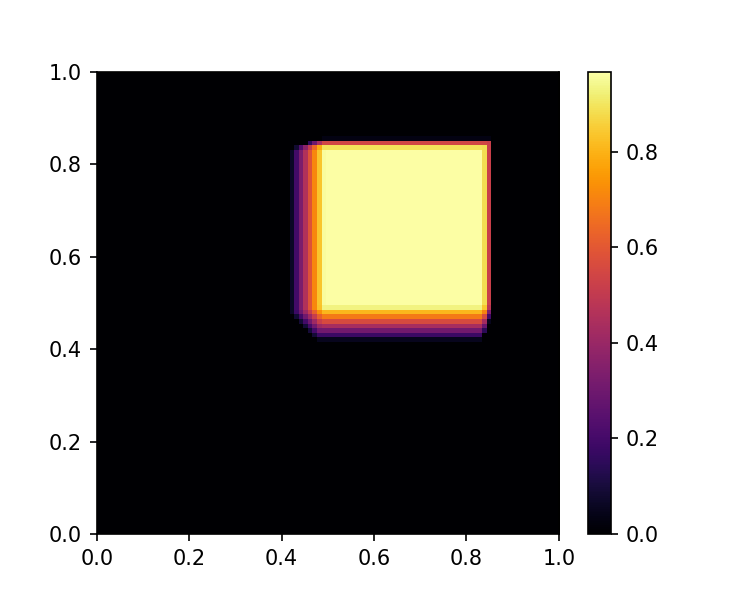}
        \caption{$\oTV_w$}
    \end{subfigure}\par
    \begin{subfigure}[t]{0.35\linewidth}        
        \centering
        \includegraphics[trim=25 15 15 30, clip, width=\linewidth]{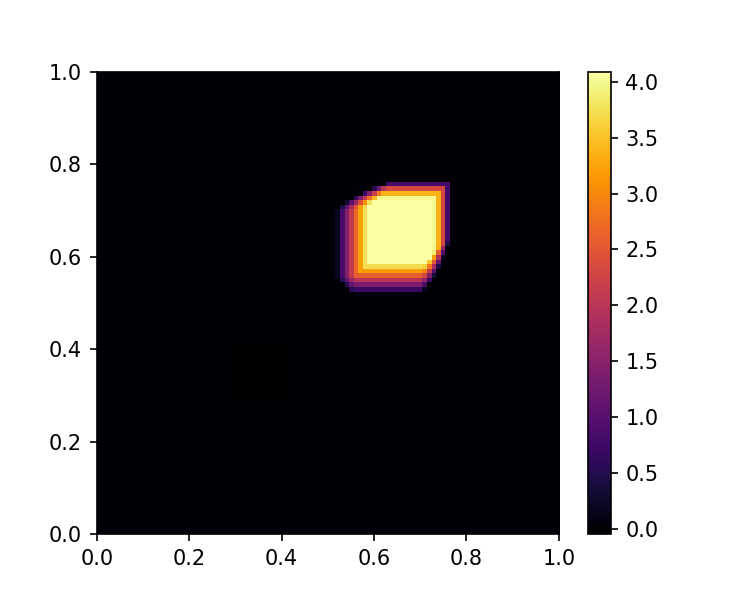}
        \caption{$\oTV_w+\ell^1_{\tilde{w}}$
        }
    \end{subfigure}
    \begin{subfigure}[t]{0.35\linewidth}        
        \centering
        \includegraphics[trim=25 15 15 30, clip, width=\linewidth]{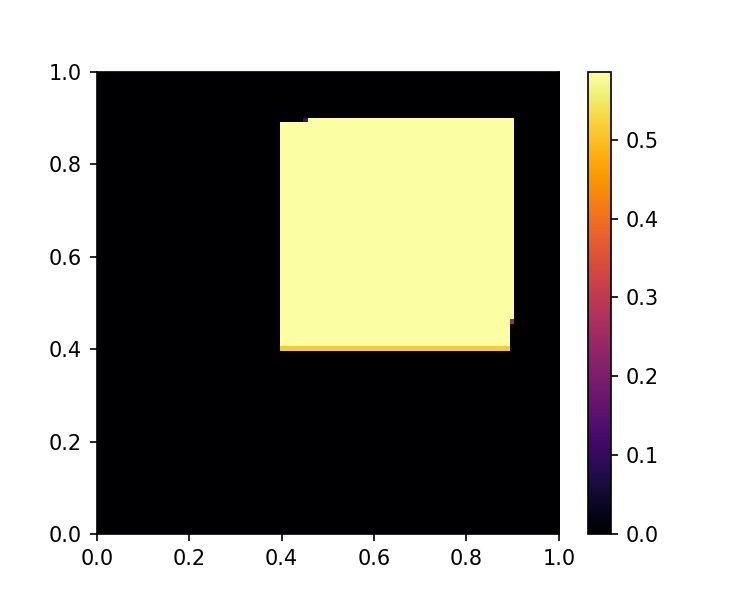}
        \caption{$\oTV_I+\ell^1_{\tilde{w}}$
        }
    \end{subfigure}
    \caption{Comparison of the true source and inverse recoveries using weighted TV ($TV_w$) or unweighted/standard TV ($TV_I$), as well as hybrid TV and weighted $\ell^1_w$-norm.}
    \label{fig:largesource}
\end{figure}

\subsection*{Positional influence}
The next example illuminates the ability of the methods to locate small true sources in different parts of the domain. When applying weighted TV, we observe that the further into the domain the true source is located, the more its size is overestimated by the inversion procedure, cf. panels (b), (e) and (h) of Figure \ref{fig:position}. The use of a hybrid approach can (somewhat) mitigate this problem. 
\begin{figure}[H]
    \centering
    \begin{subfigure}[t]{0.32\linewidth}        
        \centering
        \includegraphics[trim=25 15 15 30, clip, width=\linewidth]{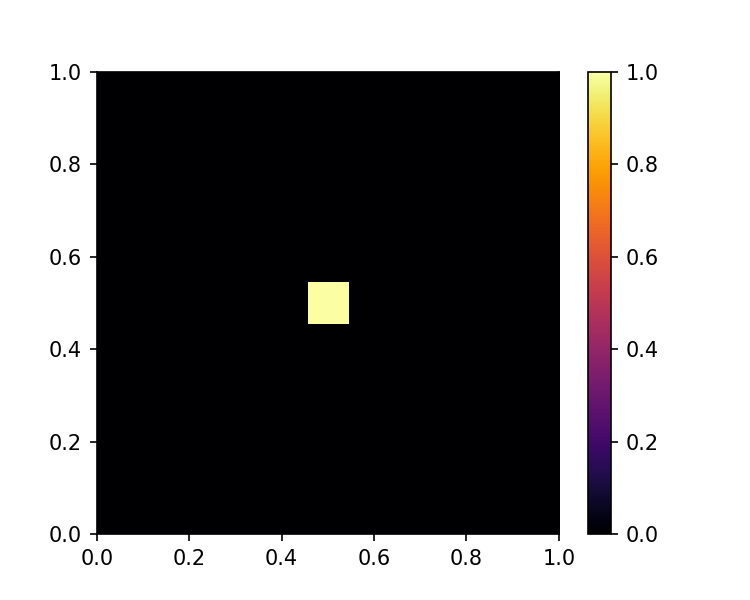}
        \caption{\rev{Deep} source}
    \end{subfigure}
    \begin{subfigure}[t]{0.32\linewidth}        
        \centering
        \includegraphics[trim=25 15 15 30, clip, width=\linewidth]{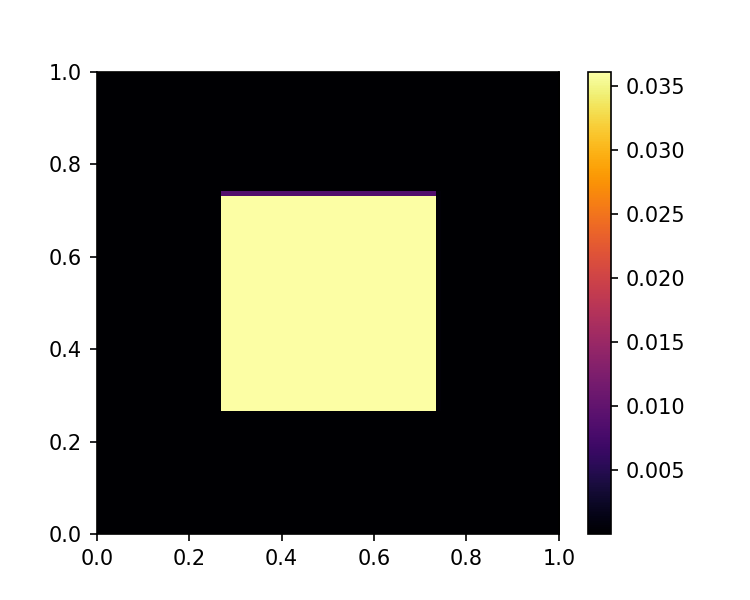}
        \caption{$\oTV_w$}
    \end{subfigure}
    \begin{subfigure}[t]{0.32\linewidth}        
        \centering
        \includegraphics[trim=25 15 15 30, clip, width=\linewidth]{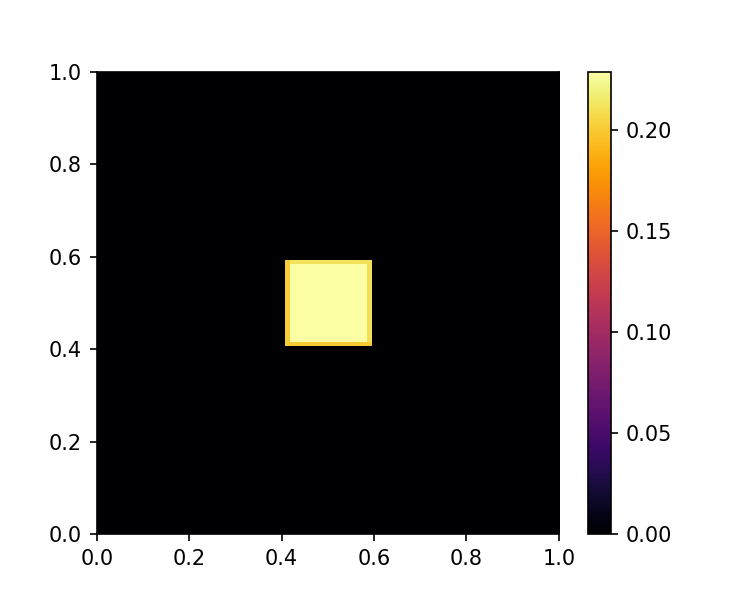}
        \caption{$\oTV_w+\ell^1_{\tilde{w}}$, 
        }
    \end{subfigure}\par
    \begin{subfigure}[t]{0.32\linewidth}        
        \centering
        \includegraphics[trim=25 15 15 30, clip, width=\linewidth]{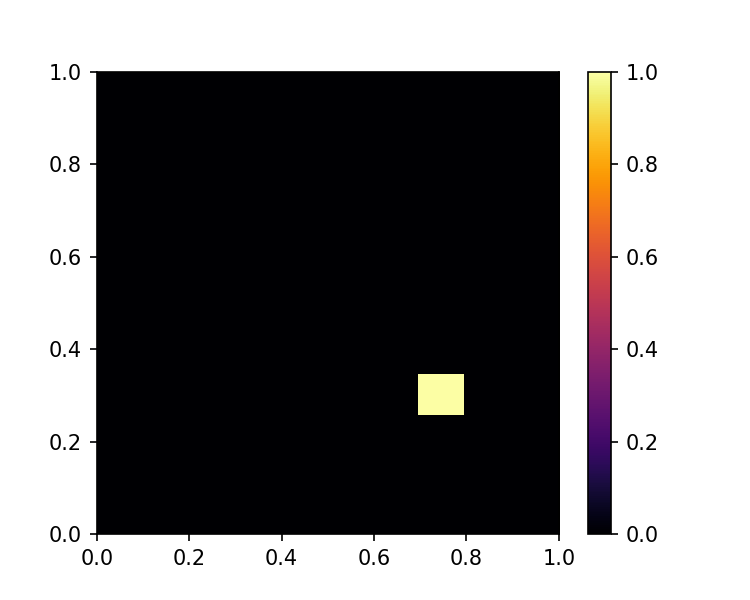}
        \caption{\rev{Intermediate} source}
    \end{subfigure}
    \begin{subfigure}[t]{0.32\linewidth}        
        \centering
        \includegraphics[trim=25 15 15 30, clip, width=\linewidth]{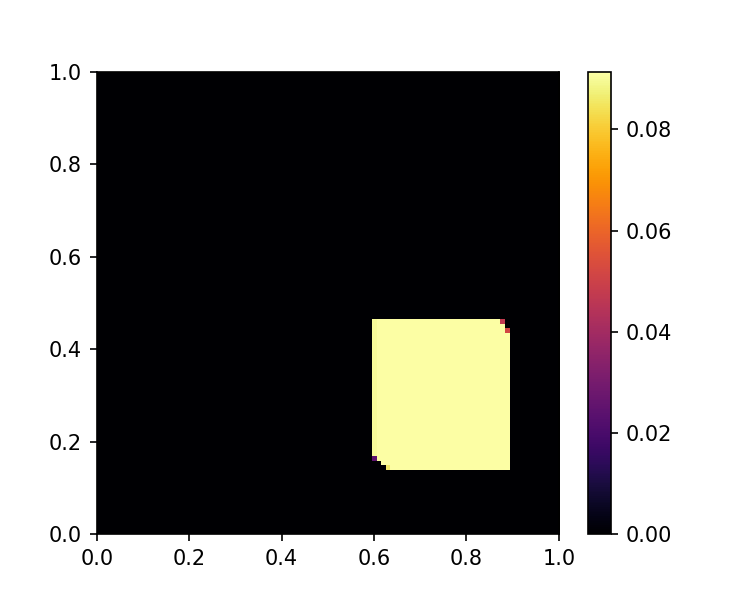}
        \caption{$\oTV_w$}
    \end{subfigure}
    \begin{subfigure}[t]{0.32\linewidth}        
        \centering
        \includegraphics[trim=25 15 15 30, clip, width=\linewidth]{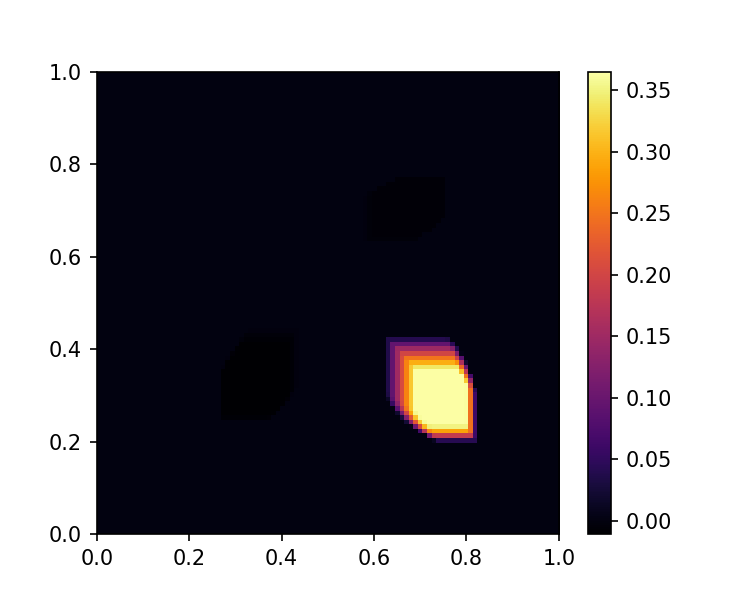}
        \caption{$\oTV_w+\ell^1_{\tilde{w}}$, 
        }
    \end{subfigure}\par
    \begin{subfigure}[t]{0.32\linewidth}        
        \centering
        \includegraphics[trim=25 15 15 30, clip, width=\linewidth]{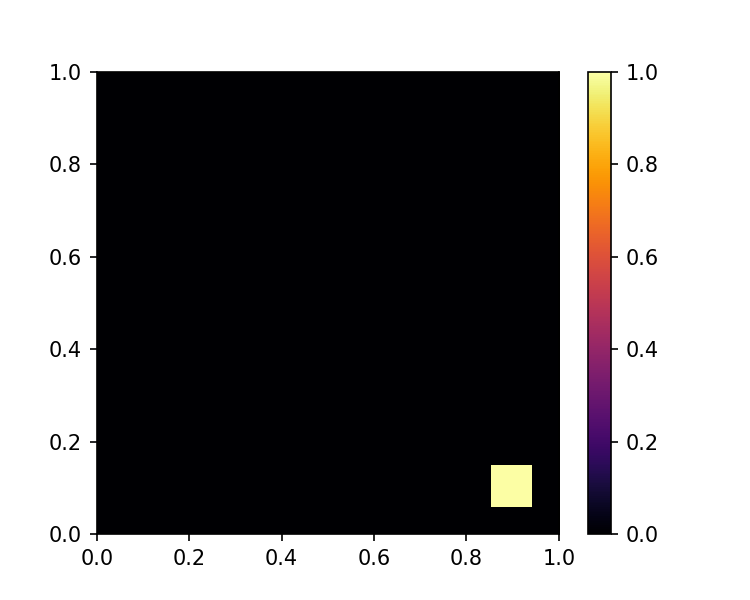}
        \caption{\rev{Shallow} source}
    \end{subfigure}
    \begin{subfigure}[t]{0.32\linewidth}        
        \centering
        \includegraphics[trim=25 15 15 30, clip, width=\linewidth]{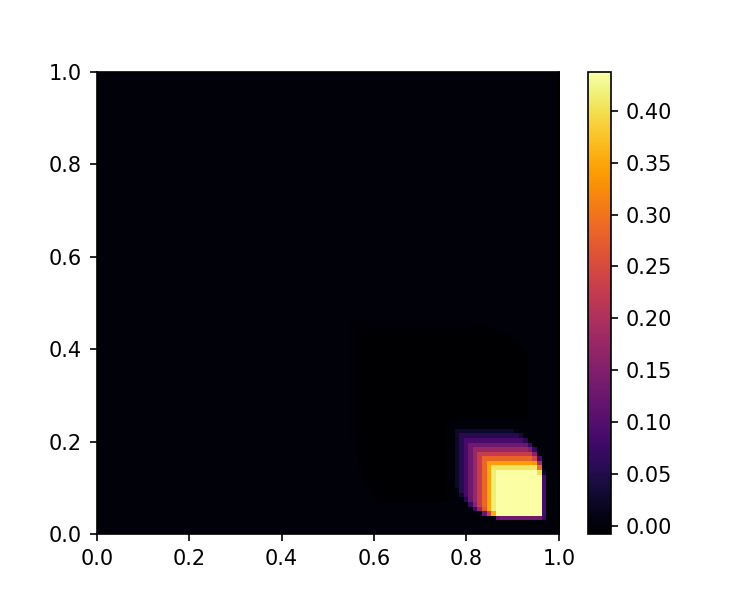}
        \caption{$\oTV_w$}
    \end{subfigure}
    \begin{subfigure}[t]{0.32\linewidth}        
        \centering
        \includegraphics[trim=25 15 15 30, clip, width=\linewidth]{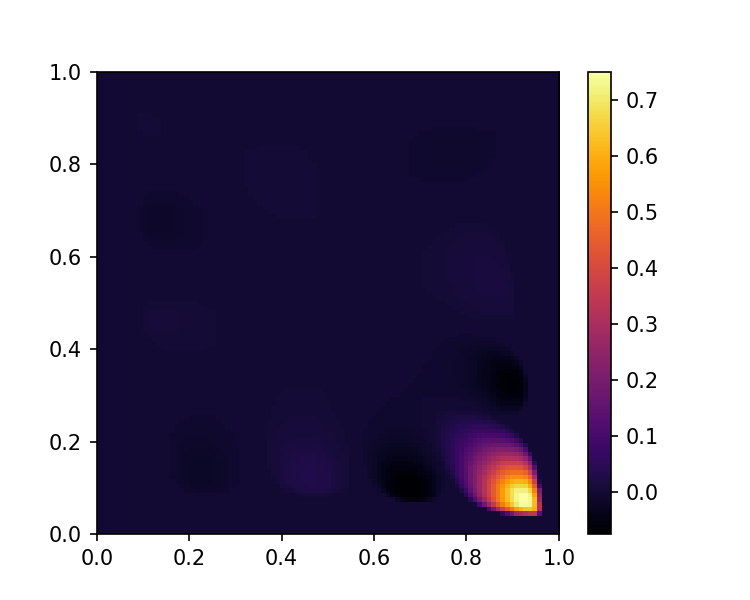}
        \caption{$\oTV_w+\ell^1_{\tilde{w}}$, 
        }
    \end{subfigure}\par
    \caption{Illustration of how the depth of the true source influences the size of the inversely recovered regions.}
    \label{fig:position}
\end{figure}

\subsection*{More advanced shapes}
For sources with more complex geometries, as shown in Figure \ref{fig:advancedsources}, recovering the true source shape becomes increasingly challenging. The weighted TV method tends to produce more block-like reconstructions (middle panels). When the hybrid method is applied (right panels), the reconstructions \rev{tend to underestimate the size of the true source(s), but the position is retained}.

\begin{figure}[H]
    \centering
    \begin{subfigure}[b]{0.32\linewidth}        
        \centering
        \includegraphics[trim=25 15 15 30, clip, width=\linewidth]{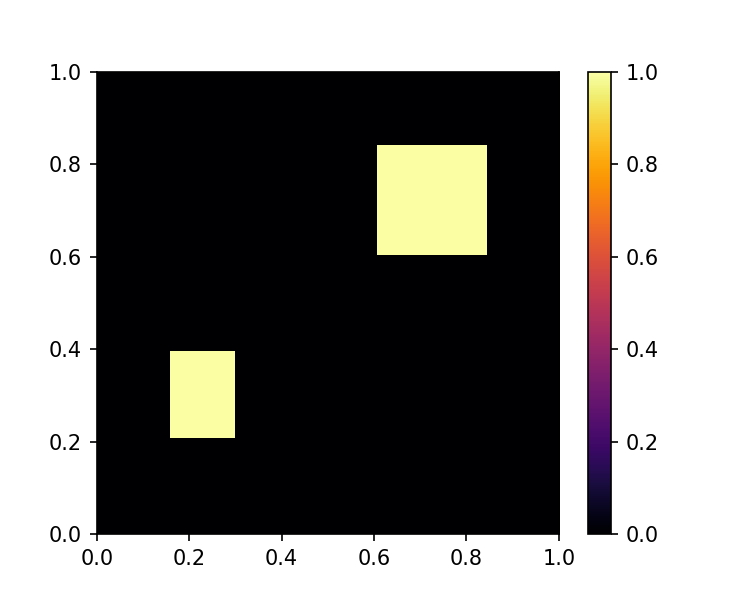}
        \caption{\rev{Double} source}
    \end{subfigure}
    \begin{subfigure}[b]{0.32\linewidth}        
        \centering
        \includegraphics[trim=25 15 15 30, clip, width=\linewidth]{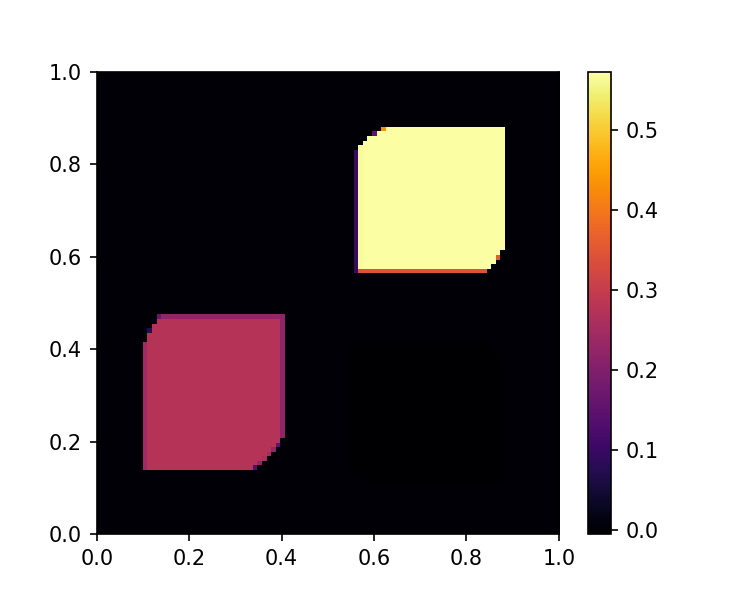}
        \caption{$\oTV_w$}
    \end{subfigure}
    \begin{subfigure}[b]{0.32\linewidth}        
        \centering
        \includegraphics[trim=25 15 15 30, clip, width=\linewidth]{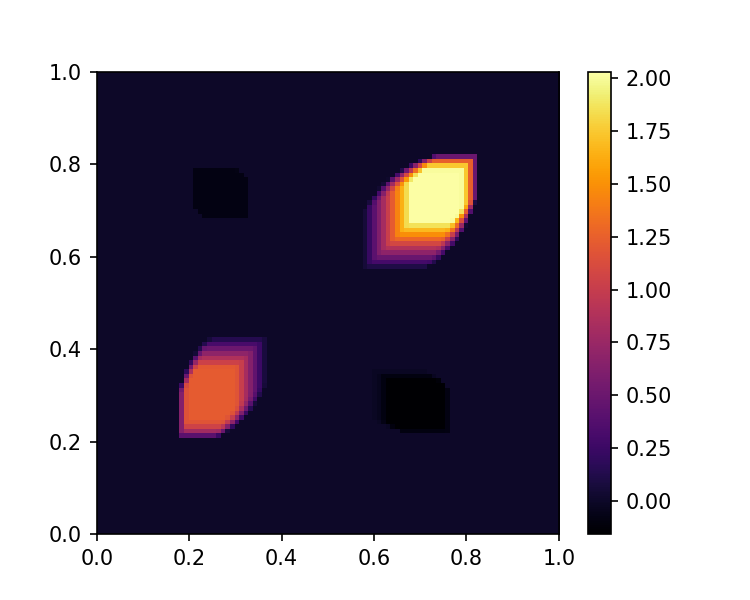}
        \caption{$\oTV_w + \ell^1_{\tilde{w}}$}
    \end{subfigure}\par
    \begin{subfigure}[b]{0.32\linewidth}        
        \centering
        \includegraphics[trim=25 15 15 30, clip, width=\linewidth]{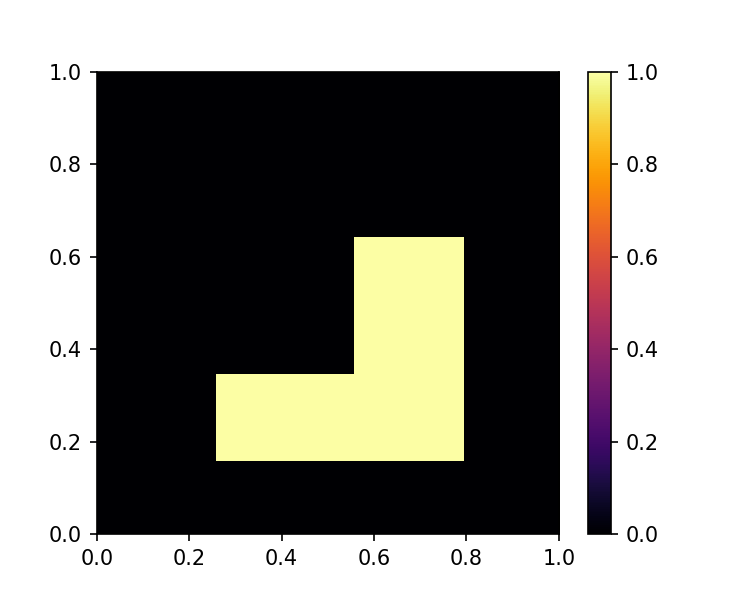}
        \caption{\rev{L-shaped} source}
    \end{subfigure}
    \begin{subfigure}[b]{0.32\linewidth}        
        \centering
        \includegraphics[trim=25 15 15 30, clip, width=\linewidth]{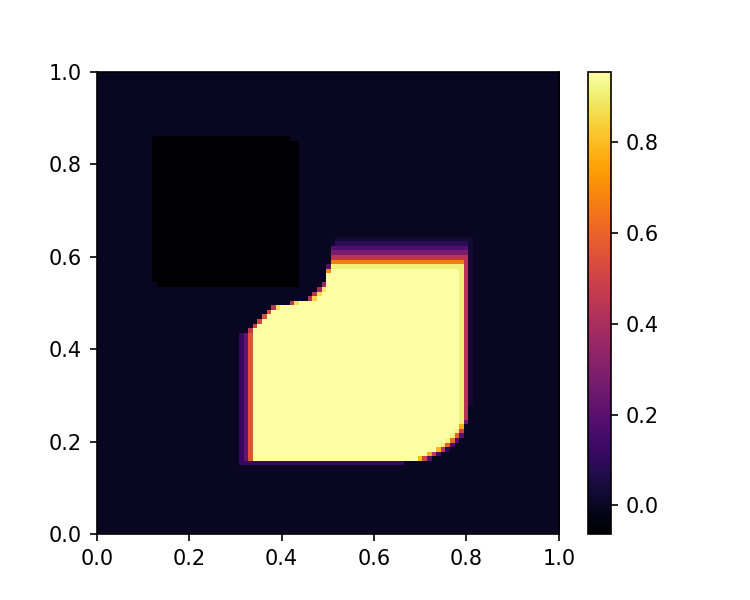}
        \caption{$\oTV_w$}
    \end{subfigure}
    \begin{subfigure}[b]{0.32\linewidth}        
        \centering
        \includegraphics[trim=25 15 15 30, clip, width=\linewidth]{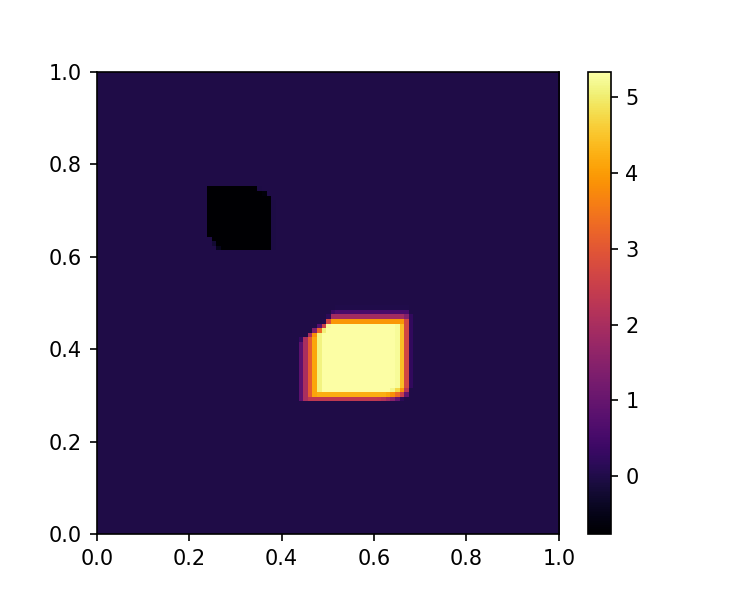}
        \caption{$\oTV_w + \ell^1_{\tilde{w}}$}
    \end{subfigure}\par
    \begin{subfigure}[b]{0.32\linewidth}        
        \centering
        \includegraphics[trim=25 15 15 30, clip, width=\linewidth]{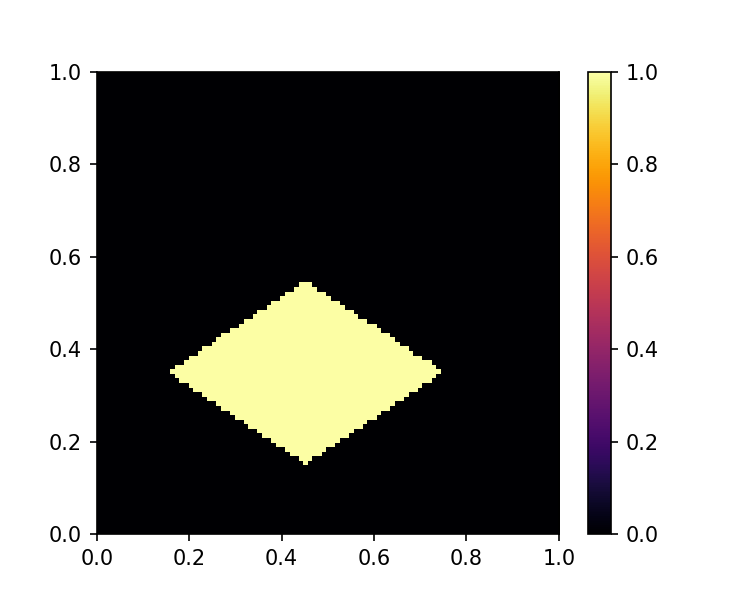}
        \caption{\rev{Diamond} source}
    \end{subfigure}
    \begin{subfigure}[b]{0.32\linewidth}        
        \centering
        \includegraphics[trim=25 15 15 30, clip, width=\linewidth]{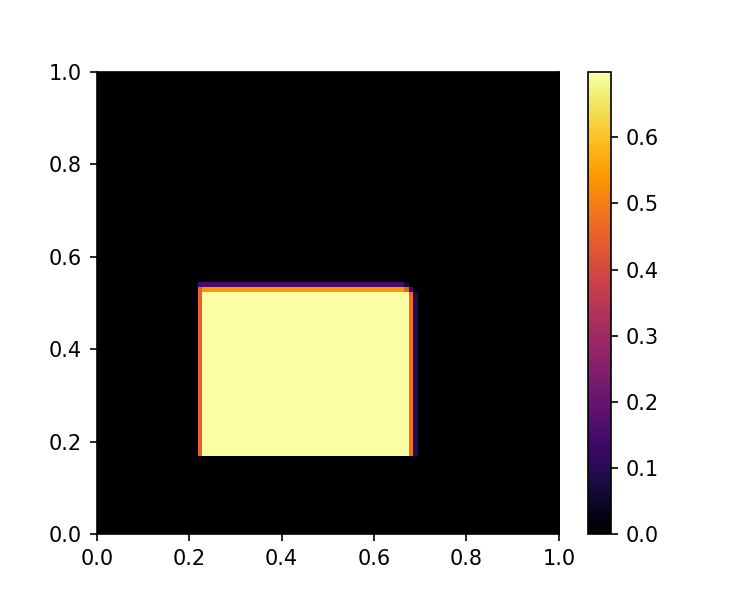}
        \caption{$\oTV_w$}
    \end{subfigure}
    \begin{subfigure}[b]{0.32\linewidth}        
        \centering
        \includegraphics[trim=25 15 15 30, clip, width=\linewidth]{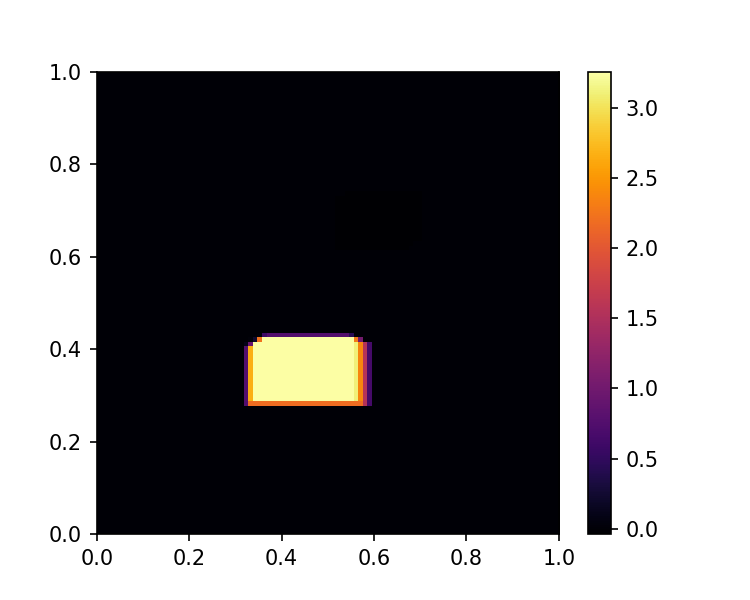}
        \caption{$\oTV_w + \ell^1_{\tilde{w}}$}
    \end{subfigure}\par
    \caption{Comparison of true generating source and the recovered source using weighted and unweighted total variation approaches.}
    \label{fig:advancedsources}
\end{figure}

\subsection*{\rev{Quantitative evaluation}}
\begin{revblock}
To complement the visual comparisons above, we assess the reconstructions quantitatively. For a reconstruction $\xx$ and the discrete true source $\xx^*$, we report the relative $L^1$ and $L^2$ errors
\begin{equation*}
    E_{1} = \frac{\|\xx - \xx^*\|_1}{\|\xx^*\|_1}, \qquad
    E_{2} = \frac{\|\xx - \xx^*\|_2}{\|\xx^*\|_2}.
\end{equation*}
To quantify how well the geometry of the source is recovered, we threshold the
reconstruction at half of its maximum value and compare the resulting support
\begin{equation*}
    S = \Big\{\, i : \xx_i > \tfrac{1}{2}\, \max_j \xx_j \,\Big\}
\end{equation*}
to the true support $S^* = \{\, i : \xx^*_i > \tfrac{1}{2} \,\}$, the true sources being indicator functions. Their overlap is measured by the Dice coefficient
\begin{equation*}
    \mathrm{Dice}(S,S^*) = \frac{2\,|S \cap S^*|}{|S| + |S^*|} \in [0,1],
\end{equation*}
and the discrepancy of their boundaries $\partial S$ and $\partial S^*$ by the Hausdorff distance
\begin{equation*}
    d_H(\partial S, \partial S^*) = \max\left\{\,
        \max_{a \in \partial S} \min_{b \in \partial S^*} |a-b|, \;
        \max_{b \in \partial S^*} \min_{a \in \partial S} |a-b|
    \,\right\},
\end{equation*}
where the boundary points $a,b$ are the centres of the corresponding grid cells. These quantities are collected in Table~\ref{tab:metrics}, averaged over three independent noise realizations. Besides weighted TV ($\oTV_w$), the hybrid methods ($\oTV_w+\ell^1_{\tilde w}$ and $\oTV_I+\ell^1_{\tilde w}$) and standard TV ($\oTV_I$), we include an iteratively reweighted TV baseline (IRW)~\cite{candes2008enhancing}, i.e.\ a purely content-based weighting whose weights are derived from the current iterate rather than from the forward operator. The weighted schemes outperform both baselines on every source, partially with the exception of the boundary source, while IRW performs essentially on par with standard TV. Consistent with the qualitative observations, the hybrid method improves the small and compact sources but degrades the larger, extended ones, where the $\ell^1$-term over-sparsifies the reconstruction (cf.\ the $\oTV_w$ and hybrid rows for \emph{large}, \emph{lshape} and \emph{diamond}).
\end{revblock}

\begin{table}[H]
\centering
\begin{revblock}
\begin{tabular}{llrrrr}
\hline
source & method & rel.\ $L^1$ & rel.\ $L^2$ & $d_H$ & Dice \\
\hline
\emph{small} & $\oTV_w$ & 1.690 & 0.920 & 0.156 & 0.263 \\
 & $\oTV_w+\ell^1_{\tilde w}$ & \textbf{0.713} & \textbf{0.522} & \textbf{0.040} & \textbf{0.781} \\
 & $\oTV_I$ & 1.929 & 0.986 & 0.707 & 0.048 \\
 & $\oTV_I+\ell^1_{\tilde w}$ & 1.806 & 0.951 & 0.226 & 0.170 \\
 & IRW & 1.943 & 0.987 & 0.665 & 0.050 \\
\hline
\emph{large} & $\oTV_w$ & \textbf{0.136} & \textbf{0.199} & \textbf{0.013} & \textbf{0.983} \\
 & $\oTV_w+\ell^1_{\tilde w}$ & 1.464 & 1.820 & 0.173 & 0.350 \\
 & $\oTV_I$ & 1.390 & 0.825 & 0.226 & 0.465 \\
 & $\oTV_I+\ell^1_{\tilde w}$ & 0.823 & 0.640 & 0.085 & 0.740 \\
 & IRW & 1.396 & 0.831 & 0.226 & 0.477 \\
\hline
\emph{deep} & $\oTV_w$ & 1.919 & 0.982 & 0.269 & 0.071 \\
 & $\oTV_w+\ell^1_{\tilde w}$ & \textbf{1.541} & \textbf{0.875} & \textbf{0.071} & \textbf{0.367} \\
 & $\oTV_I$ & 1.910 & 0.988 & 0.450 & 0.048 \\
 & $\oTV_I+\ell^1_{\tilde w}$ & 1.947 & 0.992 & 0.438 & 0.032 \\
 & IRW & 1.963 & 0.996 & 0.651 & 0.016 \\
\hline
\emph{intermediate} & $\oTV_w$ & 1.819 & 0.952 & 0.156 & 0.171 \\
 & $\oTV_w+\ell^1_{\tilde w}$ & \textbf{1.255} & \textbf{0.739} & \textbf{0.054} & \textbf{0.566} \\
 & $\oTV_I$ & 1.945 & 0.995 & 0.948 & 0.018 \\
 & $\oTV_I+\ell^1_{\tilde w}$ & 1.860 & 0.965 & 0.191 & 0.127 \\
 & IRW & 1.976 & 0.995 & 0.948 & 0.018 \\
\hline
\emph{shallow} & $\oTV_w$ & 1.249 & 0.677 & 0.031 & 0.661 \\
 & $\oTV_w+\ell^1_{\tilde w}$ & 3.337 & 0.809 & \textbf{0.022} & \textbf{0.851} \\
 & $\oTV_I$ & 2.000 & 0.993 & 0.877 & 0.027 \\
 & $\oTV_I+\ell^1_{\tilde w}$ & \textbf{0.878} & \textbf{0.594} & 0.028 & 0.761 \\
 & IRW & 2.042 & 0.993 & 0.849 & 0.028 \\
\hline
\emph{double} & $\oTV_w$ & 1.063 & \textbf{0.719} & 0.258 & 0.611 \\
 & $\oTV_w+\ell^1_{\tilde w}$ & \textbf{0.874} & 0.903 & \textbf{0.099} & \textbf{0.622} \\
 & $\oTV_I$ & 1.793 & 0.945 & 0.385 & 0.201 \\
 & $\oTV_I+\ell^1_{\tilde w}$ & 1.302 & 0.783 & 0.142 & 0.557 \\
 & IRW & 1.768 & 0.938 & 0.300 & 0.250 \\
\hline
\emph{lshape} & $\oTV_w$ & \textbf{0.481} & \textbf{0.565} & \textbf{0.177} & \textbf{0.825} \\
 & $\oTV_w+\ell^1_{\tilde w}$ & 1.898 & 2.331 & 0.243 & 0.220 \\
 & $\oTV_I$ & 1.425 & 0.841 & 0.357 & 0.445 \\
 & $\oTV_I+\ell^1_{\tilde w}$ & 0.928 & 0.671 & 0.226 & 0.716 \\
 & IRW & 1.426 & 0.843 & 0.353 & 0.447 \\
\hline
\emph{diamond} & $\oTV_w$ & \textbf{0.698} & \textbf{0.604} & \textbf{0.113} & \textbf{0.770} \\
 & $\oTV_w+\ell^1_{\tilde w}$ & 1.315 & 1.424 & 0.163 & 0.480 \\
 & $\oTV_I$ & 1.620 & 0.898 & 0.386 & 0.314 \\
 & $\oTV_I+\ell^1_{\tilde w}$ & 1.210 & 0.778 & 0.202 & 0.568 \\
 & IRW & 1.638 & 0.906 & 0.405 & 0.301 \\
\hline
\end{tabular}
\caption{Quantitative reconstruction errors for the isotropic experiments ($n=101$): the best value in each column (per source) is shown in bold. All values are averaged over three noise realizations.}
\label{tab:metrics}
\end{revblock}
\end{table}

\begin{table}[H]
\centering
\begin{revblock}
\begin{tabular}{lrrrrr}
\hline
source & $\alpha_{\oTV_w}$ & $\alpha_{\oTV_w+\ell^1}$ & $\alpha_{\oTV_I+\ell^1}$ & $\alpha_{\oTV_I}$ & $\alpha_{\mathrm{IRW}}$ \\ \noalign{\smallskip}
\hline
\emph{small} & $3.6\times 10^{-5}$ & $1.7\times 10^{-6}$ & $2.2\times 10^{-7}$ & $1.7\times 10^{-6}$ & $1.7\times 10^{-6}$ \\
\emph{large} & $1.0\times 10^{-4}$ & $1.3\times 10^{-5}$ & $2.8\times 10^{-6}$ & $4.7\times 10^{-6}$ & $1.0\times 10^{-4}$ \\
\emph{deep} & $1.0\times 10^{-4}$ & $2.2\times 10^{-5}$ & $1.0\times 10^{-6}$ & $1.7\times 10^{-6}$ & $6.0\times 10^{-5}$ \\
\emph{intermediate} & $7.8\times 10^{-6}$ & $3.6\times 10^{-7}$ & $7.8\times 10^{-8}$ & $1.7\times 10^{-6}$ & $1.7\times 10^{-6}$ \\
\emph{shallow} & $2.2\times 10^{-7}$ & $3.7\times 10^{-9}$ & $3.7\times 10^{-9}$ & $1.7\times 10^{-6}$ & $1.7\times 10^{-6}$ \\
\emph{double} & $6.0\times 10^{-5}$ & $1.7\times 10^{-6}$ & $6.0\times 10^{-7}$ & $1.7\times 10^{-6}$ & $4.0\times 10^{-6}$ \\
\emph{lshape} & $1.0\times 10^{-4}$ & $2.2\times 10^{-5}$ & $2.8\times 10^{-6}$ & $7.8\times 10^{-6}$ & $1.0\times 10^{-4}$ \\
\emph{diamond} & $1.0\times 10^{-4}$ & $1.3\times 10^{-5}$ & $2.8\times 10^{-6}$ & $4.7\times 10^{-6}$ & $1.0\times 10^{-4}$ \\
\hline
\end{tabular}
\caption{Regularization parameter $\alpha$ selected by the Morozov discrepancy principle for each source. For the two hybrid schemes ($\oTV_w+\ell^1$ and $\oTV_I+\ell^1$) the weight $\beta=10^{3}\alpha$ is fixed, so only the scale is data-driven; the pure-TV methods use $\beta=0$.}
\label{tab:alpha}
\end{revblock}
\end{table}

\subsection*{Anisotropic conductivity}

We will now consider two problems involving anisotropic conductivities, that is, the forward operator now involved the boundary value problem \eqref{eq:ellipticPDE}-\eqref{eq:neumannBC} with ${\color{revcol}\kappa}$ equal to 

\begin{equation}
{\color{revcol}\kappa}(x_1,x_2) := {\color{revcol}\kappa_1}(x_1,x_2) = (1+9x_1)\begin{bmatrix} 5 & 0 \\ 0 & 1  \end{bmatrix}, \label{eq:conductivity1}    
\end{equation}
and
\begin{equation}
{\color{revcol}\kappa}(x_1,x_2) := {\color{revcol}\kappa_2}(x_1,x_2) = 
    \begin{cases}
        10, &x_1 \leq 0.4, x_2\leq 0.4, \\
        1, &x_1 > 0.4, x_2 > 0.4,
    \end{cases} \label{eq:conductivity2}
\end{equation}
respectively.

Figure \ref{fig:anisotropic1} displays simulations with the conductivity given in \eqref{eq:conductivity1}. The conductivity is strongly increasing in the $x_1$-direction, which might explain the somewhat more blurry solution seen in panel h). Nevertheless, the weighted inversion scheme is still able to recover the position of the true sources.

\begin{figure}[H]
    \centering
    \begin{subfigure}[b]{0.32\linewidth}        
        \centering
        \includegraphics[trim=60 30 60 30, clip, width=\linewidth]{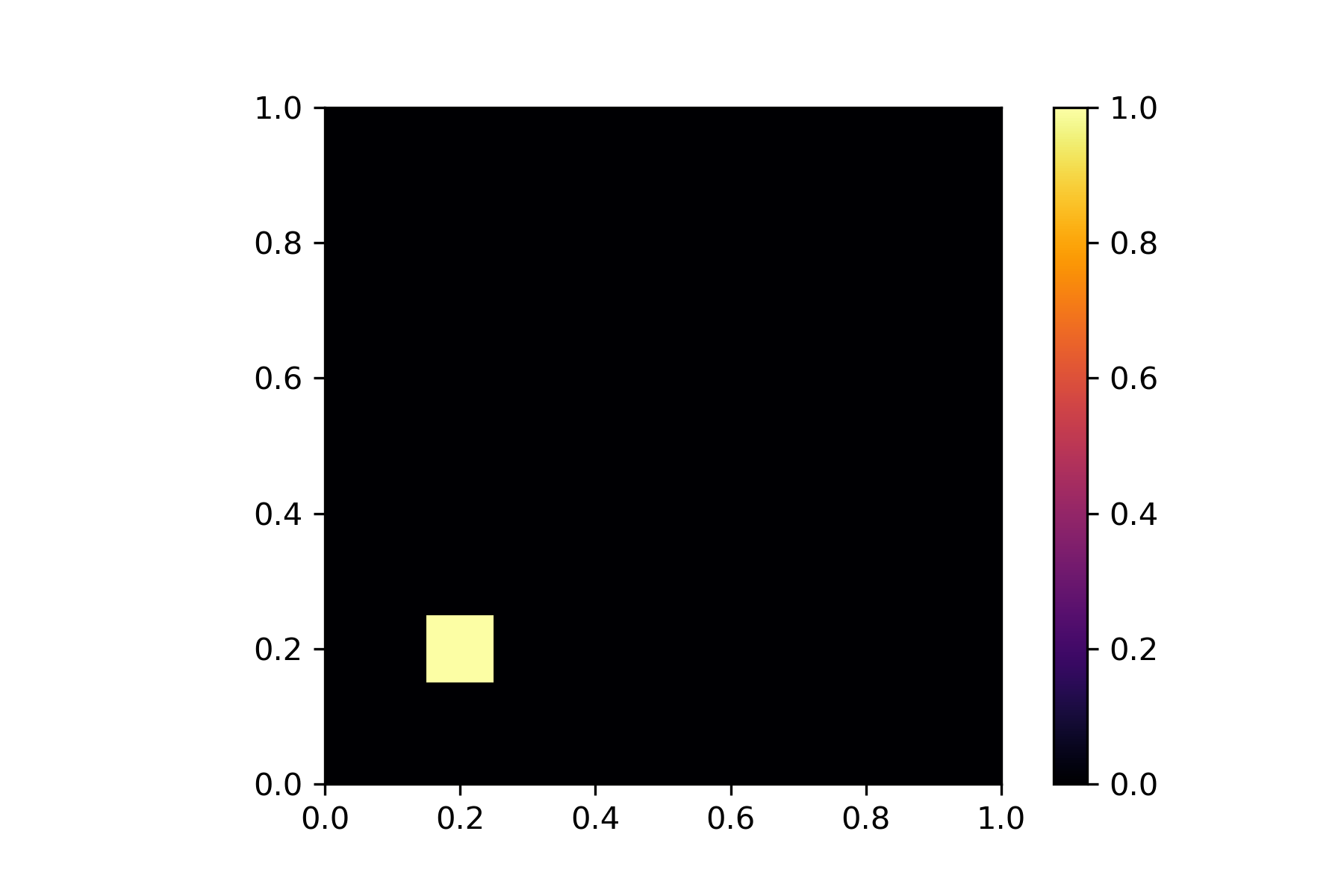}
        \caption{True source}
    \end{subfigure}
    \begin{subfigure}[b]{0.32\linewidth}        
        \centering
        \includegraphics[trim=60 30 60 30, clip, width=\linewidth]{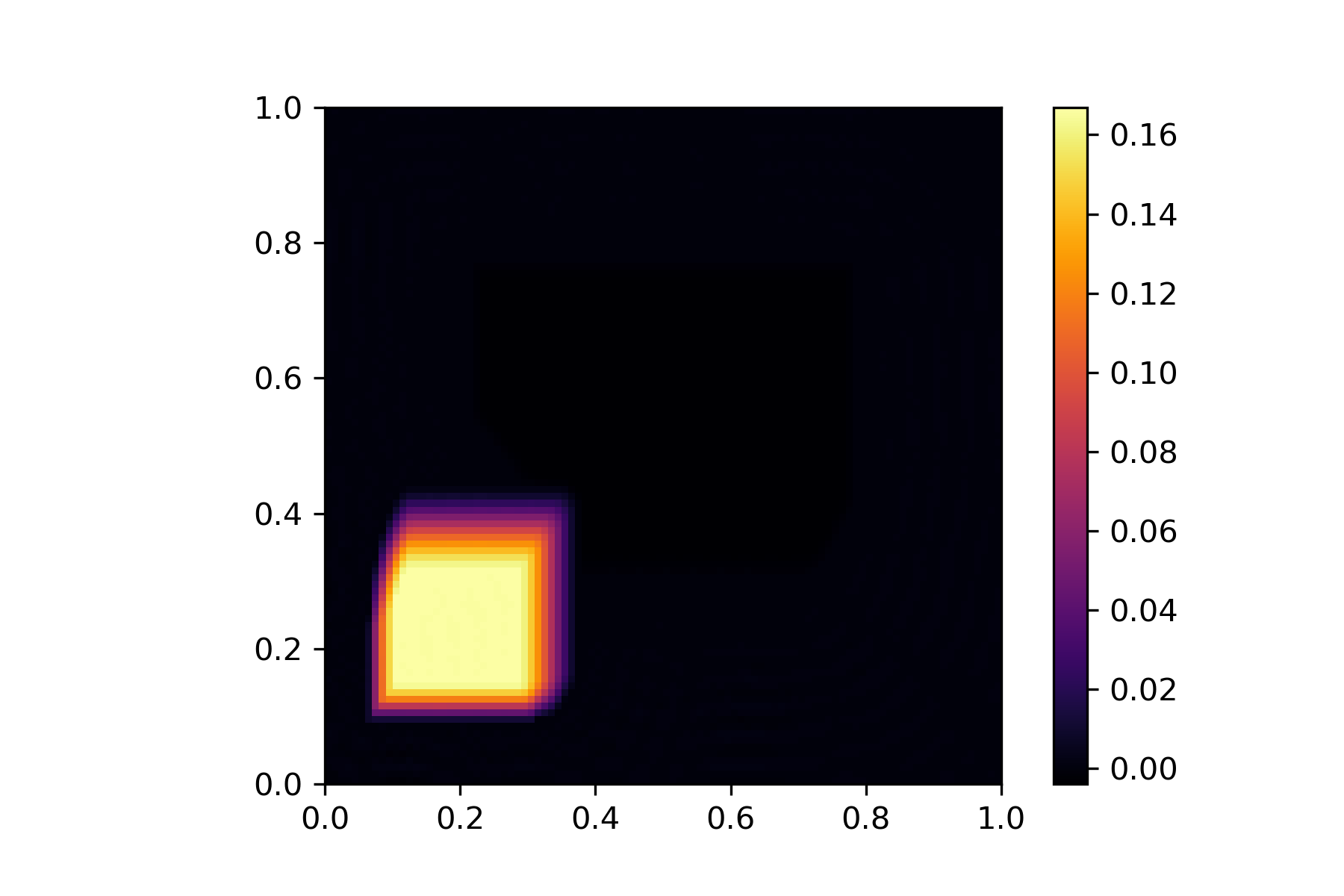}
        \caption{$TV_w$}
    \end{subfigure}
    \begin{subfigure}[b]{0.32\linewidth}        
        \centering
        \includegraphics[trim=60 30 60 30, clip, width=\linewidth]{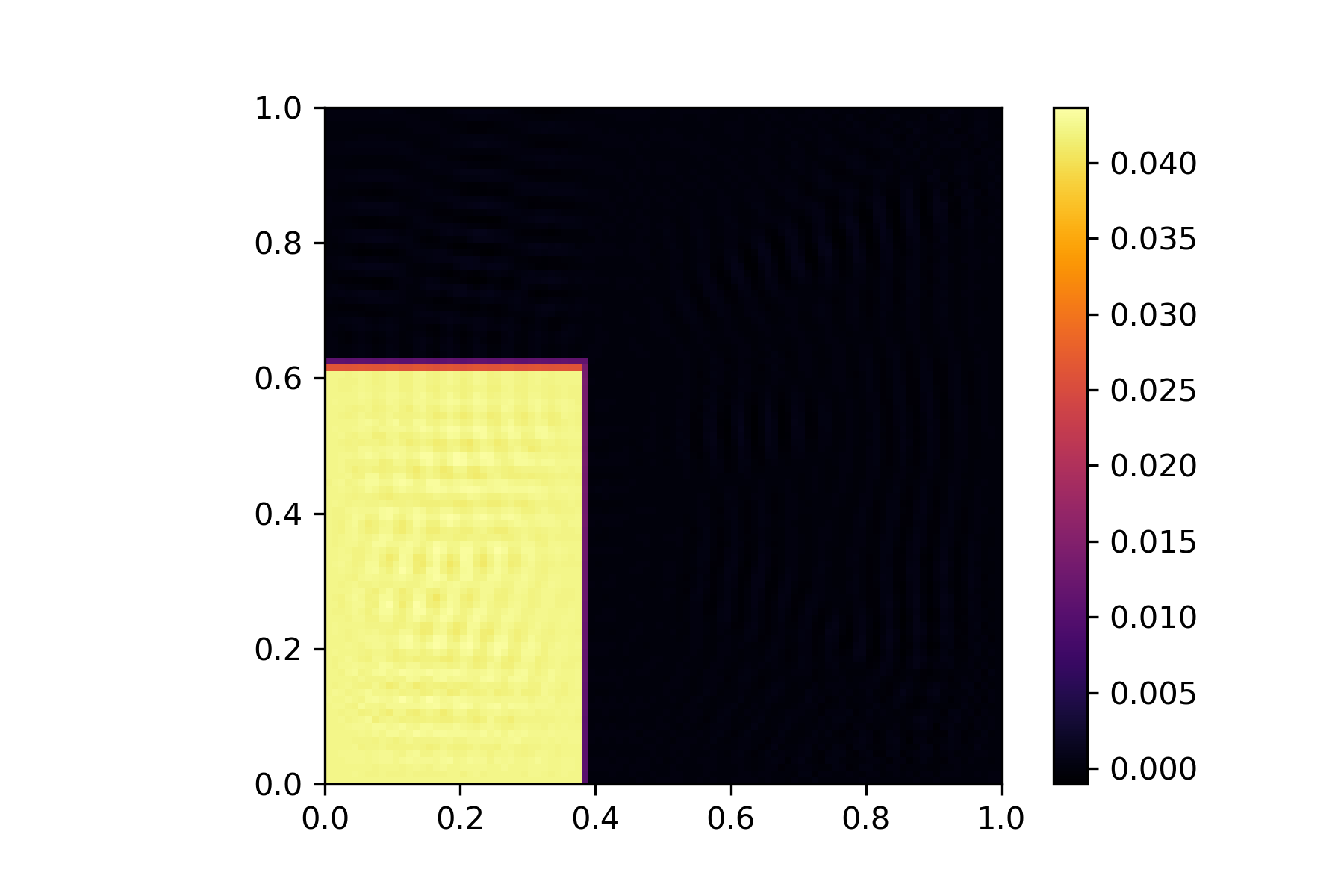}
        \caption{$TV_I$}
    \end{subfigure}\par
    \begin{subfigure}[b]{0.32\linewidth}        
        \centering
        \includegraphics[trim=60 30 60 30, clip, width=\linewidth]{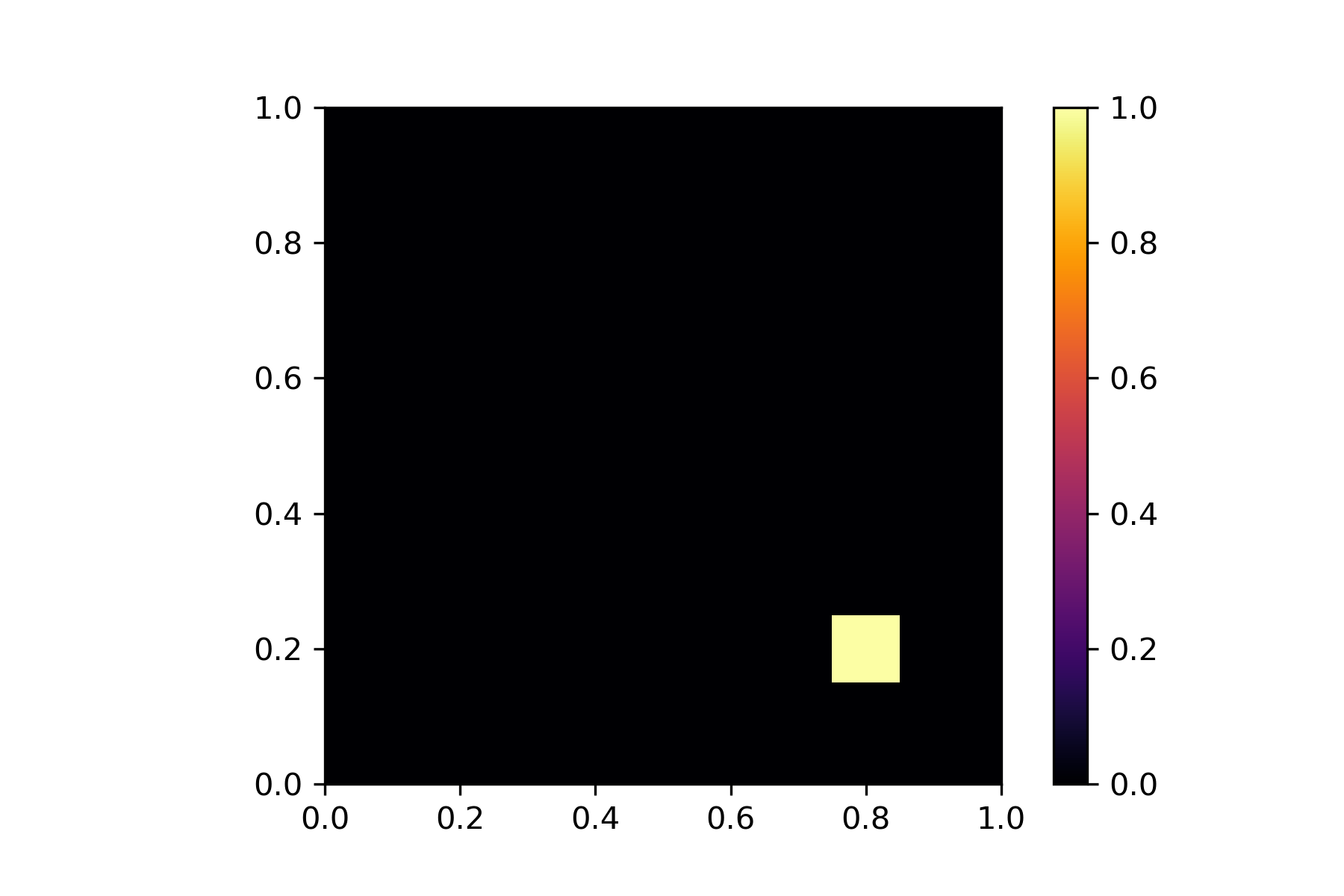}
        \caption{True source}
    \end{subfigure}
    \begin{subfigure}[b]{0.32\linewidth}        
        \centering
        \includegraphics[trim=60 30 60 30, clip, width=\linewidth]{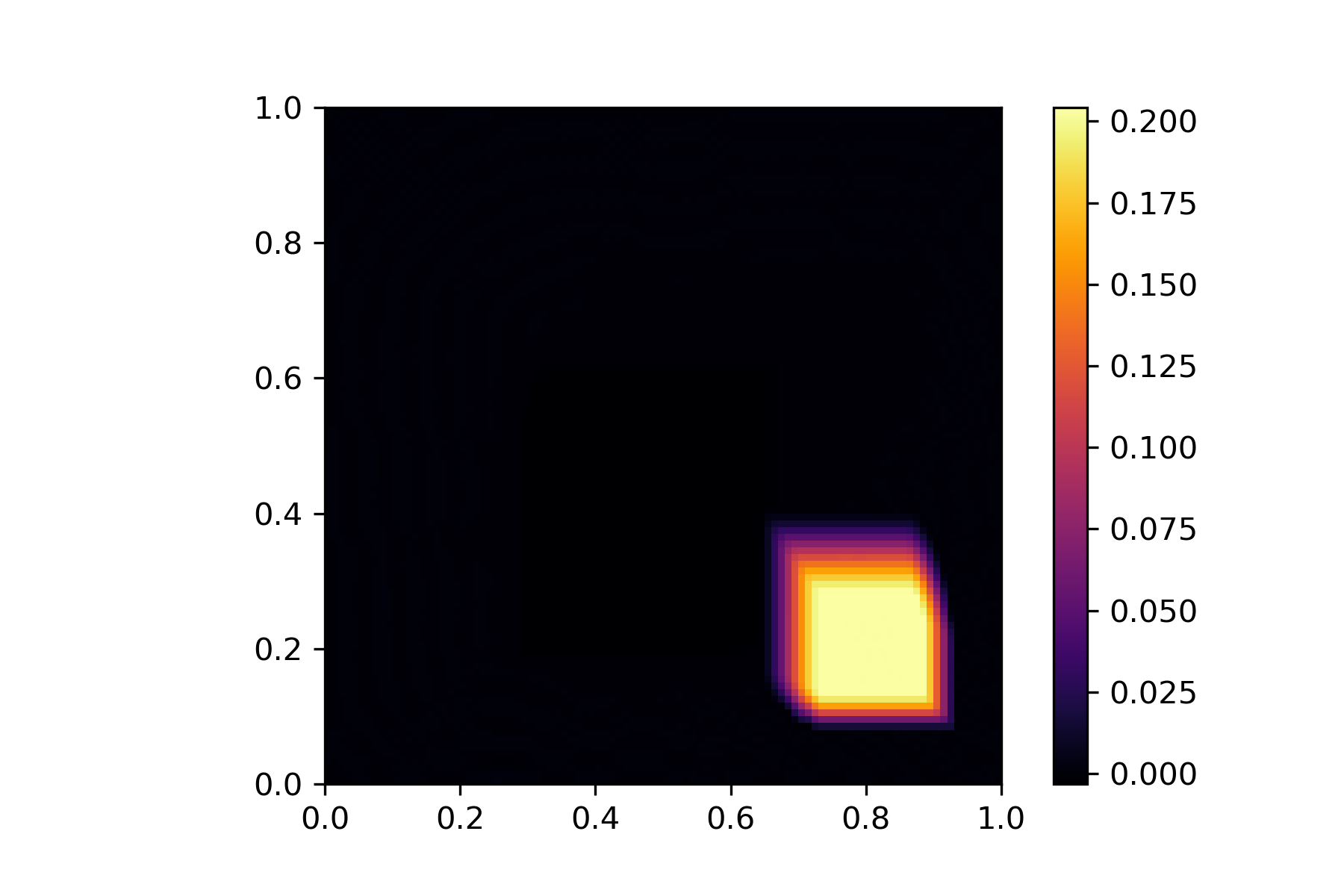}
        \caption{$TV_w$}
    \end{subfigure}
    \begin{subfigure}[b]{0.32\linewidth}        
        \centering
        \includegraphics[trim=60 30 60 30, clip, width=\linewidth]{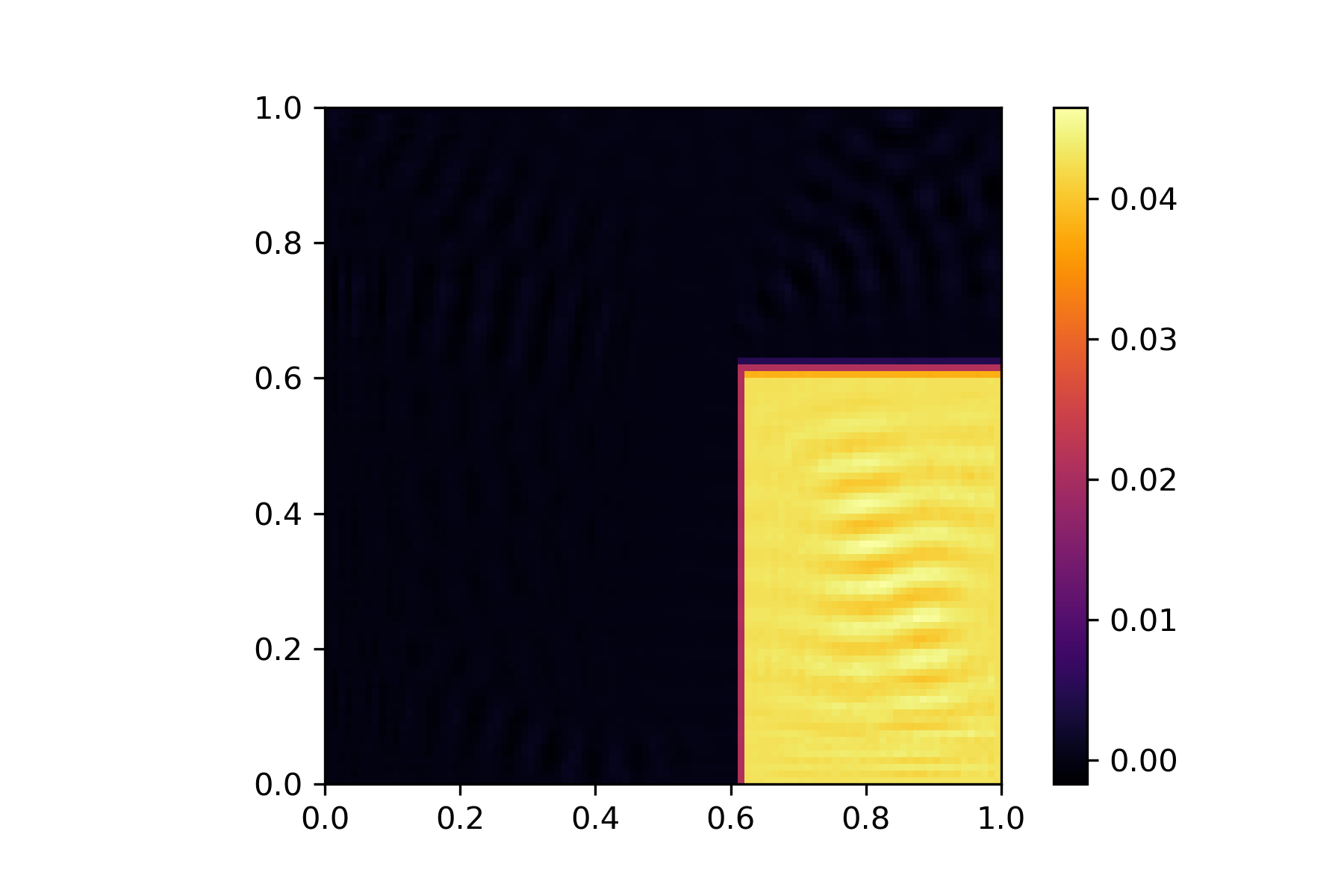}
        \caption{$TV_I$}
    \end{subfigure}\par
        \begin{subfigure}[b]{0.32\linewidth}        
        \centering
        \includegraphics[trim=60 30 60 30, clip, width=\linewidth]{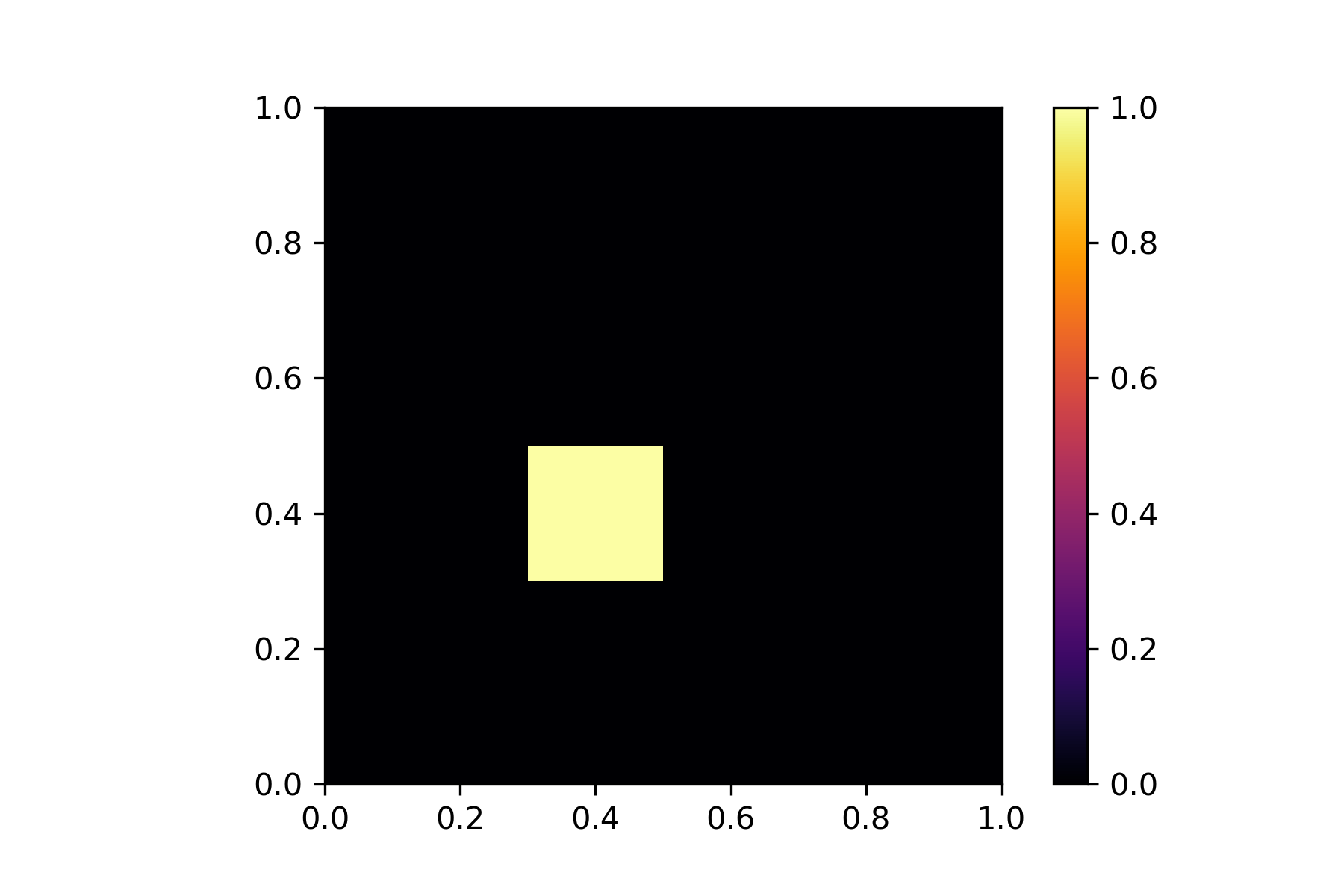}
        \caption{True source}
    \end{subfigure}
    \begin{subfigure}[b]{0.32\linewidth}        
        \centering
        \includegraphics[trim=60 30 60 30, clip, width=\linewidth]{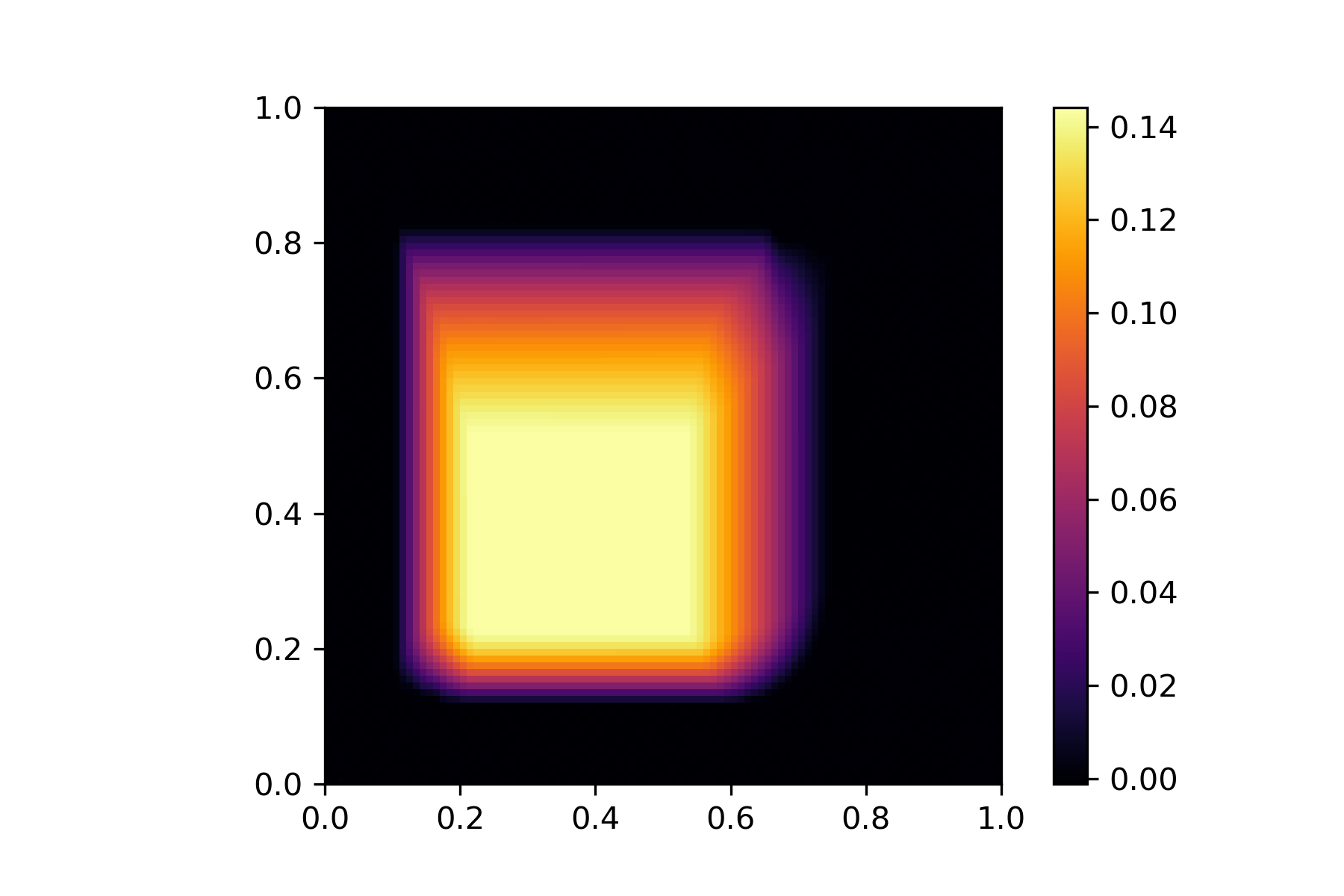}
        \caption{$TV_w$}
    \end{subfigure}
    \begin{subfigure}[b]{0.32\linewidth}        
        \centering
        \includegraphics[trim=60 30 60 30, clip, width=\linewidth]{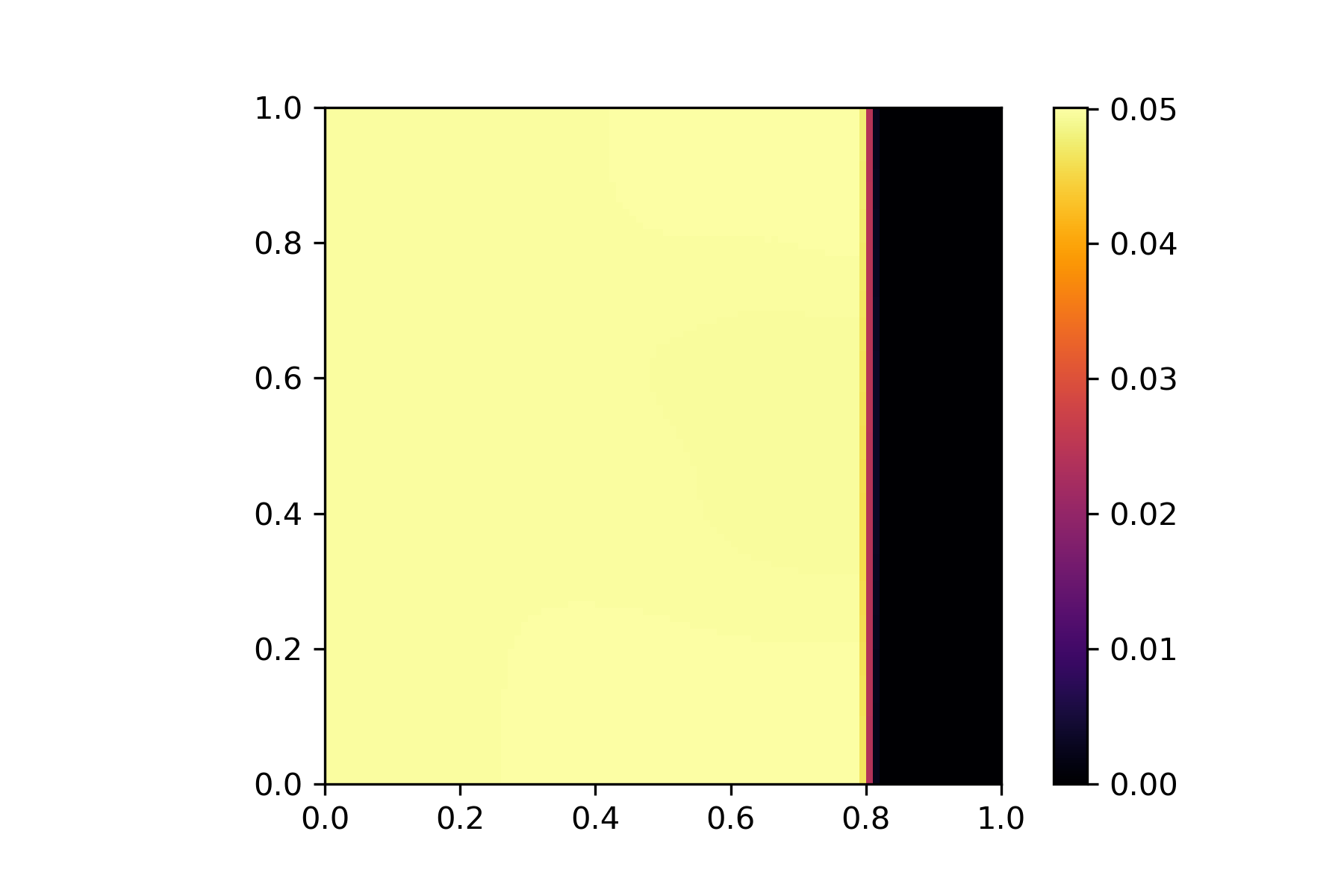}
        \caption{$TV_I$}
    \end{subfigure}\par
    \caption{Comparison of true generation source and inverse recovery involving the anisotropic conductivity $D_1$, cf. \eqref{eq:conductivity1}.}
    \label{fig:anisotropic1}
\end{figure}

The weighted TV algorithm for the case with discontinuous conductivity defined in \eqref{eq:conductivity2} also successfully locates the sources; see Figure \ref{fig:anisotropic2}. Note in particular that the source displayed in panel g) has support at the discontinuity, but this does not seem to affect the ability of the weighted scheme to detect the source.

\begin{figure}[H]
    \centering
    \begin{subfigure}[b]{0.32\linewidth}        
        \centering
        \includegraphics[trim=60 30 60 30, clip, width=\linewidth]{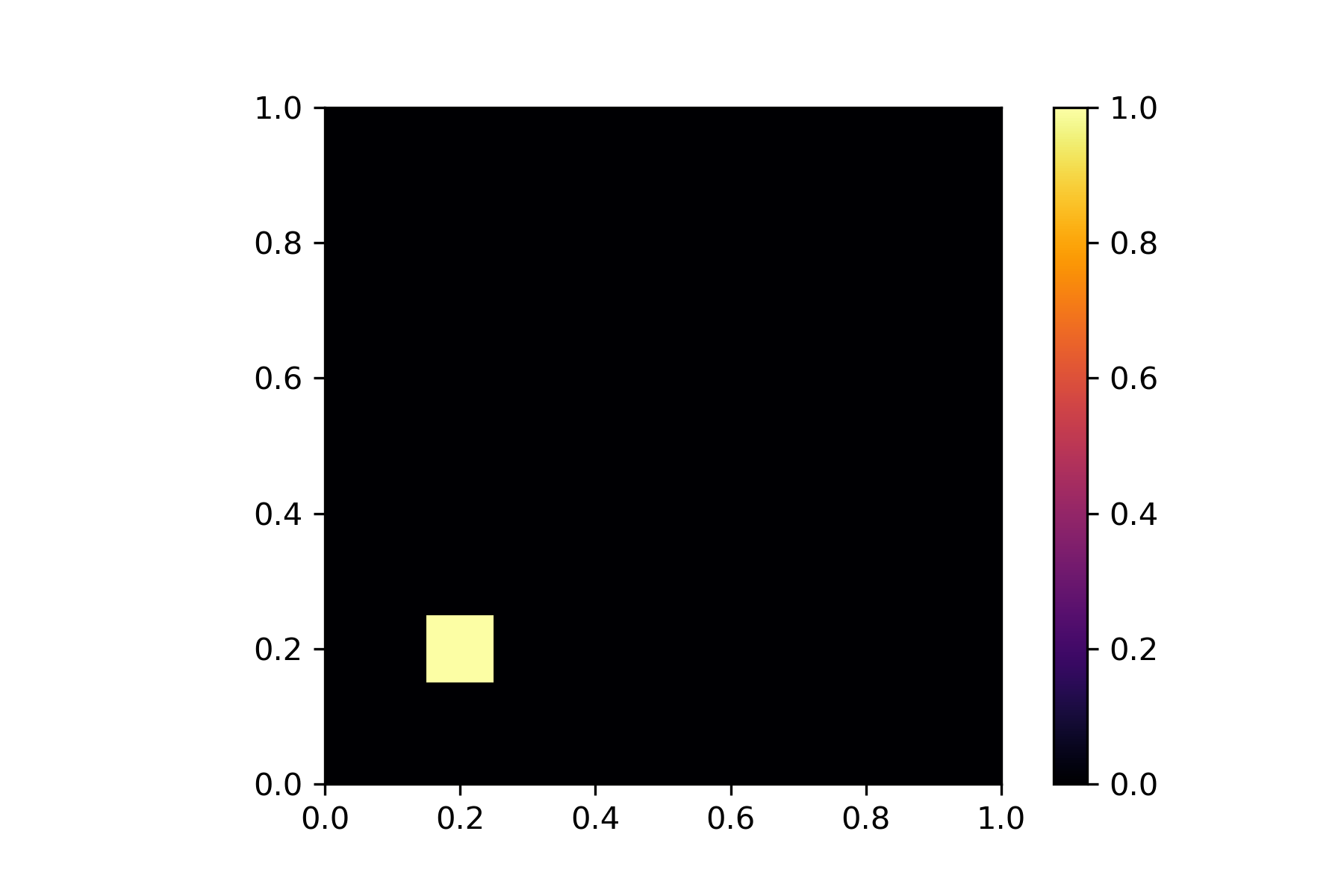}
        \caption{True source}
    \end{subfigure}
    \begin{subfigure}[b]{0.32\linewidth}        
        \centering
        \includegraphics[trim=60 30 60 30, clip, width=\linewidth]{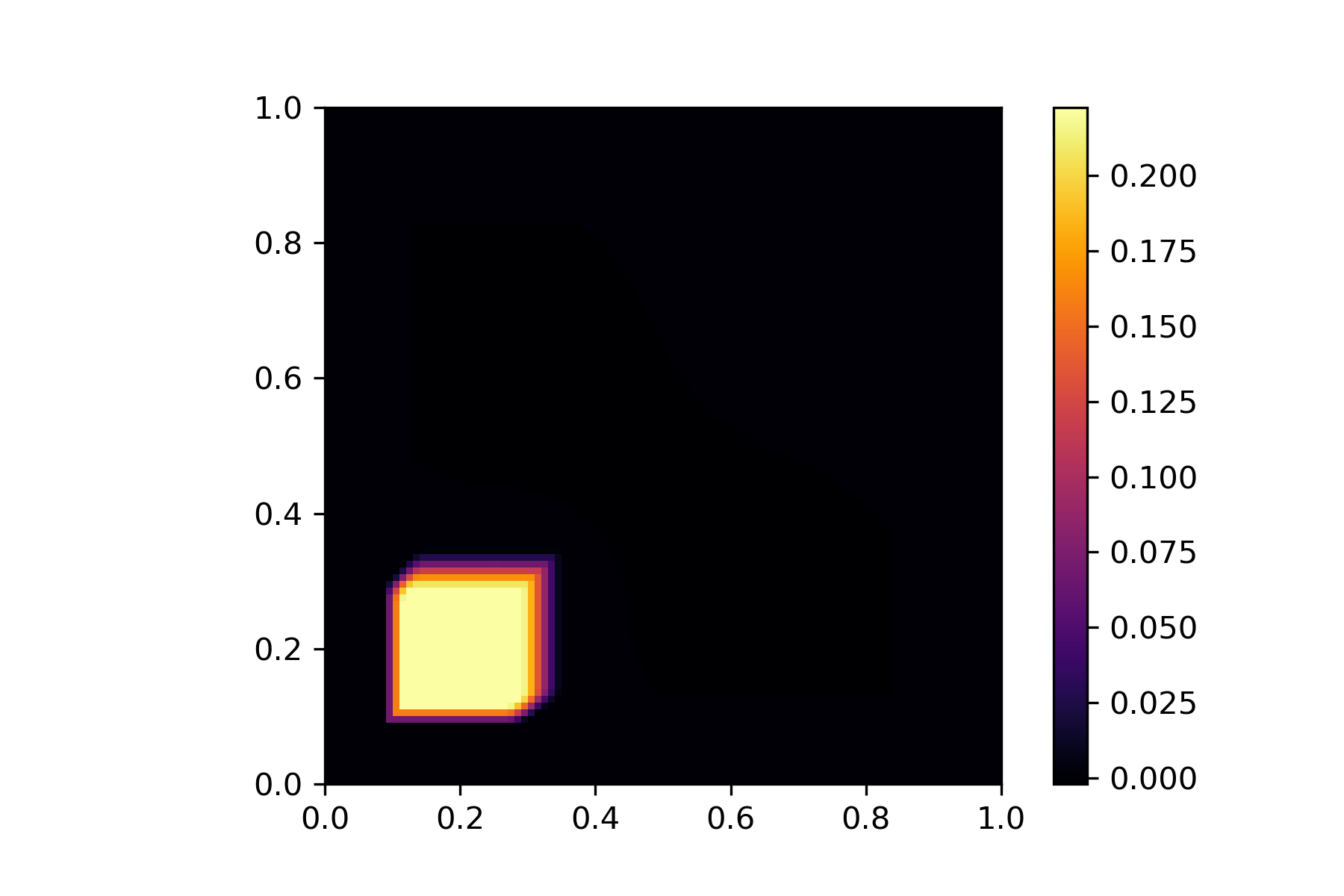}
        \caption{$TV_w$}
    \end{subfigure}
    \begin{subfigure}[b]{0.32\linewidth}        
        \centering
        \includegraphics[trim=60 30 60 30, clip, width=\linewidth]{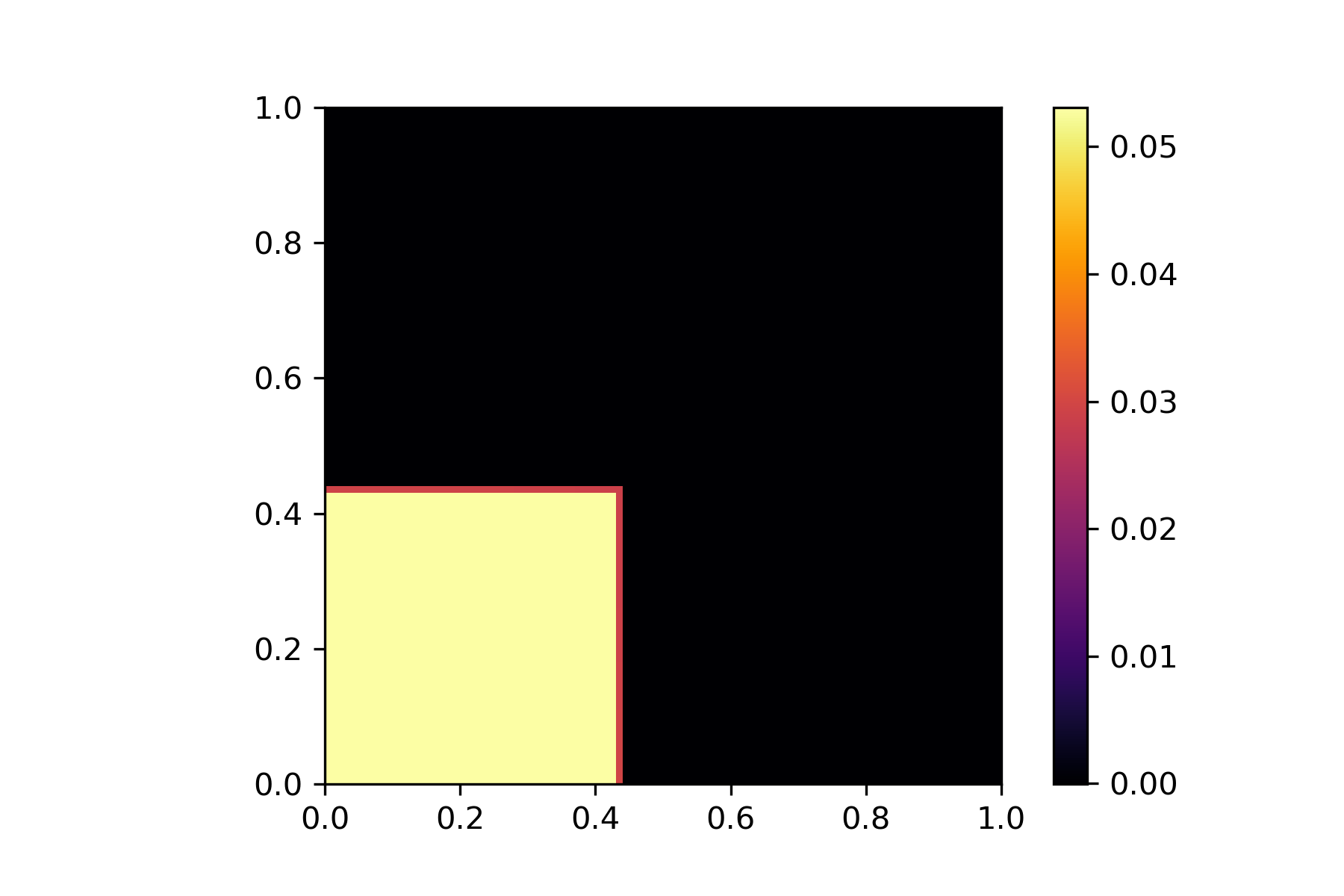}
        \caption{$TV_I$}
    \end{subfigure}\par
    \begin{subfigure}[b]{0.32\linewidth}        
        \centering
        \includegraphics[trim=60 30 60 30, clip, width=\linewidth]{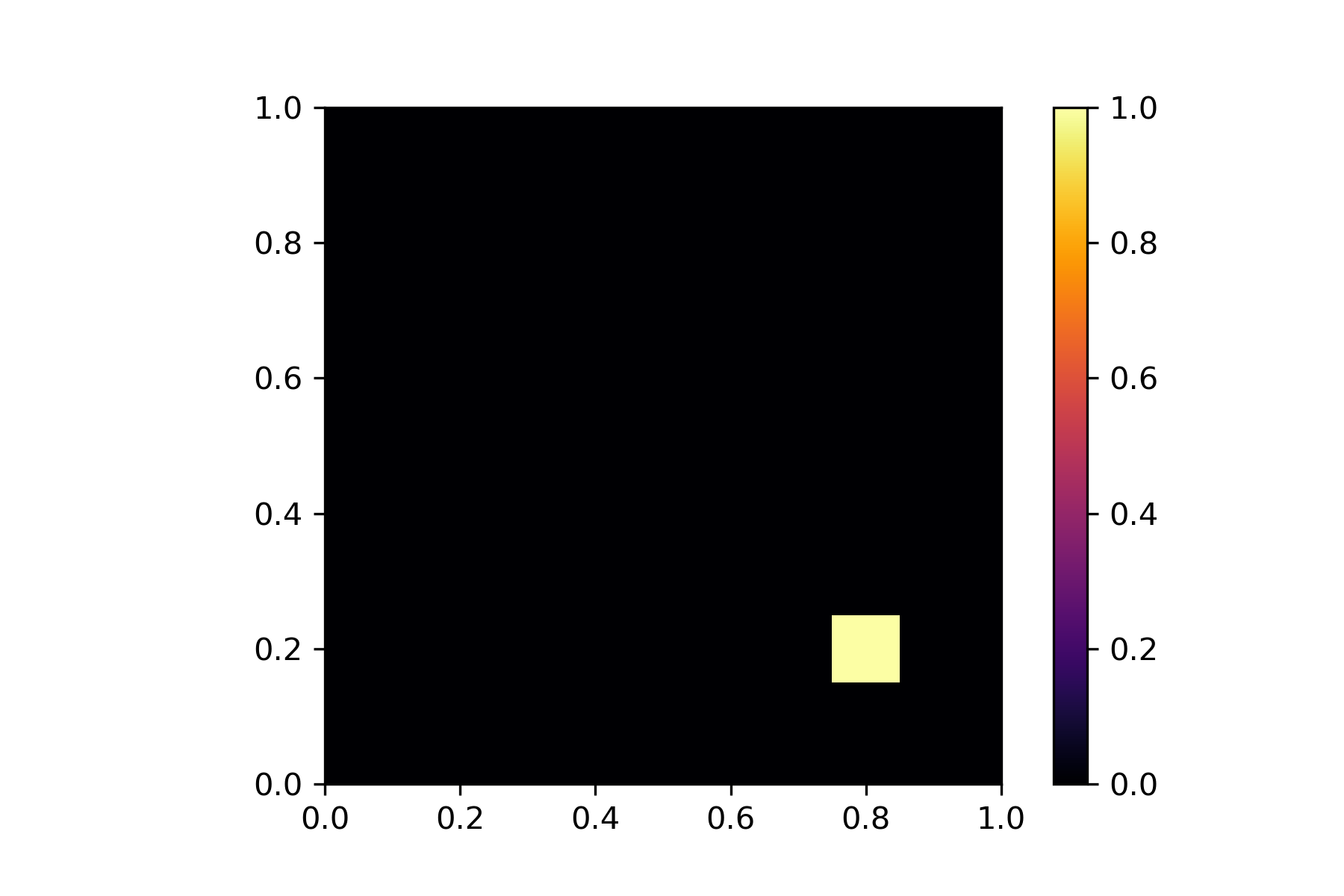}
        \caption{True source}
    \end{subfigure}
    \begin{subfigure}[b]{0.32\linewidth}        
        \centering
        \includegraphics[trim=60 30 60 30, clip, width=\linewidth]{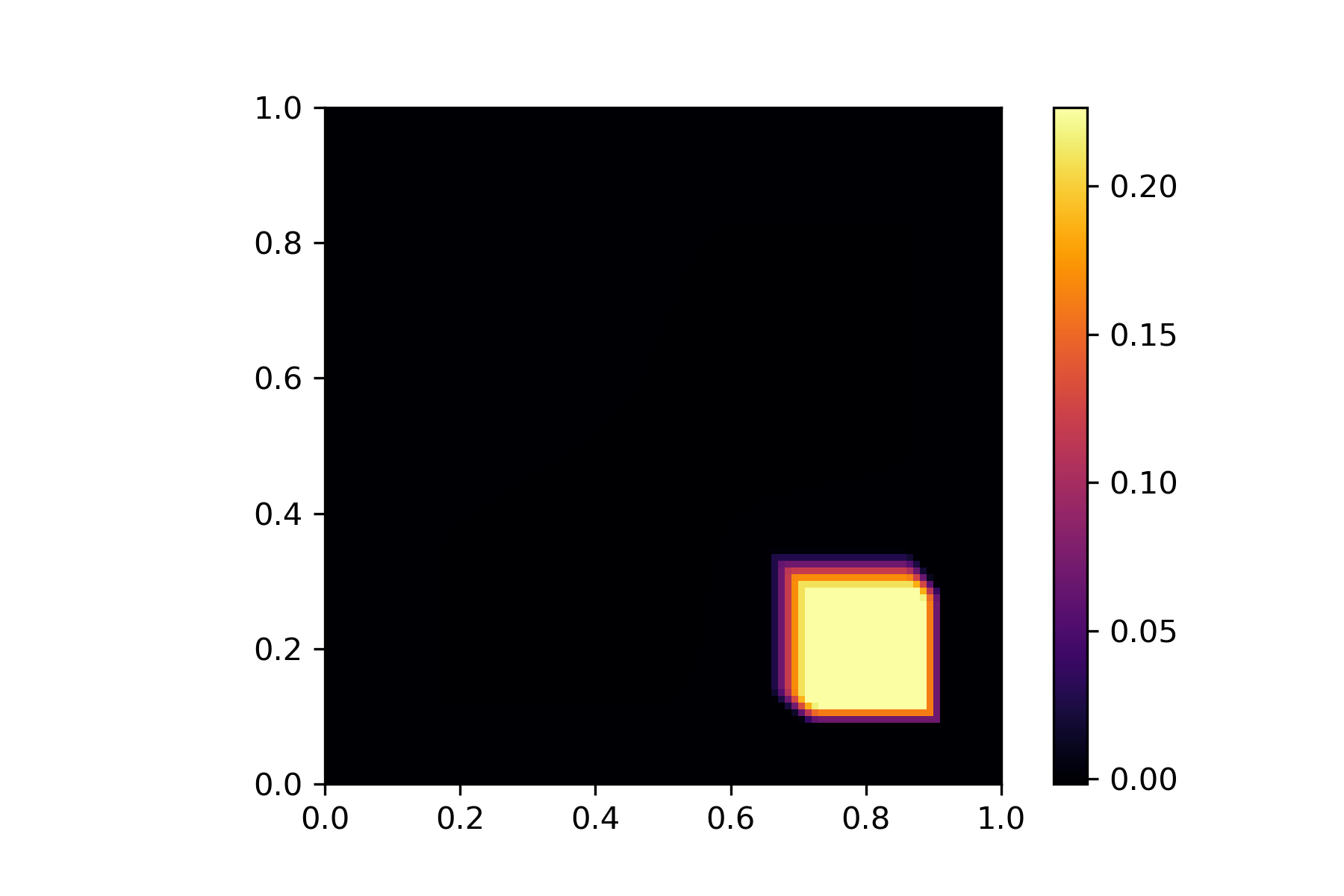}
        \caption{$TV_w$}
    \end{subfigure}
    \begin{subfigure}[b]{0.32\linewidth}        
        \centering
        \includegraphics[trim=60 30 60 30, clip, width=\linewidth]{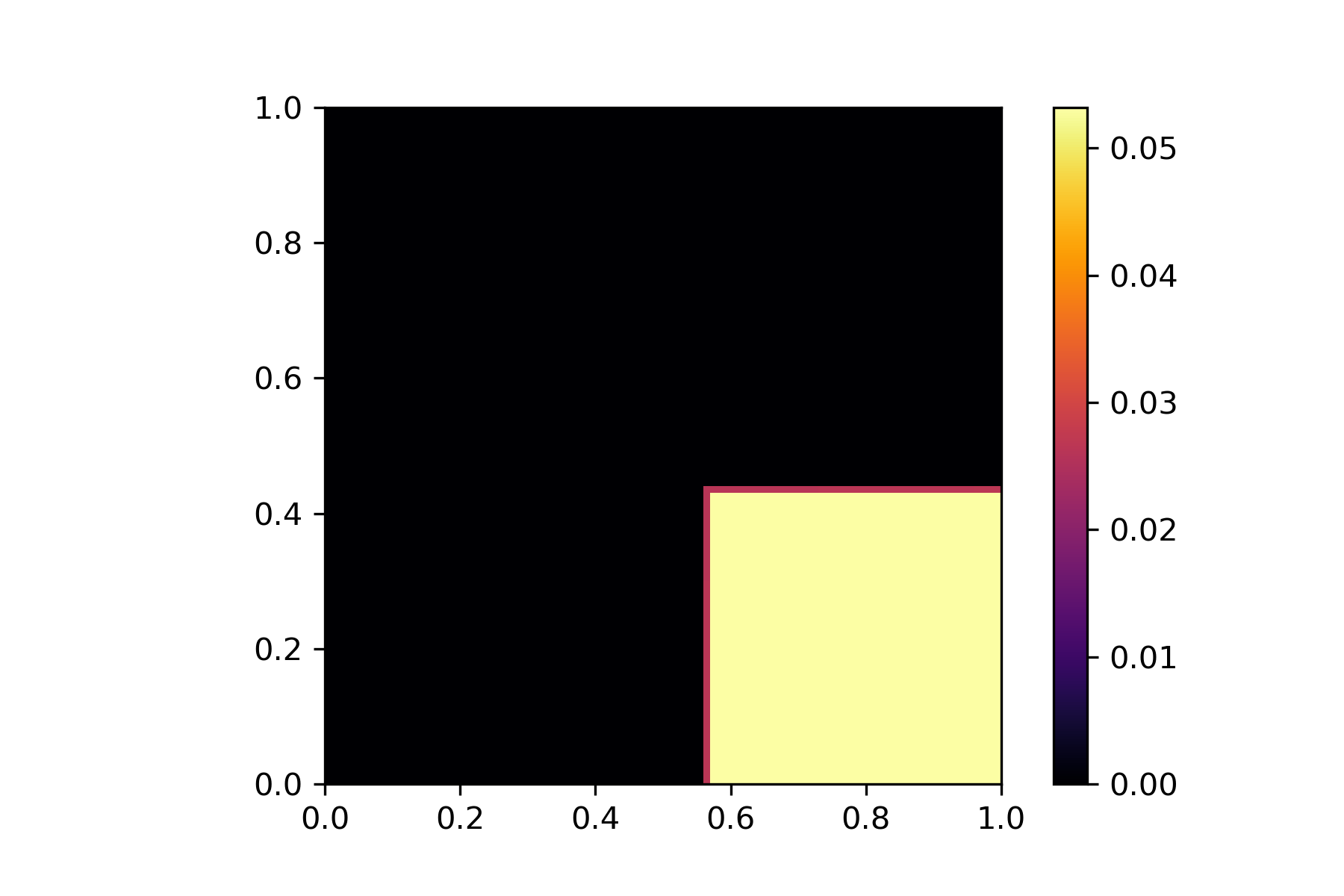}
        \caption{$TV_I$}
    \end{subfigure}\par
        \begin{subfigure}[b]{0.32\linewidth}        
        \centering
        \includegraphics[trim=60 30 60 30, clip, width=\linewidth]{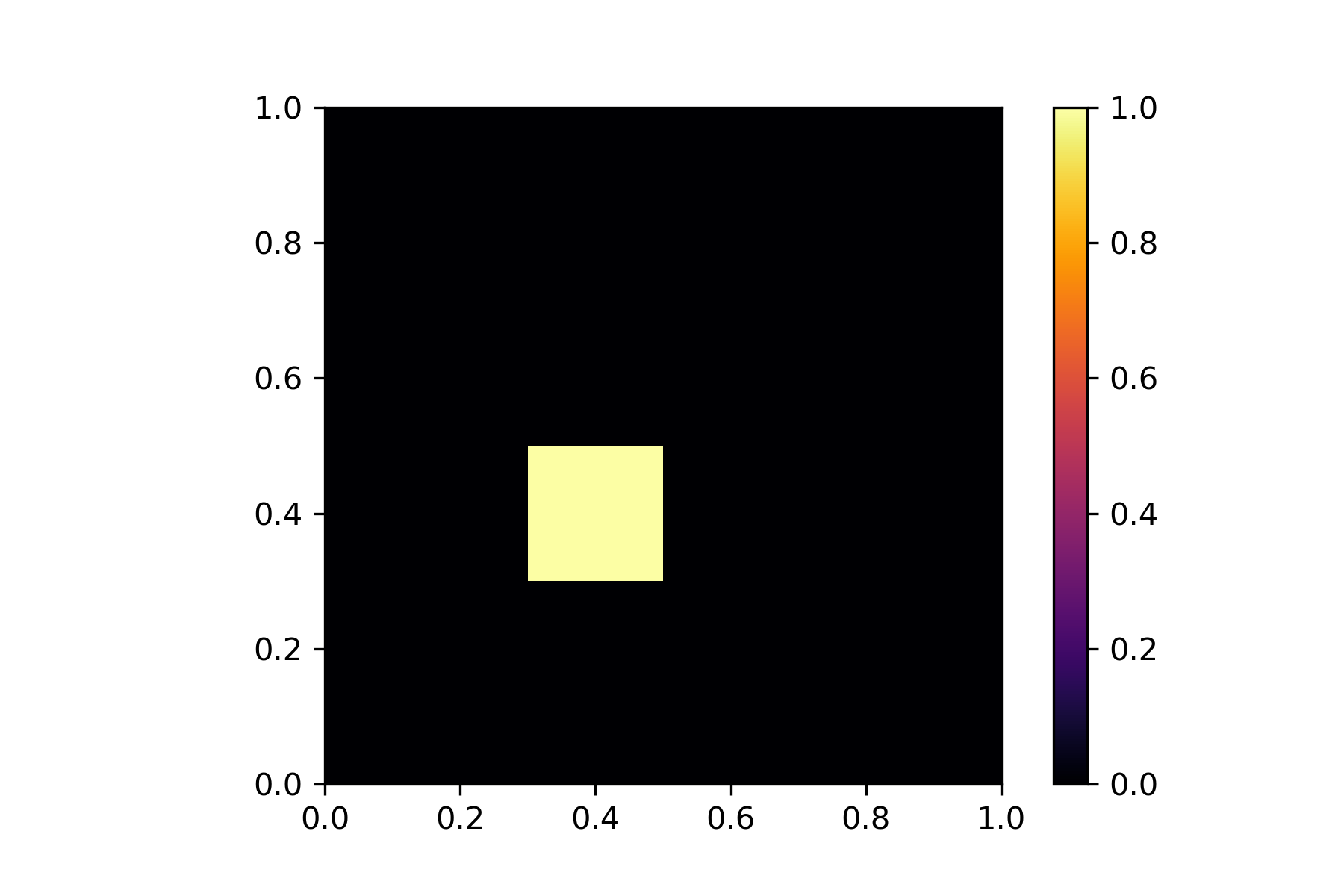}
        \caption{True source}
    \end{subfigure}
    \begin{subfigure}[b]{0.32\linewidth}        
        \centering
        \includegraphics[trim=60 30 60 30, clip, width=\linewidth]{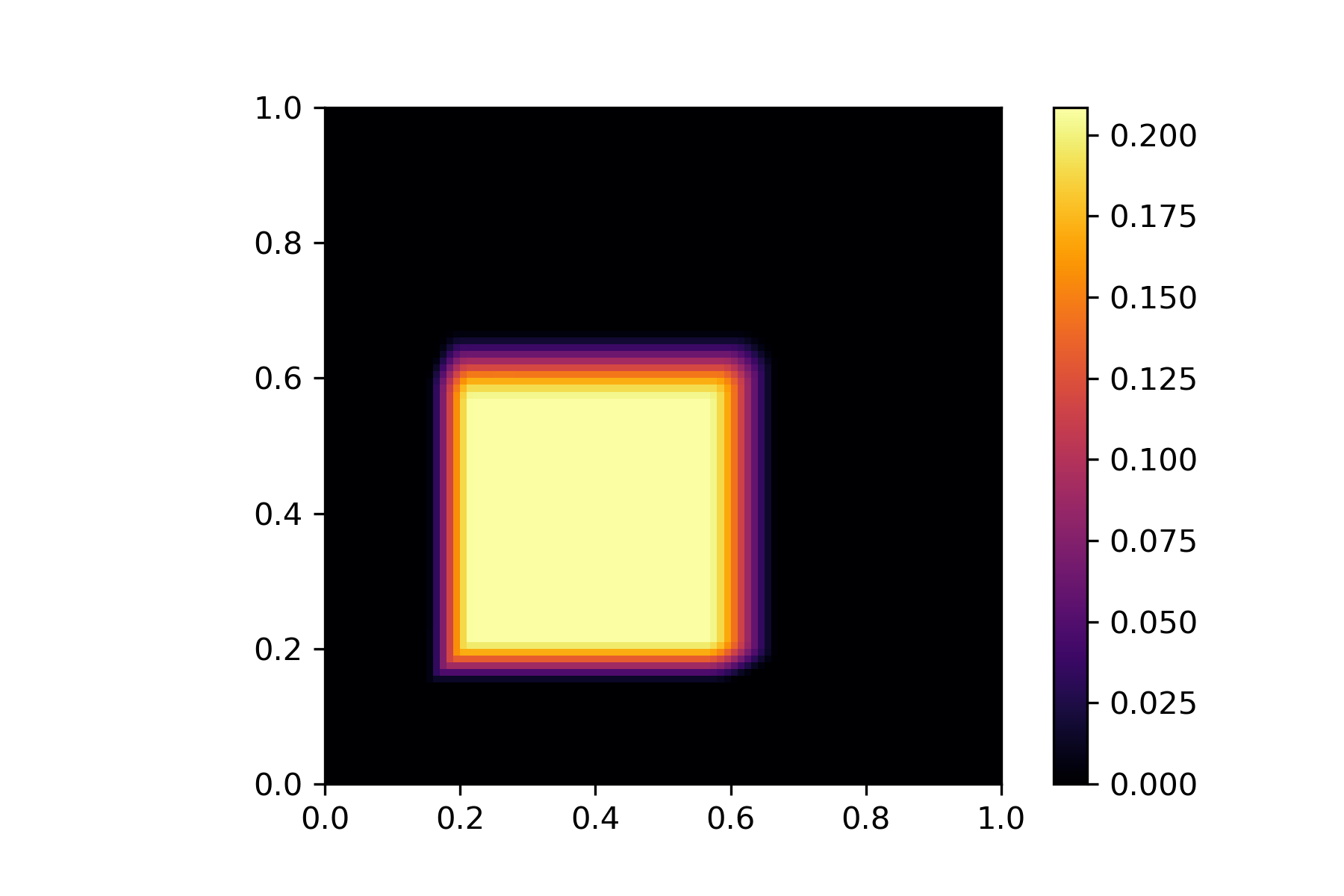}
        \caption{$TV_w$}
    \end{subfigure}
    \begin{subfigure}[b]{0.32\linewidth}        
        \centering
        \includegraphics[trim=60 30 60 30, clip, width=\linewidth]{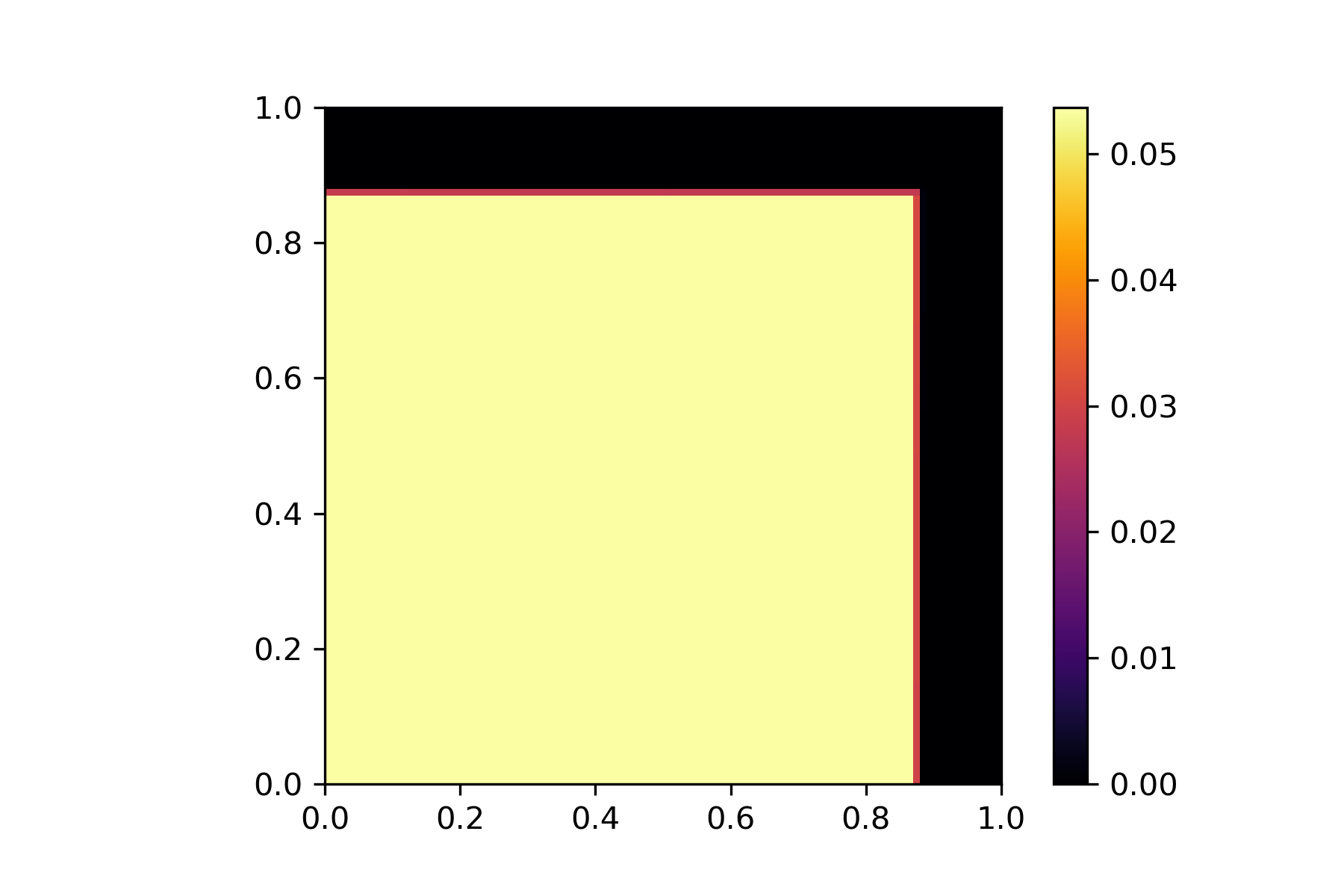}
        \caption{$TV_I$}
    \end{subfigure}\par
    \caption{Comparison of true generation source and inverse recovery involving the anisotropic conductivity $D_2$, cf. \eqref{eq:conductivity2}.}
    \label{fig:anisotropic2}
\end{figure}

\section{Conclusion}

In this paper we have studied the influence of weighting on TV-regularized solutions for inverse problems, which is introduced in order to reduce bias stemming from specific properties of the forward operator. We have seen that appropriate weighting can strongly improve the quality of reconstructions and the position of sources compared to standard TV methods. Compared to previously studied methods, such as weighted $\ell^1$, which prefers zero-dimensional structures, weighted TV enhances codimension-one structures and may introduce some bias on their shape. 
As a potential alternative, a hybrid method between TV and weighted $\ell^1$ can be considered, which seems feasible for inverse source problems with small structures. Naturally, the approach encounters problems for larger shapes.

We finally mention that in order to cure depth bias -  as appearing in inverse source problems in EEG/MEG or ECG due to the depth variation of the point-spread (respectively Green's) function - it is crucial to use an extended total variation, which also takes into account the boundary of the domain. The extended version effectively considers the total variation of a function extended by zero outside the domain, thus penalizing structures at the boundary as well.

In general, the choice of the optimal weights in TV and related methods is still a challenging question which might be solved by machine learning techniques for specific applications.

\section*{Acknowledgements}

MB acknowledges support from DESY (Hamburg, Germany), a member of the Helmholtz Association HGF.

\printbibliography

\appendix
\section{Regularized problem, 1D} \label{sec:regularized_problem_1D} 
In this appendix we prove Theorem \ref{thm:variational_rigorous_proof}. 
We need two lemmas. 

\begin{lemma}\label{lem:KtoC}
    Let $K: L^2(\Omega) \rightarrow L^2(E)$ be \rev{bounded} and let $Q$ be the orthogonal projection defined in \eqref{eq:ortho_proj}. Then, for any $f \in L^2(\Omega)$,
    \begin{equation*}
        K^*Kf = C^*Cf + \frac{(Kf,K1)}{\|K1\|^2}K^*K1, 
    \end{equation*}
    \rev{provided that $K1 \neq 0$. Here,} $C$ is the mapping introduced in \eqref{def:C}. 
\end{lemma}
\begin{proof}
  Note that we can decompose any image $Kf$ by
  \begin{equation*}
      Kf = QKf + \textnormal{proj}_{K1}(Kf) = QKf + \frac{(Kf,K1)}{\|K1\|^2}K1,
  \end{equation*}
  and consequently
  \begin{equation*}
      K^*Kf = K^*QKf + \frac{(Kf,K1)}{\|K1\|^2}K^*K1.
  \end{equation*}
  The result follows by observing that $C^*C = (QK)^*QK = K^*Q^*QK=K^*QK$, keeping in mind that $Q$ is an orthogonal projection.
\end{proof}

Inspired by the characterization of the subdifferential $\partial TV(f)$ presented in, e.g., \cite{bredies23, vaillo04} we will show that 
\begin{lemma}\label{lem:subgrad}
    Assume that $\Omega =(0,1)$ \rev{and that $w$ is continuous and non-negative.  C}onsider the weighted TV-functional $TV_w$ defined in \eqref{eq:TVw}. If $p = - z' \in L^2(\Omega)$ satisfies
    \begin{enumerate}[(i)]
      \item $|z(x)|\leq w(x) \ \forall x\in(0,1)$, 
      \item $\int_0^1 z Df(x) = TV_w(f),$ and 
      \item $Tz = 0$, where $T$ is the \rev{standard} trace operator\footnote{\begin{revblock}
Since $p=-z' \in L^2(0,1)$ by assumption, $z \in H^1(0,1)$, and Sobolev's embedding theorem gives $z \in C([0,1])$. In one space dimension the standard trace operator $T$ is therefore nothing but point evaluation at the two endpoints,
$
    Tz = (z(0),z(1)),
$
which is well defined and bounded from $H^1(0,1)$ to $\mathbb{R}^2$.
\end{revblock}},
    \end{enumerate}
then $ p \in \partial TV_w(f).$
Furthermore, $p$ has zero mean, i.e.,
  $      \int_0^1 p(x) dx = 0.$
\end{lemma}

\begin{proof}
    By definition of the subdifferential,
    $$p \in \partial TV_w(f) \iff TV_w(g) \geq TV_w(f) + (p, g-f)\quad \forall g \in BV(0,1).$$
    Using the assumptions for $z$ above, we get for any $g \in BV(0,1)$
    \begin{eqnarray*}
        TV_w(f) + (p,g-f) &=&  - \int_0^1 z'(x)g(x) \rev{dx} 
        = \int_0^1z(x)Dg(x) \\
        &\leq& \int_0^1 |z(x)|Dg(x)| 
    \leq \int_0^1 w(x)|Dg(x)| \\ 
        &=& TV_w(g),
    \end{eqnarray*}
    where the integration by parts is justified, e.g., by the generalized Green's theorem \cite[Theorem 1.9]{anzellotti83}. We employed \eqref{def:TVw2} in the last equality. 

    For the zero mean property, simply choose $g_k \in BV(0,1), \ g_k(x) = -k \, \textnormal{sgn} \left( \int_0^1 z'(x) \rev{dx}\right)$ above to get 
    \begin{eqnarray*}
        TV_w(f) + (p,g_k-f) &=& TV_w(f) - \int_0^1 z'(x)g_k(x) dx - TV_w(f) \\
        &=& k \, \textnormal{sgn} \left( \int_0^1 z'(x) \rev{dx}\right) \int_0^1 z'(x)dx, 
    \end{eqnarray*}
    which we can make as large as we want, by letting $k \rightarrow \infty$, unless $-z' = p$ has zero mean. Hence, $-z'=p$ must have zero integral because $TV_w(g_k)=0$ for any constant $k$.  

\end{proof}

\subsubsection*{Proof of Theorem \ref{thm:variational_rigorous_proof}.}
\begin{proof} 
We will prove the theorem by verifying that $$f_\alpha = \gamma\rho \bar{H}_{x^*}+\eta$$ satisfies the first-order optimality condition
\begin{equation} \label{eq:FONC}
    K^*K(\rho \bar{H}_{x^*}+\tau) - K^*Kf \in \alpha \partial TV_w(f)
\end{equation}
associated with \eqref{eq:variational_problem} for specific choices of $\gamma$ and $\eta$ (cf. Theorem \ref{thm:variational_rigorous_proof}).

Since $TV_w$ is one-homogeneous and does not see constants, it follows that, for $\gamma > 0$, 
\begin{equation*}
    \partial TV_w(f_\alpha) = \partial TV_w(\gamma\rho\bar{H}_{x^*}) = \textnormal{sgn}(\gamma\rho)\partial TV_w(\bar{H}_{x^*}) = \textnormal{sgn}(\rho)\partial TV_w(\bar{H}_{x^*}).
\end{equation*}
Thus, by inserting the expression for $f_\alpha$ into the optimality condition \eqref{eq:FONC}, and multiplying both sides by $\textnormal{sgn}(\rho)$, the problem can be reformulated to proving that there exists $$p \in \partial TV_w(\bar{H}_{x^*})$$ such that

\begin{equation}\label{eq:findGrad}
    K^*K\left((1-\gamma)|\rho|\bar{H}_{x^*}+\textnormal{sgn}(\rho)(\tau-\eta)\right) = \alpha p.
\end{equation}
Our strategy for verifying the existence of such an element $p$ in the subdifferential $\partial TV_w(\bar{H}_{x^*})$ will be to initially define $z$ (up to a constant) by 
\begin{equation*}
      -\alpha z' =K^*K\left((1-\gamma)|\rho|\bar{H}_{x^*}+\textnormal{sgn}(\rho)(\tau-\eta)\right), 
\end{equation*}
and thereafter show, for a specific choice of the free constant, that all the assumptions in Lemma \eqref{lem:subgrad} are satisfied by $p = -z'$, which will yield the desired result.

Since the definition of $z$ depends on the parameters $\gamma$ and $\eta$,  we need to determine these quantities. Let us make use of Lemma \ref{lem:KtoC} to obtain the expression
\begin{eqnarray}
  -\alpha z' &=&K^*K\left((1-\gamma)|\rho|\bar{H}_{x^*}+\textnormal{sgn}(\rho)(\tau-\eta)\right) \nonumber \\ &=& (1-\gamma)|\rho| C^*C\bar{H}_{x^*} \nonumber \\ &+& \frac{((1-\gamma)|\rho| K\bar{H}_{x^*} + \textnormal{sgn}(\rho)(\tau-\eta)K1, K1)}{\|K1\|^2}K^*K1 \nonumber \\ 
  &=& (1-\gamma)|\rho| C^*C\bar{H}_{x^*} \nonumber \\  &+& \left(\frac{(1-\gamma)|\rho|(K\bar{H}_{x^*},K1)}{\|K1\|^2} + \textnormal{sgn}(\rho)(\tau-\eta)\right)K^*K1. \label{eq:fullexpr}
\end{eqnarray}
Recall that Lemma \ref{lem:subgrad} asserts that $z'$ must have zero mean. Since $C$ annihilates constants, the observation that
\begin{equation*}
    \int_0^1 (C^*C\bar{H}_{x^*})(x)dx = (C^*C\bar{H}_{x^*},1) = (C\bar{H}_{x^*},C1) = 0,
\end{equation*}
therefore motivates the removal of the last term in \eqref{eq:fullexpr}. Consequently, we must choose 
\begin{equation*}
  \eta = \tau + \frac{(1-\gamma)\rho(K\bar{H}_{x^*},K1)}{\|K1\|^2}.
\end{equation*}
Then \eqref{eq:fullexpr} simplifies to
\begin{equation} \label{eq:eta_found}
    -\alpha z' =(1-\gamma)|\rho| C^*C\bar{H}_{x^*},
\end{equation}
which guarantees that $z'$ has zero mean. 

The choice of $\gamma$ is somewhat more involved, but observe that by multiplying \eqref{eq:eta_found} with $\bar{H}_{x^*}$ and integrating, we get 
\begin{eqnarray}
    \nonumber
    \int_0^1 -z'(x)\bar{H}_{x^*}(x)dx 
    &=& \frac{(1-\gamma)|\rho|}{\alpha}\int_0^1C^*C\bar{H}_{x^*}(x)\bar{H}_{x^*}(x)dx \\
    \nonumber
    &=& \frac{(1-\gamma)|\rho|}{\alpha}(C^*C\bar{H}_{x^*},\bar{H}_{x^*}) \\
    \nonumber
    &=& \frac{(1-\gamma)|\rho|}{\alpha}\|C\bar{H}_{x^*}\|^2 \\ 
    \nonumber
    &=& (1-\gamma)\frac{|\rho|\|C\bar{H}_{x^*}\|}{\alpha}\|C\bar{H}_{x^*}\| \\
    \label{eq:only_integration_by-parts_missing}
    &=& \|C\bar{H}_{x^*}\| := TV_w(\bar{H}_{x^*}),
\end{eqnarray}
when
\begin{equation} \label{eq:expression_for_gamma}
    \gamma = 1 - \frac{\alpha}{|\rho|\|C\bar{H}_{x^*}\|}.
\end{equation}
Furthermore, provided that $z$ has zero trace, i.e., $Tz=0$,  
\begin{equation}
   \int_0^1 -z'(x)\bar{H}_{x^*}(x)dx = \int_0^1 z(x)D\bar{H}_{x^*}(x). \label{eq:measNorm}
\end{equation}
Combining this with \eqref{eq:only_integration_by-parts_missing} would imply that assumption (ii) of Lemma \ref{lem:subgrad} is satisfied. However, we must make sure that $Tz=0$. We now address this issue.  

Since $z'$ has zero mean, we find that 
\begin{equation*}
  \int_0^1 z'(x)dx = z(1) - z(0) = 0,
\end{equation*}
which implies that $z(0) = z(1)$. This point-wise evaluation is justified because Sobolev's embedding theorem asserts that $z \in C([0,1])$ provided that $-z'=p \in L^2(0,1)$. Recall that $z$ is (until now) only defined up to a constant, and we are therefore free to choose this constant such that $z(0) = z(1) = 0$. Thus, the zero trace condition (iii) in Lemma \ref{lem:subgrad} is satisfied and as a consequence of the discussion above, so is condition (ii).

Finally, it remains to show that $|z(x)| \leq w(x) \ \forall x \in (0,1)$, i.e., to verify that requirement (i) in Lemma \ref{lem:subgrad} holds. Since $D\bar{H}_{x^*}$ is the Dirac measure at $x^*$, we can write, keeping in mind that $z$ has zero trace, 
\begin{eqnarray*}
    |z(x)| &=& \left| \int_0^1 z(y)D\bar{H}_{x}(y) \right| \\ 
    &=& \left| - \int_0^1 z'(y) \bar{H}_{x}(y)dy \right|.  
\end{eqnarray*}
Invoking \eqref{eq:eta_found} and the expression \eqref{eq:expression_for_gamma} for $\gamma$ now yield that
\begin{eqnarray*}
    |z(x)| &=& \left| \frac{1}{\|C\bar{H}_{x^*}\|}\int_0^1 (C^*C\bar{H}_{x^*})(y)\bar{H}_{x}(y)dy \right| \\ 
    &=& \frac{1}{\|C\bar{H}_{x^*}\|}\left|\left(C^*C\bar{H}_{x^*},\bar{H}_{x}\right)\right| \\
    &=& \frac{1}{\|C\bar{H}_{x^*}\|}\left|\left(C\bar{H}_{x^*},C\bar{H}_{x}\right)\right| \\
    &\leq& \|C\bar{H}_x\| := w(x),
\end{eqnarray*}
where we have employed the Cauchy-Schwarz inequality. 

Hence, all the requirements needed in Lemma \ref{lem:subgrad} are satisfied, and we can conclude that $p = -z' \in \partial TV_w(\bar{H}_{x^*})$. Combining this fact with the optimality condition \eqref{eq:findGrad} completes the proof.
\end{proof}

\begin{revblock}
\section{Proofs of Lemma \ref{lem:contW2d} and Lemma \ref{lem:upper_bound_BV}}\label{app:lemmata}

\subsubsection*{Proof of Lemma \ref{lem:contW2d}}

    First we note that, since $K: L^1(\Omega) \rightarrow L^2(E)$ is assumed continuous and $E$ is a bounded domain, $K: L^1(\Omega) \rightarrow L^1(E)$ is bounded. 
    It suffices to show that $\overline\Omega \ni y \mapsto \partial_{y_i} G(\cdot;y) \in L^1(\Omega)$ is continuous, cf. \eqref{eq:weights_definition}. 

    Let $y_{(n)} \rightarrow y_0 \in \overline\Omega$. If we can show 
    \begin{enumerate}[(i)]
        \item $\partial_{y_i} G(x;y_{(n)}) \rightarrow \partial_{y_i} G(x,y_0)$ pointwise a.e., and
        \item          $\partial_{y_i} G(\cdot;y)$ is uniformly integrable, i.e., for any measurable subset $A \subset \Omega$ we have
    \begin{equation}\label{def:uniformint}
        \lim_{|A| \rightarrow 0} \sup_{y \in \overline\Omega} \int_A |\partial_{y_i} G(x;y)|dx = 0,
    \end{equation}
    \end{enumerate}
    it will follow from Vitali convergence theorem that
    \begin{equation*}
        \| \partial_{y_i} G(\cdot;y_{(n)}) - \partial_{y_i}G(\cdot;y_0)\|_{L^1(\Omega)} \rightarrow 0,
    \end{equation*}
    and we would be done.

    \textbf{\emph{Step 1}: Pointwise convergence.} 
    
    First, consider $y_0 \in \Omega$. Then $\partial_{y_i} G(x; y_{(n)}) \rightarrow \partial_{y_i}G(x;y_0)$ for every $x \neq y_0$ since $G$ is smooth outside any ball around $y_0$. 

    The case $y_0 \in \partial\Omega$ follows from the considerations in Remark \ref{rem:poissonker}.

    In total, for any $y_0 \in \overline\Omega$, we therefore have pointwise convergence a.e.
    \\ \\
    \textbf{\emph{Step 2}: Uniform integrability.}

    We again first consider $y_0 \in \Omega$. A standard estimate for Green's functions (see e.g. \cite{DongKim2021}) gives
    \begin{equation*}
        |\nabla_y G(x;y)| \leq \frac{C}{|x-y|}.
    \end{equation*}
    For \(y\in\partial\Omega\), Remark \ref{rem:poissonker} gives
    \begin{equation*}
    \partial_{y_i}G(x;y)=n_i(y)\partial_{\mathbf n}G(x;y). 
    \end{equation*}
    The Poisson kernel estimate \cite{ChenWang2023} yields 
    \begin{equation*}
        |\partial_{\mathbf n}G(x;y)| \leq C\frac{\operatorname{dist}(x,\partial\Omega)}{|x-y|^2} \leq \frac{C}{|x-y|}.
    \end{equation*} 
    Thus, for all $y\in\overline{\Omega}$, 
    \begin{equation*}
        |\partial_{y_i}G(x;y)| \leq \frac{C}{|x-y|}.
    \end{equation*}
    Consequently, for any $A \subset \Omega$ measurable and any $y \in \overline\Omega$, we have
    \begin{equation*}
        \int_A |\partial_{y_i} G(x;y)| dx \leq C \int_A \frac{dx}{|x-y|}.
    \end{equation*}
     For any radius $r > 0$, we can split the integral $\int_A = \int_{A \cap B_r(y)} + \int_{A \setminus B_r(y)}$. From basic estimates and the specific choice $r = \sqrt{\frac{|A|}{2\pi}}$ we obtain 
    \begin{equation}\label{eq:intest1}
        \int_A \frac{dx}{|x-y|} \leq C'|A|^{1/2}.
    \end{equation}

    By realizing that this is independent of $y \in \overline\Omega$ and sending $|A|$ to zero, we obtain \eqref{def:uniformint}, which completes the proof.

\subsubsection*{Proof of Lemma \ref{lem:upper_bound_BV}}

By Lemma \ref{lem:contW2d}, we have \(w_1,w_2\in C(\overline{\Omega})\) and \(w_\partial\in C(\partial\Omega)\). 
\\ \\
\emph{Step 1: \(f\in C^\infty(\Omega)\cap BV(\Omega)\).}

The regularity of \(f\) gives \(Df=\nabla f\,dx\) and \(f\in W^{1,1}(\Omega)\). Since \(\Omega\) is Lipschitz, the boundary trace \[ \operatorname{Tr}f \in L^1(\partial\Omega) \] is well defined. Moreover, \(w_\partial\in C(\partial\Omega)\subset L^\infty(\partial\Omega)\), and therefore \[ \int_{\partial\Omega} w_\partial(y)|\operatorname{Tr}f(y)|\,dS(y)<\infty . \] In what follows we omit the notation \(\operatorname{Tr}\) and write simply \(f\) on \(\partial\Omega\), whenever no confusion can arise.

Consequently, the boundary part of the Green representation formula proceeding \eqref{eq:alfred2} is here understood in the trace sense. The computation leading to \eqref{eq:alfred2} is justified as follows. Since
\begin{equation*}
y\mapsto \partial_{y_i}G(\cdot;y) \in C(\overline{\Omega};L^1(\Omega)),
\end{equation*}
we have 
\begin{equation*}
    \sup_{y\in\overline{\Omega}} \|\partial_{y_i}G(\cdot;y)\|_{L^1(\Omega)}<\infty.
\end{equation*}
Therefore the map 
\begin{equation*}
    y\mapsto \partial_{y_i}f(y)\partial_{y_i}G(\cdot;y) 
\end{equation*}
is Bochner integrable as an \(L^1(\Omega)\)-valued map. Since \(K:L^1(\Omega)\to L^1(E)\) is bounded and linear, \(K\) commutes with this Bochner integral: 
\begin{equation*}
    K\left[ \int_\Omega \partial_{y_i}f(y)\partial_{y_i}G(\cdot;y)\,dy \right] = \int_\Omega \partial_{y_i}f(y)K\partial_{y_i}G(\cdot;y)\,dy 
\end{equation*}
in \(L^1(E)\). The same argument applies to the boundary term, since \(f|_{\partial\Omega}\in L^1(\partial\Omega)\) and 
\begin{equation*}
y\mapsto \partial_{\mathbf n}G(\cdot;y) \in C(\partial\Omega;L^1(\Omega)). 
\end{equation*}
Taking the \(L^1(E)\)-norm and using the triangle inequality for Bochner integrals gives 
\begin{equation*}
\|Kf\|_{L^1(E)} \leq \sum_{i=1}^2 \int_\Omega w_i(y)|\partial_{y_i}f(y)|\,dy + \int_{\partial\Omega}w_\partial(y)|f(y)|\,dS(y). 
\end{equation*}
\emph{Step 2: \(f\in BV(\Omega)\).} Since \(\Omega\) is a bounded Lipschitz domain, there exist \(f_k\in C^\infty({\Omega}) \cap BV(\Omega)\) converging strictly to \(f\), namely 
\begin{equation*}
f_k\to f \quad\text{in }L^1(\Omega), \qquad |Df_k|(\Omega)\to |Df|(\Omega).
\end{equation*}
By Step 1, the desired estimate holds for each \(f_k\). First, since \(K:L^1(\Omega)\to L^1(E)\) is bounded, 
\begin{equation*}
    Kf_k\to Kf \quad\text{in }L^1(E), 
\end{equation*}
and therefore 
\begin{equation*}
\|Kf_k\|_{L^1(E)}\to \|Kf\|_{L^1(E)}.
\end{equation*}
Next define 
\begin{equation*}
    \varphi(y,q):=w_1(y)|q_1|+w_2(y)|q_2|, \qquad (y,q)\in\overline{\Omega}\times\mathbb{R}^2.
\end{equation*}
Then \(\varphi\) is continuous and positively \(1\)-homogeneous in \(q\). By Reshetnyak's continuity theorem, 
\begin{equation*}
\int_\Omega \varphi(y,\nabla f_k(y))\,dy \to \int_\Omega \varphi\!\left(y,\frac{dDf}{d|Df|}(y)\right)\,d|Df|(y). 
\end{equation*}
Equivalently, 
\begin{equation*}
    \sum_{i=1}^2\int_\Omega w_i(y)|\partial_{y_i}f_k(y)|\,dy \to \sum_{i=1}^2\int_\Omega w_i(y)\,d|D_i f|(y).
\end{equation*}
Finally, the trace operator is continuous from \(BV(\Omega)\), endowed with the strict topology, into \(L^1(\partial\Omega)\). Hence 
\begin{equation*}
f_k|_{\partial\Omega}\to f|_{\partial\Omega} \quad\text{in }L^1(\partial\Omega). 
\end{equation*}
Since \(w_\partial\in C(\partial\Omega)\subset L^\infty(\partial\Omega)\), we get 
\begin{equation*}
\int_{\partial\Omega}w_\partial|f_k|\,dS \to \int_{\partial\Omega}w_\partial|f|\,dS.
\end{equation*}
Passing to the limit in the estimate for \(f_k\) gives \eqref{eq:alfred2_short} for \(f\in BV(\Omega)\).

\end{revblock}

\end{document}